\documentclass[11pt]{article}

\usepackage{amssymb,amsmath,amsthm,bm,mathrsfs,mathtools}
\usepackage{url,color,bbold,units,pifont}
\usepackage[T1]{fontenc}
\usepackage[margin=1in]{geometry}

\usepackage[labelsep=period,labelfont={bf,small},textfont={small},margin=0.8in]{caption}
\RequirePackage[colorlinks,citecolor=blue,urlcolor=blue]{hyperref}

\usepackage{natbib,paralist}

\DeclareMathOperator{\Cov}{Cov}

\newcommand{\N}{\mathbb{N}}
\newcommand{\Z}{\mathbb{Z}}

\newcommand{\R}{\mathbb{R}}
\newcommand{\T}{\mathbb{T}}

\newcommand{\cF}{\mathscr{F}}

\newcommand{\<}{\langle}
\renewcommand{\>}{\rangle}

\newcommand{\lip}{\text{\rm Lip}}

\renewcommand{\P}{\mathrm{P}}
\newcommand{\E}{\mathrm{E}}

\renewcommand{\d}{{\rm d}}

\newcommand{\e}{{\rm e}}
\renewcommand{\geq}{\geqslant}
\renewcommand{\leq}{\leqslant}
\renewcommand{\ge}{\geqslant}
\renewcommand{\le}{\leqslant}

\author{Davar Khoshnevisan\\University of Utah
	\and Kunwoo Kim\\POSTECH
	\and Carl Mueller\\University of Rochester
	\and Shang-Yuan Shiu\\National Central University}
\title{\bf Phase Analysis for a family of Stochastic Reaction-Diffusion Equations%
        \thanks{
	Research supported in part by the NSF grant DMS-1855439, DMS-1608575, and
	DMS-1307470 [D.K.], 
	the NRF grants 2019R1A5A1028324 and 2020R1A2C4002077 [K.K],
        a Simons grant [C.M.], and MOST grant MOST109-2115-M-008-006-MY2
         [S.-Y. S.].
	Parts of this research were funded by the NSF grant DMS-1440140
	while three of the authors [D.K., K.K., and C.M.]
	were in residence at the Mathematical Sciences Research
	Institute at UC Berkeley in Fall of 2015. Three of us
 	[D.K., K.K., and C.M.] wish to also thank the Banff International 
	Research Center for financial support under the Research in Teams Program.  
}}
\date{December 21, 2020}

\newtheorem{stat}{Statement}[section]
\newtheorem{proposition}[stat]{Proposition}

\newtheorem{corollary}[stat]{Corollary}
\newtheorem{theorem}[stat]{Theorem}
\newtheorem{lemma}[stat]{Lemma}
\theoremstyle{definition} 
\newtheorem*{definition}{Definition}
\newtheorem*{Itheorem}{Informal Theorem}

\newtheorem{remark}[stat]{Remark}
\newtheorem*{OP}{Open Problem}
\newtheorem{example}[stat]{Example}

\numberwithin{equation}{section}

\begin{document}
\maketitle
\begin{abstract} 
We consider a reaction-diffusion equation of the type
\[
	\partial_t\psi = \partial^2_x\psi + V(\psi) + \lambda\sigma(\psi)\dot{W}
	\qquad\text{on $(0\,,\infty)\times\T$},
\]
subject to a ``nice'' initial value and periodic boundary, where $\T=[-1\,,1]$ and
$\dot{W}$ denotes space-time white noise. The reaction term $V:\R\to\R$
belongs to a large family of functions that includes Fisher--KPP nonlinearities
[$V(x)=x(1-x)$] as well as Allen-Cahn potentials [$V(x)=x(1-x)(1+x)$], the multiplicative 
nonlinearity $\sigma:\R\to\R$ is non random and Lipschitz continuous, and $\lambda>0$
is a non-random number that measures the strength of the effect of the noise $\dot{W}$.

The principal finding of this paper is that: (i) When $\lambda$ is sufficiently large,
the above equation has a unique invariant measure; and (ii) When $\lambda$ is sufficiently small,
the collection of all invariant measures is a non-trivial line segment, in particular
infinite. This proves an earlier prediction of
\cite{ZTPS}. Our methods also say a great deal about the structure of these invariant
measures.\\

\noindent{\it Keywords:} Stochastic partial differential equations;
	invariant measures; phase transition.\\

	\noindent{\it \noindent AMS 2010 subject classification:}
	Primary: 60H15,  Secondary: 35R60.\newpage
\end{abstract}
\tableofcontents


\section{Introduction}

Let $\dot{W}=\{\dot{W}(t\,,x)\}_{t\ge0,x\in[-1,1]}$ denote a space-time white noise
on $\R_+\times\T$ where
\[
	\T:=[-1\,,1].
\]
That is, $\dot{W}$ is a centered, generalized Gaussian random field with covariance
measure
\[
	\Cov[\dot{W}(t\,,x) \,, \dot{W}(s\,,y)] = \delta_0(t-s)\delta_0(x-y)
	\text{ for all $s,t\ge0$ and $x,y\in\T$}.
\]
The principal goal of this paper is to study the large-time behavior of the stochastic reaction-diffusion
equation,
\begin{equation} \label{eq:RD}
	\partial_t\psi(t\,,x) =  \partial_x^2\psi(t\,,x) + V(\psi(t\,,x))  + \lambda\sigma(\psi(t\,,x))
	\dot{W}(t\,,x), 
\end{equation}
for $(t\,,x)\in(0\,,\infty)\times\T$,
with  periodic boundary conditions 
$\psi(t\,,-1)=\psi(t\,,1)$ for all $t>0$,
and a nice initial profile $\psi_0:\T\to[0\,,\infty)$.  The functions 
$\sigma$ and $V$ are  non-random, and fairly regular, and  
$\lambda>0$ is a  non-random quantity that represents
the strength of the noise $\dot{W}$. 

We will discuss the technical details about $\sigma$, $\psi_0$, $V$, \dots in the next section.
For now, we mention only that permissible choices of $V$ include
Fisher--KPP type non linearities [$V(z) = z(1-z)$] as well as Allen-Cahn type 
potentials [$V(z)=z(1-z)(1+z)$].

The main findings of this paper can be summarized as follows.

\begin{Itheorem}{\it
	Under nice regularity conditions on $\sigma$, $V$, and $\psi_0$:
	\begin{compactenum}
		\item \eqref{eq:RD} is well posed;
		\item \eqref{eq:RD} has an invariant probability measure $\mu_0$;
		\item \protect{\rm [High-noise regime]}
			There exists a non-random number $\lambda_1>0$ such that
			the only invariant probability
			measure of \eqref{eq:RD} is $\mu_0$ when $\lambda>\lambda_1$; and
		\item \protect{\rm [Low-noise regime]}
			There exists a non-random number $\lambda_0>0$ such that
			if $\lambda<\lambda_0$, then \eqref{eq:RD} has infinitely-many invariant measures. 
			Moreover, there 
			exists a probability measure $\mu_1$ -- singular with respect to 
			$\mu_0$ -- such that the 
			line segment 
			\begin{equation}\label{cal:L}
				\mathscr{I\hskip-0.5em{M}} :=
				\left\{ (1- a) \mu_0 + a \mu_1:\, 0\le a\le1\right\}
			\end{equation}
			coincides with all invariant probability measures of \eqref{eq:RD}.
	\end{compactenum}}
\end{Itheorem}

In their rigorous forms -- see Theorem \ref{th:RD} below -- assertions 3 and 4 together verify a
corrected version of earlier 
predictions of \cite{ZTPS}, deduced originally via 
experiments and computer simulations. Equally significantly, our proof of the rigorous form of Item
4 clearly ``explains'' why there is phase transition in the low-noise regime.

We conclude the Introduction by setting forth some notation that will be
used throughout the paper.

We will often write $\psi(t)$ in place of the mapping
$x\mapsto \psi(t\,,x)$;  we do likewise for other functions that may depend on 
extra parameters.  

We always write $\bm{1}_G$ for the indicator function of any and every set $G$. That is,
$\bm{1}_G(z)$ is equal to one if $z\in G$ and $\bm{1}_G(z)=0$ otherwise. When
$G$ is an event in our underlying probability space $(\Omega\,,\mathcal{F},\P)$, we follow
common practice and omit the variable $z\in\Omega$ in the expression $\bm{1}_G(z)$.

For any Banach space $\mathbb{X}$, we let $C(\mathbb{X})$ denote the space of
all continuous functions $f:\mathbb{X}\to\R$. The space $C(\mathbb{X})$ is always
a Banach space endowed with the supremum norm,
\[
	\|f\|_{C(\mathbb{X})} :=\sup_{x\in\mathbb{X}}|f(x)|.
\]
We always write $C_+(\mathbb{X})$ for the cone of non-negative elements of 
$C(\mathbb{X})$ and $C_{>0}(\mathbb{X})$ for the cone of strictly positive elements of 
$C(\mathbb{X})$, i.e., $f(x)>0$ for all $x\in\mathbb{T}$;  $C_b(\mathbb{X})$ denotes the bounded
elements of $C(\mathbb{X})$. The space $C_b(\mathbb{X})$ is always
endowed with the topology of pointwise convergence. That is, we endow
$C_b(\mathbb{X})\subset\R^{\mathbb{X}}$ with relative topology, where $\R^{\mathbb{X}}$ is given 
the discrete topology. By $C(\mathbb{X}\,;\mathbb{Y})$ we mean the space of all continuous
function on the Banach space $\mathbb{X}$ that take value in the Banach space $\mathbb{Y}$. 
We always topologize $C(\mathbb{X}\,;\mathbb{Y})$ using the norm
defined by 
\[
	\|f\|_{C(\mathbb{X};\mathbb{Y})} := \sup_{x\in\mathbb{X}} \|f(x)\|_{\mathbb{Y}}.
\]
In particular, $C(\mathbb{X}\,;\R)= C(\mathbb{X})$.

We topologize $\T=[-1\,,1]$ so that it is the one-dimensional 
torus. That is, $\T$ is always endowed with addition mod $2$, and $\pm1$ are identified with one
another using the usual quotient topology on $\T$. 
Let us emphasize in particular that
\[
	\lim_{x\to\pm1}f(x)= f(1)=f(-1)
	\qquad\text{for every $f\in C(\T)$.}
\]
Haar measure on $\T$ is always normalized to have total mass 2,
and its infinitesimal is denoted by symbols such as $\d x, \d y, \ldots$ in the context of 
Lebesgue integration.

For every $\alpha\in(0\,,1)$,  let $C^\alpha(\T)$ denote the collection of
all $f\in C(\T)$ that satisfy $\|f\|_{C^\alpha(\T)}<\infty$, where\footnote{
	Some authors use instead the norm defined by 
	\[
		|f|_{C^\alpha(\T)} := |f(0)|+ \sup_{\substack{x,y\in\T\\x\neq y}}
		\frac{|f(y)-f(x)|}{|y-x|^\alpha}.
	\]
	The two norms are equivalent, since
	$|f|_{C^\alpha(\T)} \le \|f\|_{C^\alpha(\T)}
	\le 2^\alpha|f|_{C^\alpha(\T)}.$
}
\[
	\|f\|_{C^\alpha(\T)} := \| f\|_{C(\T)} + \sup_{\substack{x,y\in\T\\x\neq y}}
	\frac{|f(y)-f(x)|}{|y-x|^\alpha}.
\]
Thus, $f\in C^\alpha(\T)$ iff $f$ is H\"older continuous with index $\alpha$.
In keeping with previous notation, $C^\alpha_+(\T)$ denotes the cone of all non-negative
elements of $C^\alpha(\T)$.

The $L^k(\Omega)$-norm
of a random variable $Z\in L^k(\Omega)$ is denoted by
$\|Z\|_k := \{ \E(|Z|^k)\}^{1/k}$ for all $1\le k<\infty$.

We let $\cF:=\{\cF_t\}_{t\ge0}$ designate the ``Brownian filtration.'' That is,
for all $t\ge0$, $\cF_t$ denotes the sub $\sigma$-algebra of $\mathcal{F}$
that is generated by all 
Wiener integrals of the form
\[
	\int_{(0,t)\times\T}f(x)\, W(\d s\,\d x),
\]
as $f$ ranges over $L^2(\T)$. By augmenting $\cF$ if need be, we may -- and always will -- 
assume that the filtration $\cF$ satisfies the ``usual conditions''
of stochastic integration theory. That is,
\[
	\cF_t=\bigcap_{s>t}\cF_s,
\]
and $\cF_t$ is complete for all $t\ge0$. We will require 
the ``usual conditions''  because, as is well known (and not hard to prove),
they ensure that the first hitting time of a closed set by a continuous,
$\cF$-adapted stochastic process is measurable.\footnote{A much deeper theorem of
	\cite{Hunt} extends this to cover the first hitting time of any Borel, and even analytic, set.
	We will not need that extension in the sequel.}

Finally, let us introduce two special elements $\mathbb{0},\mathbb{1}\in C_+(\T)$
as follows:
\begin{equation}\label{01}
	\mathbb{0}(x) := 0\quad\text{and}\quad
	\mathbb{1}(x) := 1\qquad\text{for all $x\in\T$}.
\end{equation}
Thus, in particular, $\delta_{\mathbb{0}}$ and $\delta_{\mathbb{1}}$ denote 
the probability measures on $C_+(\T)$ that put respective point masses
on the constant functions $\mathbb{0}$ and $\mathbb{1}$. 
The  measures $\delta_0$
and $\delta_{\mathbb{0}}$ should not be mistaken for one another; the former
is a probability measure on $\R$ and the latter is a probability measure on $C_+(\T)$.
Similar remarks apply to $\delta_1$ and $\delta_{\mathbb{1}}$.

Throughout we assume that the underlying probability space is complete.
\section{The Main Results}

In this section we introduce two theorems that make rigorous the Informal Theorem of
\S1. First, let us observe that the SPDE \eqref{eq:RD} is stated in terms of three functions
$\sigma$ (the ``diffusion coefficient''), $V$ (the ``potential''), and
$\psi_0$ (the ``initial profile'') which have not yet been described. Thus, we begin with a precise description
of those functions.

\subsection{Hypotheses on the Diffusion Coefficient}
Throughout this paper, we choose and fix a globally Lipschitz continuous
function $\sigma:\R\to\R$ such that
\[
	\sigma(0)=0.
\]
Let us define  
\begin{equation}\label{LL}
	{\rm L}_\sigma := \inf_{a\in\R\setminus\{0\}}
	\left| \frac{\sigma(a)}{a}\right|,\qquad
	\lip_\sigma:=\sup_{\substack{a,b\in\R:\\a\neq b}}
	\left| \frac{\sigma(b)-\sigma(a)}{b-a}\right|.
\end{equation}
Evidently, $0\le{\rm L}_\sigma \le \lip_\sigma$,
and
\[
	{\rm L}_\sigma|a| \le|\sigma(a)|\le\lip_\sigma|a|\qquad
	\text{for all $a\in\R$.}
\]
Because $\sigma$ is Lipschitz continuous, it follows that $\lip_\sigma<\infty$
and hence the second inequality above has content. We frequently assume
that the first inequality does too. That is, we often suppose in addition that ${\rm L}_\sigma>0$.
We will make explicit mention whenever this assumption is in place.


\subsection{Hypotheses on the Potential}\label{subsec:F}
Throughout the paper we are concerned with potentials $V$ of the form,
\begin{equation*}
	V(x) = x - F(x)\qquad\text{for all $x\in\R$},
\end{equation*}
where $F:\R_+\to\R_+$ is assumed to satisfy the following conditions:\\

\begin{compactenum}
\item[\bf (F1)] $F\in C^2(\R_+)$, $F(0)=0$, $F'\ge0$; 
\item[\bf (F2)] $\limsup_{x\downarrow0}F'(x)<1$ and $\lim_{x\to\infty}F'(x)=\infty$; and
\item[\bf (F3)] There exists a real number $m_0>1$ such that $F(x) = 
	O(x^{m_0})$ as $x\to\infty$.\\
\end{compactenum}

Examples are Allen-Cahn type potentials $[F(x)=x^3]$, 
as well as Fisher-KPP type nonlinearities $[F(x)=x^2]$. 

We make  references to $V(x)$ and $x-F(x)$ interchangeably throughout.
We also make references to the following elementary properties of the function $F$
without explicit mention.

\begin{lemma}\label{lem:F}
	$V$ and $F$ are locally Lipschitz and satisfy the following technical conditions:
	\begin{compactenum}
		\item $\limsup_{x\downarrow 0} (F(x)/x)<1$;
		\item $\lim_{x\uparrow\infty}(F(x)/x)=\infty$;
		\item $\sup_{x\ge0}V(x)<\infty$; and
		\item $\lim_{N\to \infty} \inf_{N\leq x< y \leq N+1} \{F(y)-F(x)\}/(y-x) =\infty$.
	\end{compactenum}
\end{lemma}

\begin{proof}
	$V$ and $F$ are locally Lipschitz because $F'$ is continuous.  
	Thanks to {\bf (F2)},  we can find $r_0>0$
	such that $\sup_{(0,r_0)}F'<1$.
	Apply the fundamental theorem of calculus to deduce 
	part 1 from {\bf (F1)}, and that 
	\[
		\frac{F(x)}{x} = \frac1x\int_0^x F'(a)\,\d a \ge \frac1x\int_{x/2}^x F'(a)\,\d a
		\ge \frac12 \inf_{a\ge x/2}F'(a)\qquad\text{for all $x>0$}.
	\]
	Let $x\to\infty$ and appeal to {\bf (F2)} to arrive at part 2.
	Part 3 follows immediately from part 2	and the continuity, hence local boundedness, of 
	the function $V$. Finally, we may observe that whenever $N\le x<y\le N+1$, 
	\[
		\frac{F(y)-F(x)}{y-x} \ge \inf_{x\le z\le y} F'(z) \ge \inf_{N \le z \le N+1} F'(z)\to\infty\qquad\text{as 
		$N\to\infty$},
	\]
	thanks to mean value theorem and {\bf (F2)}. This concludes our demonstration.
\end{proof}

\begin{remark}
	The astute reader might have noticed that
	Lemma \ref{lem:F} requires only that $F$ satisfies  
	{\bf (F1)} and {\bf (F2)}. Condition {\bf (F3)} will be used later on in Lemma
	\ref{lem:cont} in order to establish quantitative, global-in-time, 
	spatial continuity bounds for the 
	solution to our SPDE \eqref{eq:RD}.
\end{remark}


\subsection{Hypothesis on the initial profile}

Throughout, we assume that
\[
	\psi_0\in C_+(\T),
\]
and $\psi_0$ is non random. This assumption is used everywhere in the paper and so we assume it
here and throughout without explicit mention. 
We frequently will assume additionally that 
$\psi_0\neq\mathbb{0}$ ($\mathbb{0}$ was define in (\ref{01})); equivalently, that $\psi_0>0$
on an open ball in $\T$. This assumption will be made explicitly
every time it is needed.


\subsection{The Main Results}
We do not expect the solution $\psi$ to the SPDE \eqref{eq:RD} to
be differentiable in either 
variable. Therefore, \eqref{eq:RD} must be interpreted in the generalized sense;
see \cite{wal86}.  From now on,
we regard the SPDE \eqref{eq:RD} as shorthand
for its mild -- or integral -- formulation which can be written as follows:
\begin{equation} \label{eq:mild}\begin{split}
	\psi(t\,,x) &= (\mathcal{P}_t\psi_0)(x)+ \int_{(0,t)\times\T} p_{t-s}(x\,,y) 
			V(\psi(s\,,y))\,\d s\,\d y \\
	&\hskip2.5in+ \lambda \int_{(0,t)\times\T} p_{t-s}(x\,,y) \sigma(\psi(s\,,y))\,W(\d s\,\d y);
\end{split}\end{equation}
where the function $p$ denotes the heat kernel for the operator
$\partial_t-\partial^2_x$ on $(0\,,\infty)\times\T$ with periodic boundary conditions,
and $\{\mathcal{P}_t\}_{t\ge0}$ denotes the associated heat semigroup. That is,
\begin{equation} \label{eq:heat-kernel-expansion}
	p_t(x\,,y)= \frac{1}{\sqrt{4\pi t}}\sum_{k=-\infty}^\infty
	\exp\left\{- \frac{(x-y+2k)^2}{4t}\right\}
	\quad\text{for all $t>0$ and $x,y\in\T$},
\end{equation}
and
\[
	\mathcal{P}_0f=f
	\quad\text{and}\quad
	(\mathcal{P}_tf)(x)=\int_{\T} p_t(x\,,y)f(y)\,\d y,
\]
for all $(t\,,x)\in(0\,,\infty)\times\T$ and for every $f\in C(\T)$ (say).

The first of our two main theorems is a standard existence and uniqueness theorem.  

\begin{theorem} \label{th:exist-unique}
	The random integral equation
	\eqref{eq:mild} has a predictable random-field solution 
	$\psi=\{\psi(t\,,x)\}_{t\in\R_+,x\in\T}$ that is unique among all solutions that are 
	a.s.\ in $C_+(\R_+\times\T)$. Moreover,
	$\psi(t)\in C_+^\alpha(\T)$ a.s.\ for every $t>0$ and $\alpha\in(0\,,1/2)$.
	If in addition $\psi_0\neq\mathbb{0}$,
	then $\psi(t\,,x)>0$ for all $(t\,,x)\in(0\,,\infty)\times\T$ a.s.
\end{theorem}

The second result of this paper describes the invariant measure(s) of the
solution to \eqref{eq:RD}. This is a meaningful undertaking since, as we shall see in 
Proposition \ref{pr:feller} below,
the  infinite-dimensional stochastic process $\{\psi(t)\}_{t\ge0}$ is a Feller process.

Recall the function $\mathbb{0}$ from \eqref{01}, 
and recall also that $\delta_{\mathbb{0}}$ denotes point mass on $\mathbb{0}$.
Because $\sigma(0)=0$, it is easy to see that if we replaced the initial profile $\psi_0$
with the initial profile $\mathbb{0}$, then the solution $\psi$ would be identically zero. This is basically
another way to say that $\delta_{\mathbb{0}}$ is always an invariant measure for \eqref{eq:RD}.
The next theorem explores the question of uniqueness for this invariant measure 
[``phase transition'' refers to the lack of uniqueness of an invariant measure
in this context].

\begin{theorem} \label{th:RD}
	If ${\rm L}_\sigma>0$, then there exist $\lambda_1>\lambda_0>0$
	such that the following are valid independently of the choice
	of the initial profile $\psi_0\in C_+(\T)\setminus\{\mathbb{0}\}$:
	\begin{compactenum}
		\item If $\lambda\in(0\,,\lambda_0)$, then:
			\begin{compactenum}
			\item There exists a unique  probability measure 
			$\mu_+$ on $C_+(\T)$ that is invariant for \eqref{eq:RD} 
			and $\mu_+\{\mathbb{0}\}=0$. Moreover,
			$\mu_+$ charges $C_{>0}(\T)$;
		\item {\rm (Ergodic decomposition).}
			The set of all  probability measures on $C_+(\T)$ that
			are invariant for \eqref{eq:RD}
			is the collection $\mathscr{I\hskip-0.5em{M}}$ $[$see \eqref{cal:L}$]$
			of all convex combinations of
			$\mu_1 := \mu_+$ and $\mu_0:= \delta_{\mathbb{0}}$;
		\item  For every $\alpha\in(0\,,1/2)$,  $\mu_+$ is a probability measure on $C^\alpha_+(\T)$
			and
			\begin{equation}\label{tail:mu+}
				\int  \|\omega\|_{C^\alpha(\T)}^k\,\mu_+(\d \omega)<\infty
				\qquad\text{for every real number $k\ge2$}.
			\end{equation}
		\item {\rm (Ergodic theorem).}
			$\mu_+(\bullet) =\lim_{T\to\infty}T^{-1} \int_0^T \P\{ \psi(t)\in \bullet\}\,\d t$,
			where convergence holds in total variation; and
		\end{compactenum}
		\item If $\lambda>\lambda_1$, then $\delta_{\mathbb{0}}$ 
			is the only invariant measure for \eqref{eq:RD},
			and almost surely, $\lim_{t\to\infty}\psi(t)=\mathbb{0}$ in $C(\T)$. In fact, 
			$\limsup_{t\to\infty} t^{-1} \log \|\psi(t)\|_{C(\T)}<0$ a.s.
	\end{compactenum}
\end{theorem}

\begin{remark}\label{rem:RD}
	As a by-product of our {\it a priori} estimates, we shall prove 
	that \eqref{tail:mu+} can sometimes be improved upon.
	For instance, if there exists $\nu>0$ such that $F(x)=x^{1+\nu}$
	for all $x\ge0$, then it follows from Example \ref{ex:R:1}
	below [see also Example \ref{ex:R}] that for every $\alpha\in(0\,,1/2)$ there exists a real number
	$q>0$ such that
	\[
		\int \exp\left( q\|\omega\|_{C^\alpha(\T)}^{\nu/2(1+\nu)}\right)
		\mu_+(\d\omega)<\infty.
	\]
	Among others, this remark applies to the Fisher-KPP $(\nu=1)$ and the Allen-Cahn  ($\nu=2$) SPDEs.
\end{remark}

Some of the content of Theorem \ref{th:RD} was predicted earlier by \cite{ZTPS}.
One might ask the following open question.

\begin{OP}
	Does there exist a unique number $\lambda_c\in[\lambda_0\,,\lambda_1]$ such 
	that \eqref{eq:RD} has a unique invariant measure when $\lambda>\lambda_c$
	and infinitely-many invariant measures when $0<\lambda<\lambda_c$? 
\end{OP}

A standard method for proving that critical exponents exist in interacting particle systems
is to establish a suitable monotonicity property and use comparison arguments.
Since $\psi$ is not monotone in $\lambda$, comparison arguments are likely to not 
help with this problem. In this context, 
we hasten to add that we have also found no mention of this problem in the work of  
\cite{ZTPS}.  

Part (c) of Theorem  \ref{th:RD} and Remark \ref{rem:RD} contain information about
the non-trivial extremum $\mu_+$ of the collection $\mathscr{I\hskip-0.5em{M}}$ of all invariant measures of
\eqref{eq:RD} when $\lambda$ is small. More detailed information is presented in 
\S\ref{sec:support}. For example, we prove that $\mu_+(C^{1/2}(\T))=0$
under a mild additional
constraint \eqref{(F4)} on the nonlinearity $F$. This complements part (c) of
Theorem \ref{th:RD}, as the latter implies that $\mu_+(C^\alpha(\T))=1$
for every $\alpha\in(0\,,1/2)$. It is also proved in \S\ref{sec:support} that the typical
function in the support of $\mu_+$ doubles the Hausdorff dimension of every non-random set
in the same way as Brownian motion \citep{McKean}.

Let us conclude this section by mentioning that
we will prove Theorem \ref{th:exist-unique} in 
Section \ref{sec:pf-thm-1}. Theorem \ref{th:RD} and Remark \ref{rem:RD} are proved
in Section \ref{sec:pf-RD}, and the intervening five sections
are devoted to developing the requisite technical results
that support the argument of Section \ref{sec:pf-RD}.
A few additional technical results are included in the appendices.




\section{Proof of Theorem \ref{th:exist-unique}}\label{sec:pf-thm-1}

One of the first goals of this section is to establish the 
well-posedness of the SPDE \eqref{eq:RD} 
subject to periodic boundary, and  non-random initial value 
$\psi(0)=\psi_0\in C_+(\T)$.

For the duration of the proof we will 
choose and fix a number $\alpha\in(0\,,1/2)$, and assume that
\[
	\lambda=1.
\]
This assumption can be made without incurring any loss in generality.
For otherwise we can replace $\sigma$ by $\bar\sigma:=\lambda\sigma$
and observe that $\bar\sigma$ too is Lipschitz continuous and vanishes at 
the origin.

Standard well-posedness theory for SPDE requires Lipschitz continuous 
coefficients, but the diffusion term $V$ is not Lipschitz.  Thus, we will 
truncate $V$ and take the limit as we remove the truncation.  To that 
end, let us define for all integers $N\ge 1$ a function $V_N:\R\to\R$ 
as follows:
\begin{equation}\label{V_N}
	V_N(w) := \begin{cases}
		0 &\text{if $w\le 0$},\\
		V(w)&\text{if $0<w< N$},\\
		V(N)&\text{if $w\geq N$}.
	\end{cases}
\end{equation}

Every  $V_N$ is a bounded and Lipschitz continuous function. 
Standard results \citep[see][Chapter 3]{wal86} ensure 
that for each integer $N\ge 1$ there exists a 
unique continuous random field $\psi_N$ that solves the SPDE,
\begin{equation} \label{eq:RD-N}
	\partial_t\psi_N(t\,,x) =  \partial_x^2\psi_N(t\,,x) +
	V_N\left(\psi_N(t\,,x)\right) + 
	\sigma\left(\psi_N(t\,,x)\right) \dot{W}(t\,,x),
\end{equation}
for $(t\,,x)\in(0\,,\infty)\times\T$, and subject to periodic boundary conditions
with $\psi_{N,0}(x) = \psi_N(0\,,x) = \psi_0(x)$ for all $x\in\T$. 
To be precise, $\psi_N$ is the unique solution to the following mild 
formulation of \eqref{eq:RD-N} [compare with \eqref{eq:mild}]:
\begin{equation} \label{eq:key:psi}\begin{split}
	\psi_N(t\,,x) &= (\mathcal{P}_t\psi_0)(x)+ \int_{(0,t)\times\T} p_{t-s}(x\,,y) 
		V_N(\psi_N(s\,,y))\,\d s\,\d y \\
	&\hskip2in + \int_{(0,t)\times\T} p_{t-s}(x\,,y) \sigma(\psi_N(s\,,y))\,W(\d s\,\d y).  
\end{split}\end{equation}
where, as in \eqref{eq:mild} and \eqref{eq:heat-kernel-expansion}, 
$p$ denotes the heat kernel and $\{\mathcal{P}_t\}_{t\ge0}$ the
corresponding heat semigroup. We pause to mention a technical aside: The standard literature on
SPDE does not treat $\T$ as the torus, rather as an interval in $\R$. Therefore, ``continuity''
in that setting does not quite mean what it does in the present setting. For this reason,
we complete the above picture by claiming, and subsequently proving, that $(t\,,x)\mapsto \psi_N(t\,,x)$
is indeed a.s.\ continuous in the present sense for every $N\ge1$. Because of the Kolmogorov
continuity theorem, this claim follows as soon as we prove that
$\psi_N(t\,,1)=\psi_N(t\,,-1)$ a.s.\ for every $t>0$ and $N\in\N$.
But this follows immediately from \eqref{eq:key:psi}
and the elementary fact that  $p_t(x\,,y)=p_t(x+2\,,y)$ for every $x,y\in\T$,
where we recall addition in $\T$ is always addition mod 2. We return to the bulk of the argument.

Because $V_N(0)=0$, $\sigma(0)=0$, and $\psi_0\ge0$, the comparison theorem 
for SPDEs (Lemma \ref{lem:comparison}) tells us that
with probability one,
\begin{equation}\label{eq:psi_N:ge:0}
	\psi_N(t\,,x) \ge 0\qquad\text{for all $t\ge0$,
	$x\in\T$, and $N\ge 1$.}    
\end{equation}

Next, let us observe that the $V_N$'s have the following consistency property:
$V_N(x)=V_{N+1}(x)$ for all 
$x\in[0\,,N]$. Furthermore,
Part 4 of Lemma \ref{lem:F} ensures that there exists a non-random integer $N_0>1$ such that
\begin{equation*}
	V_N(x) \geq V_{N+1}(x)\quad\text{for all $x\in\R$ and all $N\ge N_0$.}
\end{equation*}
Thus, a second appeal to the comparison theorem for SPDEs 
(Lemma \ref{lem:comparison}) yields the following a.s\ bound:
\begin{equation*}
	\psi_N(t\,,x) \geq \psi_{N+1}(t\,,x)\quad\text{for all $t\ge0$, 
	$x\in\T$, and all $N\ge N_0.$}
\end{equation*}
In particular,
\begin{equation*}
	\psi(t\,,x) := \lim_{N\to\infty}\psi_N(t\,,x)
\end{equation*}
exists for all $t\ge0$ and $x\in\T$ off a single $\P$-null set. 
Since the infinite-dimensional process $t\mapsto\psi_N(t)$ is 
progressively measurable with respect to the underlying Brownian 
filtration $\cF:=\{\cF_t\}_{t\geq0}$, so is the random process 
$t\mapsto\psi(t)$. 

Define  
\begin{equation*}
	T_N := \inf\left\{ t\ge0: \|\psi_N(t)\|_{C(\T)} \ge N\right\}
	\qquad\text{for all $N\in\N$},
\end{equation*}
where $\inf\varnothing:=\infty$. Since $\{f\in C(\T):\, \|f\|_{C(\T)}\ge N\}$
is closed for every $N\in\N$, and because $\cF$ satisfies the ``usuall conditions'' of
stochastic integration theory, every $T_N$ is an $\cF$-stopping time.
The uniqueness assertion
for each $\psi_N$, and elementary properties of the 
Walsh stochastic integral \citep{wal86}, together imply that
\begin{equation*}
	\psi_N(t) = \psi_{N+1}(t)\qquad\text{for all $t\in[0\,,T_N)$,}
\end{equation*}
almost surely for all large $N$. We can iterate this to see that $\psi_N(t)=\psi_{N+M}(t)$
for all $t\in[0\,,T_N)$ a.s.\ for all large $N$ and $M\ge1$. Let $M\to\infty$ to find that
\begin{equation*}
	\psi(t) = \psi_N(t)\qquad\text{for all $t\in[0\,,T_N)$},
\end{equation*}
almost surely for every $N\ge1$. Because $\psi_N$ is a.s.\ continuous, 
it follows that $T_N$ can also be realized as the first
time $t\ge0$ that $\|\psi(t)\|_{C(\T)}\ge N$. In particular, 
$T_N\le T_{N+1}$ a.s.\ for all $N\ge1$ and hence,
\begin{equation*}
	T_\infty := \lim_{N\to\infty}T_N\quad\text{exists a.s.}
\end{equation*}
Moreover, it is immediate that $\psi$ is almost surely continuous on 
$(0\,,T_\infty)\times\T$ a.s. 

\begin{proposition}\label{pr:tail:T_N}
	$T_\infty=\lim_{N\to\infty} T_N=\infty$ a.s. Therefore, $\psi\in C_+(\R_+\times\T)$ a.s.
\end{proposition}

\begin{proof}
	Since $\psi_N\ge0$ a.s.\ [see \eqref{eq:psi_N:ge:0}]
	and $V_N(w)\le w$ for all $w\ge0$, the comparison 
	theorem for SPDEs (Lemma \ref{lem:comparison})
	ensures that $0\le \psi_N(t\,,x) \le  u (t\,,x)$ for all $t\ge0$ 
	and $x\in\T$, where $ u $
	denotes the  unique, continuous mild solution to the SPDE ,
	\begin{equation*}
		\partial_t u (t\,,x) = 
		\partial_x^2 u (t\,,x) 
		+  u (t\,,x)+ 
		\sigma( u (t\,,x))\dot{W}(t\,,x),
	\end{equation*}
	for $t>0$ and $x\in\T$, subject to periodic boundary conditions
	and  initial profile $u (0)=\sup_{x\in \T} \psi_0(x)$.  See \cite{wal86}.
	
	Define
	\begin{equation*}
		g(t\,,w) := \e^{-t}\sigma\left( \e^t w\right)\qquad\text{for all $t\ge0$
		and $w\in\R$}.
	\end{equation*}
	It is easy to see that we may write $ u (t\,,x) = \e^t   U (t\,,x)$,
	where $  U $ denotes the unique, continuous mild solution to the SPDE,
	\begin{equation*}
		\partial_t  U (t\,,x) = \partial_x^2
		  U (t\,,x) 
		+ g(t\,,  U (t\,,x))\dot{W}(t\,,x),
	\end{equation*}
	for $t>0$ and $x\in\T$, subject to same boundary and initial values as $ u $ was.
	This elementary fact holds because:
	\begin{compactenum}
		\item The fundamental solution to the linear operator
			$\varphi \mapsto \partial_t \varphi - \partial^2_x \varphi - \varphi$ is
			$\exp(t)$ times the fundamental solution to the heat operator
			$\varphi\mapsto \partial_t \varphi - \partial^2_x\varphi$; and
		\item $|g(t\,,w)-g(t\,,z)|\le\lip_\sigma |w-z|$ uniformly for all $t\ge0$
			and $w,z\in\R$. Therefore, we can apply standard methods from SPDEs to establish
			the said results about the random field $  U $; see  \citet[][Chapter 3]{wal86}.
	\end{compactenum}
	Because $ u $ and $  U $ are
	also nonnegative a.s., we can combine the preceding comparison arguments in order
	to see that
	\begin{equation} \label{psi:V}
		\sup_{N\in\N}
		\E\left(\sup_{t\in[0,T]}\left\| \psi_N(t)\right\|_{C(\T)}^k\right) \le 
		\e^{kT}\E\left( \sup_{t\in[0,T]}\left\| 
		  U (t)\right\|_{C(\T)}^k\right),
	\end{equation}
	for all real numbers $k\ge2$ and $T>0$.
	
	Now the proof of Proposition 4.1 in 
	\cite{KKMS20} can easily be repurposed in order to show that 
	there exist real numbers $D_1,D_2>0$
	(depending only on $\lip_\sigma$ and $\|\psi_0\|_{C(\T)}$) such that
	\[
		\E\left( |   U (t\,,x)|^k\right) \le D_1^k\exp\{D_2 k^3t\},
	\]
	uniformly for all $x\in\T$ and $t\ge0$. Also, standard methods
	\citep[see][Remark 4.3]{KKMS20}
	can be employed to show that
	\begin{equation}\label{U:KCT}
		\sup_{\substack{x,y\in\T\\ 0\le s,t\le T\\(t,x)\neq(s,y)}}\left\| 
		\frac{|  U (t\,,x)-  U (s\,,y)|}{|x-y|^{1/2}
		+|s-t|^{1/4}}\right\|_k<\infty
		\qquad\text{for every $T>0$ and $k\ge2$.}
	\end{equation}
       We may take the supremum
	over all $s,t\in[0\,,T]$ (and not $s,t\in[t_0\,,T]$) because we are not
	maximizing over an auxiliary parameter $\lambda$ in the present
	setting.
	In any case,
	we may combine the preceding  in order to 
	deduce that
	\begin{equation}\label{U:in:Lk}
		\sup_{t\in[0,T]}\|  U (t)\|_{C(\T)}\in L^k(\Omega)
		\qquad\text{for every $k\ge 2$ and $T>0$.}
	\end{equation}
	This and \eqref{psi:V} in turn imply that
	\begin{equation}\label{P(T_N>T)}
		\P\left\{ T_N\le T\right\} = 
		\P\left\{ \sup_{t\in[0,T]}\left\| \psi_N(t)\right\|_{C(\T)}
		\ge N\right\} = O(N^{-k})\qquad\text{as $N\to\infty$},
	\end{equation}
	which is more than good enough to complete the proof, since we have demonstrated already
	that $\psi_N\in C_+(\T)$ a.s.\ and $\psi=\psi_N$
	a.s.\ on $[0\,,T_N)\times\T$ for every $N\in\N$.
\end{proof}

We can now return to the proof of Theorem \ref{th:exist-unique} and establish 
that the continuous, progressively-measurable process $\psi$ is indeed
the solution to \eqref{eq:RD}.

As part of the proof of Proposition \ref{pr:tail:T_N} we showed that
\[
	\sup_{t\in[0,T]}\|\psi(t)\|_{C(\T)}\in \bigcap_{k\ge 2}L^k(\Omega),
\]
for all real numbers $T>0$; see for example \eqref{P(T_N>T)}. Since
$\sigma$ is Lipschitz, this fact and a standard appeal to the Burkholder-Davis-Gundy
inequality together show that
\begin{equation}\label{I}
	\mathcal{I}(t\,,x) := \int_{(0,t)\times\T}
	p_{t-s}(x\,,y) \sigma(\psi(s\,,y))\, W(\d s\,\d y)
\end{equation}
is a well-defined Walsh stochastic integral for every $(t\,,x)\in\R_+\times\T$,
and $\mathcal{I}$ is continuous a.s. Furthermore, elementary properties of
stochastic integrals show that, because $\psi=\psi_N$ on $[0\,,T_N)\times\T$,
\begin{equation*}
	\mathcal{I}(t\,,x)=\int_{(0,t)\times\T}
	p_{t-s}(x\,,y) \sigma(\psi_N(s\,,y))\, W(\d s\,\d y),
\end{equation*}
for all $(t\,,x) \in[0\,,T_N)\times\T$, almost surely. 
This and \eqref{eq:key:psi} together imply that, with probability one,
$\psi$ solves \eqref{eq:mild}
simultaneously for all $(t\,,x)\in[0\,,T_N)\times\T$. Let $N\to\infty$
and appeal to Proposition \ref{pr:tail:T_N} in order to see that, indeed,
\eqref{eq:RD} has a continuous mild solution $\psi$. The uniqueness 
of this $\psi$ follows from the uniqueness of the $\psi_N$'s.

Next we prove that $\psi(t\,,x)\ge0$ a.s.\ for every $t>0$ and $x\in\T$. Indeed,
for every $N\in\N$,
\[
	\P\left\{ \psi(t\,,x) <0 ~,~ T_N >t \right\}
	= \P\left\{ \psi_N(t\,,x) <0 ~,~ T_N >t \right\}=0,
\]
since $\psi_N$ solves an SPDE with Lipschitz-continuous coefficients; 
see \cite{Mueller-Support} and \cite{Shiga94}. Since $\lim_{N\to\infty}\P\{T_N\le t\}=0$,
this proves that $\psi(t\,,x)\ge0$ a.s. 

Suppose in addition that $\psi_0\neq\mathbb{0}$. Thanks to the comparison theorem
\citep{Mueller-Support,Shiga94},
$\psi$ is everywhere bounded below by the solution to \eqref{eq:RD}.
[Apply comparison first to $\psi_N$ and then use the facts that $\psi=\psi_N$
on $[0\,,T_N)\times\T$ and $\lim_{N\to\infty}T_N=
\infty$ a.s.] Therefore, the strict positivity theorem of \cite{Mueller-Support} implies that
\[
	\P\left\{\inf_{s\in[0,t]}\inf_{x\in\T}\psi_N(s\,,x)\le0\right\}=0,
\]
for all $t>0$ and $N\in\N$;
see also \cite{Mueller-Nualart} and \citet[eq.\ (5.15)]{CJK12}. In particular, it follows that
\[
	\P\left\{ \inf_{x\in\T}\psi(t\,,x) \le 0 \right\}
	= \lim_{N\to\infty} \P\left\{ \inf_{x\in\T}\psi_N(t\,,x) \le 0 ~,~ T_N >t \right\}=0,
\]
for all $t\ge0$.
This completes all but one part of the proof of Theorem \ref{th:exist-unique}:
It remains only to prove that
$\P\{\psi(t)\in C^\alpha(\T)\}=1$ for every $\alpha\in(0\,,1/2)$ and $t>0$.\footnote{This is a somewhat subtle statement.
	For example, the condition ``$t>0$'' cannot in general be replaced by
	``$t\ge0$,'' as $\psi(0)$ need not be in $C^\alpha(\T)$ for any $\alpha\in(0\,,1/2)$.
}
This fact follows
immediately from the continuity of $\psi(t)$ -- shown earlier here --
and the fact that
\[
	\E\left( \|\psi(t)\|_{C^\alpha(\T)}\right)<\infty
	\qquad
	\text{for every $\alpha\in(0\,,1/2)$ and $t>0$};
\]
see Proposition \ref{pr:tight} below.
An inspection of the
proof of Proposition \ref{pr:tight} shows that our reasoning is not circular. Thus, we may conclude the proof of
Theorem \ref{th:exist-unique}.\qed\\

%
%

We pause to observe that the proof of Proposition \ref{pr:tail:T_N} can be slightly 
generalized in order to yield information about the rate at which $T_N$ goes to infinity as 
$N\to\infty$. Although we will not need the following in the sequel, we state and prove
it for potential future uses.

\begin{corollary}
	There exists a real number $L\ge 0$ such that 
	\begin{equation}\label{eq:tail:T_N}
		\P\{T_N \le T\} \le L\exp\left\{ - \frac{(\log N)^{3/2}}{L\sqrt T} \right\}
		\qquad\text{for all $N\in\N$ and $T>0$}.
	\end{equation}
	In particular,\footnote{We recall that $a_N=\Omega(b_N)$ as $N\to\infty$ for
		positive $a_N$ and $b_N$'s iff
		$\liminf_{N\to\infty}(a_N/b_N)>0$.}
	\[
		T_N =\Omega\left( \frac{(\log N)^3}{(\log\log N)^2}\right)\qquad\text{as $N\to\infty$ a.s.}
	\]
\end{corollary}

\begin{proof}
	One can combine \eqref{U:KCT} and \eqref{U:in:Lk}, and apply a chaining argument
	or see \citet[Proposition 5.8]{KKMS20},
	in order to see that there exists a number $L_1>0$ such that
	\[
		\E\left( \sup_{t\in[0,T]}\|U(t)\|_{C(\T)}^k\right) \le L_1^k\e^{L_1 k^3T}
		\qquad\text{for all $T>0$ and $k\ge 2$}.
	\]
	Therefore, we may deduce from \eqref{psi:V} that
	\[
		\sup_{N\in\N}
		\E\left( \sup_{t\in[0,T]}\|\psi_N(t)\|_{C(\T)}^k\right) \le L_1^k\e^{L_2 k^3T}
		\qquad\text{for all $T>0$ and $k\ge 2$},
	\]
	for a suitably large number $L_2>1$ that does not depend on $(k\,,T)$. In particular,
	the Chebyshev inequality yields a number $L_3>0$ such that
	\[
		\P\{T_N \le T\} \le \inf_{k\ge2}\frac{L_1^k\e^{L_2k^3T}}{N^k}
		\le \exp\left\{ - \frac{(\log N)^{3/2}}{L_3\sqrt T} \right\}
		\qquad\text{for all $T>0$, $N>L_1^2$}.
	\]
This implies \eqref{eq:tail:T_N} for all $N\in\N$, provided that
	we choose a sufficiently large $L$. The lower bound for the growth rate of $T_N$
	follows from that probability bound, a suitable appeal to the Borel-Cantelli lemma along
	the subsequence $N=2^n$ [$n\in\N$], and monotonicity. We omit the details.
\end{proof}

We now return to the main discussion and conclude the section by observing
that the proof of the positivity portion of
Theorem \ref{th:exist-unique} used a localization argument -- via stopping times
$\{T_N\}_{N=1}^\infty$ -- that reduced the positivity of the solution of
\eqref{eq:RD} to the positivity of the solution of an SPDE with Lipschitz-continuous coefficients.
The very same localization argument can be used, in conjunction with the comparison theorem
of SPDEs with Lipschitz-continuous
coefficients \citep{Mueller-Support,Shiga94}, in order to yield the following.

\begin{lemma} \label{lem:comparison}
	For every $i=1,2$, let $\psi^{(i)}$ denote the continuous solution to the following SPDE
	with periodic boundary conditions, as in \eqref{eq:RD}:
	\begin{equation*}\left[\begin{split}
		&\partial_t u(t\,,x)  =  \partial_x^2 u(t\,,x) + a_i(u(t\,,x)) 
			+ \sigma(u(t\,,x)) \dot{W}(t\,,x)
			\qquad\text{for $(t\,,x)\in(0\,,\infty)\times\T$};\\
		&\text{subject to}\quad u(0\,,x) =\psi_0^{(i)}(x)
			\hskip2in\text{for all $x\in\T$};   
	\end{split}\right.\end{equation*}
	where $\psi_0^{(i)}\in C(\T)$ is non random, and $a_i:\R\to\R$ is Lipschitz continuous.
	Suppose, in addition, that $\psi_0^{(1)}\le\psi_0^{(2)}$ everywhere on $\T$, and $a_1\le a_2$
	everywhere on $\R$.
	Then,
	\[
		\P\left\{\psi^{(1)}(t\,,x) \le \psi^{(2)}(t\,,x)
		\quad\text{for all $(t\,,x)\in \R_+\times\T$}\right\}=1.
	\]
\end{lemma}

We will repeatedly appeal to this comparison result.




\section{Existence of Invariant Measures}

We now begin the proof of Theorem \ref{th:RD}.  Throughout, let 
$\psi$ denote the unique mild
solution to \eqref{eq:RD} with values in $C_+(\R_+\times\T)$, whose existence was established in 
Theorem \ref{th:exist-unique}.


\subsection{Tightness}

The first stage of our demonstration of Theorem \ref{th:RD} is proof of tightness
in a suitable space. We begin the proof of tightness with a few technical lemmas.


\begin{lemma}\label{lem:p-p}
	There exists a real number $K>0$ such that
	\[
		\int_0^\infty\d s\int_{\T}\d y\
		| p_s(x\,,y)-p_s(z\,,y)| \le K|x-z|\log_+(1/|x-z|)
		\qquad\text{for all $x,z\in\T$},
	\]
	where $\log_+a:=\log(\e\vee a)$ for all $a\ge0$.
\end{lemma}

\begin{proof}
	Consider the free-space heat kernel for the operator $\partial_t-\partial^2_x$,
	given by
	\[
		G_t(x) := \frac{1}{\sqrt{4\pi t}}\exp\left( -\frac{x^2}{4t}\right)
		\qquad\text{for $t>0$ and $x\in\R$}.
	\]
	In this way we can see from \eqref{eq:heat-kernel-expansion} that 
	$p_t(x\,,y) = \sum_{k=-\infty}^\infty G_t(2k+x-y)$ for
	all $t>0$ and $x,y\in\T$. Thus, we find that
	\begin{equation}\label{p=exp}
		p_t(x\,,y) = \sum_{k=-\infty}^\infty \int_{-\infty}^\infty
		G_t(2w+x-y)\e^{2\pi iwk}\,\d w
		= \frac{1}{2}\sum_{k=-\infty}^\infty\e^{ - \pi^2k^2t -
		i\pi (x-y)k},
	\end{equation}
	by the Poisson summation formula \citep[see][p.\ 630]{FellerVol2}.
	It follows that
	\[
		|p_t(x\,,y) - p_t(z\,,y)| \le \frac{1}{2}\sum_{k=-\infty}^\infty
		\e^{-\pi^2 k^2t}\left| 1 - \e^{i\pi (x-z)k}\right|
		=\frac{1}{\sqrt{2}} \sum_{k=-\infty}^\infty
		\e^{-\pi^2 k^2t}\sqrt{1-\cos(\pi k|x-z|)},
	\]
	uniformly for all $t>0$ and $x,y,z\in\T$.
	Since $1-\cos a \le a^2\wedge 1$ for every $a\in\R$,
	\begin{equation}\label{p-p}
		|p_t(x\,,y) - p_t(z\,,y)| \le \sqrt 2 \sum_{k=1}^\infty
		\e^{-\pi^2 k^2t}\left((|x-z|\pi k)\wedge 1\right),
	\end{equation}
	where the implied constant is universal and finite.
	Integrate both sides over $y\in\T$ and $t\ge0$, and apply Tonelli's theorem to find that
	\[
		\int_0^\infty\d t \int_{\T}\d y\ |p_t(x\,,y)-p_t(z\,,y)|
		\le \frac{2\sqrt 2}{\pi^2} \sum_{k=1}^\infty \frac{(|x-z|\pi k)\wedge 1}{k^2}.
	\]
	Consider separately the cases that
	$|\pi k|\le |x-z|^{-1}$ and $|\pi k|>|x-z|^{-1}$ in order to finish in the case
	that $|x-z|\le1$. If $|x-z|>1$, then we use the trivial bound $|x-z|\le2$
	to see that 
	\[
		C:=\sup_{x,z\in\T}
		\int_0^\infty\d t \int_{\T}\d y\ |p_t(x\,,y)-p_t(z\,,y)|<\infty.
	\]
	The case $|x-z|>1$ follows from this, since $C\le C|x-z|\log_+(1/|x-z|)$.
\end{proof}

\begin{lemma}\label{p-p:2}
	There exists a real number $K>0$ such that
	\[
		\int_0^\infty\d s\int_{\T}\d y\,
		| p_s(x\,,y)-p_s(z\,,y)|^2 \le K|x-z|
		\qquad\text{for all $x,z\in\T$}.
	\]
\end{lemma}

\begin{proof}
  Use (\ref{p=exp}) to obtain that
	\begin{equation*}
	   p_t(x\,,y) - p_t(z\,,y) 
		= \frac{1}{2} \sum_{k=-\infty}^{\infty} \e^{-\pi^2 k^2 t}\e^{-i\pi yk}\left( \e^{-i\pi x k} - \e^{-i \pi z k}  \right) .
	\end{equation*}
	The orthogonality of $\left\{ \e^{-i\pi yk} :  k\in\Z  \right\}$ in $L^2([-1, 1])$  and the elementary bound used for (\ref{p-p}) give
	\begin{eqnarray}
	    \int_{\mathbb{T}}\d y\,|p_t(x\,,y) - p_t(z\,,y)|^2
		&=&  \sum_{k=-\infty}^{\infty} \e^{-2 \pi^2 k^2 t} \label{eq:p2}
		    \left[1 -  \cos\left(\pi k (x-z)\right)\right]\\
		&\le& \sum_{k=-\infty}^{\infty} \e^{-2 \pi^2 k^2 t}[\left(\pi^2 k^2(x-z)^2\right)\wedge 1]. \nonumber
	\end{eqnarray}
	Integrate both sides over $t\geq 0$ to have
	\begin{equation*}
	   \int_{0}^{\infty}\d t \int_{\R} \d y\,
		 | p_s(x\,,y)-p_s(z\,,y)|^2
		 \le \sum_{k=-\infty}^{\infty} \frac{(\pi^2(x-z)^2k^2)\wedge 1}{2\pi^2 k^2 t}.
	\end{equation*}
	Like in lemma \ref{lem:p-p}, we consider the cases that
	$|\pi k|\le |x-z|^{-1}$ and $|\pi k|>|x-z|^{-1}$ to finish the proof.
	
\end{proof}

\begin{lemma}\label{lem:moment}
	Let $\gamma := (64\lip_\sigma^2)^2 \vee \frac{1}{4}$.	
	For every real number $k\ge2$,
	\[
		\adjustlimits
		\sup_{t\ge0}\sup_{x\in\T}
		\E\left(|\psi(t\,,x)|^k\right) \le
		\left[ 2\|\psi_0\|_{C(\T)} + 4\mathcal{R}(k)\right]^k 
		\quad\text{where}\quad
		\mathcal{R}(k) := \sup_{y\ge0}\frac{V(y)+\gamma k^2y}{1+\gamma k^2},
	\]
	and $\mathcal{R}(k)<\infty$. Furthermore, $\mathcal{R}(k)\uparrow\infty$ as 
	$k\uparrow\infty$.
\end{lemma}

\begin{example}\label{ex:R}
	Suppose there exist real numbers
	$A,\nu>0$ such that $F(x)\ge Ax^{1+\nu}$ for all $x\ge0$. This happens
	with an identity and with $\nu=1$ in the Fisher-KPP case and 
	$\nu=2$ in the Allen-Cahn case. Then,
	$\mathcal{R}(k) \le \text{const}\cdot k^{2/\nu}.$
	Thus, Lemma \ref{lem:moment} asserts that there exists a positive real number $c
	=c(A\,,\nu\,,\psi_0)$ such that
	$\E(|\psi(t\,,x) |^k)\le c^k k^{2k/\nu}$ uniformly for all $k\ge2$, $t\ge0$, and $x\in\T$.
\end{example}

\begin{proof}
	Throughout, we choose and fix an integer $k\ge2$. 
	Recall the properties of the function $F$, hence also $V$, from \S\ref{subsec:F}.
	On one hand, Lemma \ref{lem:F} ensures that there exists $\rho = \rho(k)>0$ such that
	$(1+\gamma k^2)y\le F(y)$ for all $y\ge \rho$. On the other hand, 
	$(1+\gamma k^2)y\le (1+\gamma k^2)\rho $ when $y\in[0\,,\rho)$.
	Therefore, $(1+\gamma k^2)y\le (1+\gamma k^2)\rho  + F(y)$ for all $y\ge0$. This is another way to say that
	\[
		V(y) \le (1+\gamma k^2)\rho  - \gamma k^2 y
		\qquad\text{for all $y\ge0$.}
	\]
	In particular, $\mathcal{R}(k)<\infty$ because $\mathcal{R}(k)$ is the smallest such $\rho$.
	A comparison lemma (see Lemma \ref{lem:comparison}) now ensures that
	\begin{equation}\label{psi<phi}
		\P\left\{ \psi(t\,,x) \le \varphi(t\,,x)\ \text{for all $t\ge0$ and $x\in\T$}
		\right\}=1,
	\end{equation}
	where $\varphi$ solves the following semi-linear SPDE,
	\begin{equation}\label{eq:dom:spde}
		\partial_t\varphi(t\,,x)
		= \partial_x^2\varphi(t\,,x) + (1+\gamma k^2)\mathcal{R}(k)  - \gamma k^2\varphi(t\,,x) +
		\lambda\sigma(\varphi(t\,,x))\dot{W}(t\,,x),
	\end{equation}
	subject to a periodic boundary condition and the same initial condition $\varphi(0)=\psi_0$
	as for $\psi$. One can understand \eqref{eq:dom:spde} in mild form in two different ways:
	One way is to write \eqref{eq:dom:spde} in integral form using the Green's function for the heat operator
	$\partial_t-\partial^2_x$, as is done in \cite{wal86}. An equivalent but slightly different way is to
	write \eqref{eq:dom:spde} in mild form in terms of the fundamental solution $\tilde{p}$ of
	the perturbed operator $f\mapsto \partial_t f - \partial^2_x f - \gamma k^2 f$. Note that 
	\[
		\tilde{p}_t(x\,,y) = \e^{-\gamma k^2 t} p_t(x\,,y)
		\qquad\text{for every $t>0$ and $x,y\in[-1\,,1]$},
	\]
	where $p$ continues to denote the heat kernel for $\partial_t-\partial^2_x$; 
	see \eqref{eq:heat-kernel-expansion}. It follows from this discussion
	that $\varphi$ is the unique solution
	to the random integral equation,
	\begin{equation}\label{IE}\begin{split}
		\varphi(t\,,x) &= \e^{-\gamma k^2 t} (\mathcal{P}_t\psi_0)(x) + (1+\gamma k^2)
			\mathcal{R}(k)\int_{(0,t)\times\T}
			\tilde{p}_{t-s}(x\,,y)\,\d s\,\d y + \lambda\mathcal{J}(t\,,x)\\
		&= \e^{-\gamma k^2 t} (\mathcal{P}_t\psi_0)(x) + 
			\frac{(1+\gamma k^2)\mathcal{R}(k)}{\gamma k^2}\left( 1- \e^{-\gamma k^2 t}\right)
			+ \lambda\mathcal{J}(t\,,x),
	\end{split}\end{equation}
	where 
	\[
		\mathcal{J}(t\,,x) := \int_{(0,t)\times\T} \e^{-\gamma k^2(t-s)}
		p_{t-s}(x\,,y)\sigma(\varphi(s\,,y))\,W(\d s\,\d y).
	\]
	General theory \citep{wal86,Dalang1999} tells us that $\varphi$ exists and solves the 
	integral equation \eqref{IE} uniquely among all continuous predictable random fields.
	Moreover, 
	\[
		\adjustlimits
		\sup_{t\in(0,T)}\sup_{x\in\T}\E\left(|\varphi(t\,,x)|^k\right)<\infty
		\qquad\text{for all $T>0$}.
	\]
	We now estimate the above moments slightly more carefully.
	
	Because
	\begin{equation*}\label{L}
		0\le \e^{-\gamma k^2 t} (\mathcal{P}_t\psi_0)(x) + 
		\frac{(1+\gamma k^2)\mathcal{R}(k)}{\gamma k^2}\left( 1- \e^{-\gamma k^2 t}\right)\le
		\|\psi_0\|_{C(\T)} + 2\mathcal{R}(k) =: L,
	\end{equation*}
	uniformly over all $t\ge0$, an application of the Burkholder-Davis-Gundy (BDG) inequality 
	to the a.s.\ identity \eqref{IE} yields
	\begin{align*}
		\E\left(|\varphi(t\,,x)|^k\right) \le 2^{k-1}L^k + 2^{k-1}A_k
		\E\left( | \< \mathcal{J}\>_{t,x} |^{k/2}\right),
	\end{align*}
	where $A_k$ denotes the optimal constant in the BDG inequality \citep{BDG}, and
	\[
		\< \mathcal{J}\>_{t,x} := \int_0^t\d s\int_{\T}\d y\
		\e^{-2\gamma k^2(t-s)} p_{t-s}^2(x\,,y)\sigma^2(\varphi(s\,,y)).
	\]
	Thanks to \eqref{LL}, $\sigma^2(\varphi(s\,,y))\le \lip_\sigma^2|\varphi(s\,,y)|^2$ a.s.
	Therefore, Minkowski's inequality yields
	\begin{align*}
		\left\| \<\mathcal{J}\>_{t,x} \right\|_{k/2}
			&\le \lip_\sigma^2 \adjustlimits
			\sup_{s\in(0,t)}\sup_{y\in\T}
			\|\varphi(s\,,y)\|_k^2\int_0^t\d s\int_{\T}\d y\
			\e^{-2\gamma k^2(t-s)} p_{t-s}^2(x\,,y)\\
		&\le 2\lip_\sigma^2 \adjustlimits
			\sup_{s\in(0,t)}\sup_{y\in\T}
			\|\varphi(s\,,y)\|_k^2\int_0^t \left(1\vee (2s)^{-1/2}\right) \e^{-2\gamma k^2 s}\,\d s\\
		&\le \frac{4\lip_\sigma^2}{\sqrt{\gamma k^2}} \adjustlimits
			\sup_{s\in(0,t)}\sup_{y\in\T}
			\|\varphi(s\,,y)\|_k^2,
	\end{align*}
	using the facts that: (a) $p_{s}(x\,,y)\le 2(1\vee s^{-1/2})$ \citep[Lemma B.1]{KKMS20};
	(b) $p_{2(t-s)}(x\,,x) = \int_{\T} p_{t-s}^2(x\,,y) \d y$;
	and (c)
	\begin{align*}
		\int_0^t\left(1\vee (2s)^{-1/2}\right) \e^{-2\gamma k^2 s}\,\d s
			&\le \int_0^{1/2} (2s)^{-1/2} \e^{-2\gamma k^2 s}\,\d s + 
			\int_{1/2}^\infty \e^{-2\gamma k^2 s}\,\d s
			\le  2 (\gamma k^2)^{-1/2}.
	\end{align*}
Here, we used the fact that $\gamma k^2 \geq \sqrt{\gamma k^2}$ for all $k\geq 2$, thanks to the assumption that $\gamma= (64\lip_\sigma^2)^2 \vee \frac{1}{4}$.   
	According to \cite{CK91}, $A_k\le (4k)^{k/2}$ for all $k\ge 2$. Thus, we combine to find that
	\begin{align*}
		\E\left(|\varphi(t\,,x)|^k\right) &\le 2^{k-1}L^k + 2^{2k-1}k^{k/2}
			\left[ \frac{4\lip_\sigma^2}{\sqrt{\gamma k^2}} \adjustlimits
			\sup_{s\in(0,t)}\sup_{y\in\T}
			\|\varphi(s\,,y)\|_k^2\right]^{k/2}\\
		&\le 2^{k-1}L^k + \tfrac12
			\adjustlimits\sup_{s\in(0,t)}\sup_{y\in\T}
			\E\left(|\varphi(s\,,y)|^k\right),
	\end{align*}
	using the fact that $\gamma\geq (64\lip_\sigma^2)^2.$
	This immediately yields
	\[
		\adjustlimits
		\sup_{s\in(0,t)}\sup_{y\in\T}\E\left(|\varphi(s\,,y)|^k\right) \le (2L)^k.
	\]
	Let $t\uparrow\infty$ and recall \eqref{psi<phi} in order to deduce the announced bound
	for the moments of $\psi(t\,,x)$.
	
	In order to complete the proof, we  observe that $\mathcal{R}$ is nondecreasing, and
	for every $m\ge0$,
	\[
		\lim_{k\to\infty}
		\mathcal{R}(k) \ge  \liminf_{k\to\infty}
		\frac{V(m)+\gamma k^2 m}{1+\gamma k^2}=m.
	\]
	Let $m\to\infty$ to see that $\lim_{k\to\infty}\mathcal{R}(k)=\infty$.
\end{proof}

\begin{lemma}\label{lem:cont}
	Recall the constant $m_0>1$  from {\bf (F3)} and the function 
	$\mathcal{R}$ from Lemma \ref{lem:moment}.
	For every $\tau>0$, there exists $L_0=L_0(\tau\,,\lip_{\sigma})>0$ -- independent of $\psi_0$
	-- such that
	\[
		\sup_{t\ge\tau}
		\E\left( |\psi(t\,,x) - \psi(t\,,z)|^k\right) \le 
		L_0^k\left(k^{k/2}\left[\|\psi_0\|_{C(\T)} + \mathcal{R}(k)\right]^k 
		+ \left[ \|\psi_0\|_{C(\T)} + \mathcal{R}(m_0k)\right]^{m_0k}\right)
		|x-z|^{ k/2},
	\]
	uniformly for all $k\ge2$ and $x,z\in\T$. If, in addition, $\psi_0\in C^\alpha_+(\T)$ for
	some $\alpha\in(0\,,1/2)$, then for every $k\ge2$ there exists $L_k>0$ -- independent of $\psi_0$
	-- such that
	\[
		\sup_{t\ge0}
		\E\left( |\psi(t\,,x) - \psi(t\,,z)|^k\right) \le 
		L_k \left\{\|\psi_0\|_{C^\alpha(\T)}^k +1\right\}
		|x-z|^{\alpha k},
	\]
	uniformly for all $x,z\in\T$. 
\end{lemma}

\begin{proof}
	Choose and fix $t\ge\tau>0$ and $x,z\in\T$.	
	Thanks to \eqref{eq:mild}, we can write
	\[
		\|\psi(t\,,x) - \psi(t\,,z)\|_k \le I_1 + I_2 + \lambda I_3,
	\]
	where
	\begin{align*}
		I_1 &:= |(\mathcal{P}_t\psi_0)(x) - (\mathcal{P}_t\psi_0)(z)|,\\ 
		I_2 &:= \int_0^t\d s\int_{\T}\d y\
			\left| p_{t-s}(x\,,y) - p_{t-s}(z\,,y)\right|\| V(\psi(s\,,y))\|_k,\\
		I_3 &:= \|\mathcal{I}(t\,,x)-\mathcal{I}(t\,,y)\|_k,
	\end{align*}
	and where $\mathcal{I}$ is the random field that was defined by a stochastic convolution
	in \eqref{I}. We estimate $I_1$, $I_2$, and $I_3$ separately and in this order.
	
	We apply the triangle inequality in conjunction with \eqref{p-p} to find that
	\begin{align*}
		I_1 &\le  \|\psi_0\|_{C(\T)}\int_{\T} \left| p_t(x\,,y)- p_t(z\,,y)\right|\d y
			\le \text{const}\cdot  \|\psi_0\|_{C(\T)}|x-z| \sum_{k=1}^\infty
			k\e^{-\pi^2 k^2\tau} \\
		&\le \text{const}\cdot  \|\psi_0\|_{C(\T)} |x-z|,
	\end{align*}
	uniformly for all $x,z\in\T$, where the implied constant also does not depend on $\psi_0$.
	
	{\bf (F3)} ensures that there exists a number $A>0$
	such that 
	\[
		|V(z)|\le A(|z|+|z|^{m_0})\qquad
		\text{for all $z\ge0$.}
	\]
	Consequently,
	\begin{equation}\label{||V||}\begin{split}
		\E\left( |V(\psi(s\,,y))|^k\right) &\le 2^{k-1}A^k\left\{
			\E\left( |\psi(s\,,y)|^k\right)  + \E\left( |\psi(s\,,y)|^{m_0k}\right) \right\}\\
		&\le 2^{k-1}A^k\left\{ \left[ 2\|\psi_0\|_{C(\T)} + 4\mathcal{R}(k)\right]^k +
			\left[ 2\|\psi_0\|_{C(\T)} + 4\mathcal{R}(m_0k)\right]^{m_0k} \right\}\\
		&\le C_3^k\left\{\left[ \|\psi_0\|_{C(\T)}+\mathcal{R}(k)\right]^k
			+\left[ \|\psi_0\|_{C(\T)}+\mathcal{R}(m_0k)\right]^{m_0k}\right\},
	\end{split}\end{equation}
	uniformly for all $s>0$, $y\in\T$, and $k\ge2$, where $C_3=C_3(F)>0$.
	In this way, we find that
	\begin{align*}
		I_2 &\le C_4\left\{\left[ \|\psi_0\|_{C(\T)}+\mathcal{R}(k)\right]
			+\left[ \|\psi_0\|_{C(\T)}+\mathcal{R}(m_0k)\right]^{m_0}\right\}
			\int_0^\infty\d s\int_{\T}\d y\
			\left| p_{t-s}(x\,,y) - p_{t-s}(z\,,y)\right|\\
		&\le C_5 \left\{ \left[ \|\psi_0\|_{C(\T)}+\mathcal{R}(k)\right] +
			\left[ \|\psi_0\|_{C(\T)}+\mathcal{R}(m_0k)\right]^{m_0}\right\}
			|x-z|\log_+(1/|x-z|),
	\end{align*}
	where $C_4,C_5>0$ do not depend on $(k\,,x\,,z\,,\psi_0)$; see Lemma  \ref{lem:p-p}.
	
	Finally, we apply the BDG inequality \citep[see][]{BDG}, using the
	\cite{CK91} bound for the optimal BDG constant, in order to see that
	\begin{align*}
		I_3^k&\le (4k)^{k/2}\lip_\sigma^k \left\|\int_0^t\d s\int_{\T}\d y\
			\left| p_{t-s}(x\,,y) - p_{t-s}(z\,,y)\right|^2
			|\psi(s\,,y)|^2 \right\|_{k/2}^{k/2}\\
		&\le(4k)^{k/2}\lip_\sigma^k\left( \int_0^t\d s\int_{\T}\d y\
			\left| p_{t-s}(x\,,y) - p_{t-s}(z\,,y)\right|^2
			\|\psi(s\,,y)\|_{k}^2 \right)^{k/2}.
	\end{align*}
	We have used the Minkowski inequality in the final bound. Apply Lemma \ref{lem:moment}
	above together with Lemma \ref{p-p:2} in order to find 
	that 
	there exists a number $C_5=C_5(\lip_\sigma)>0$ such that
	\begin{align*}
		I_3 &\le  2\sqrt{k} \lip_\sigma
			\left[ 2\|\psi_0\|_{C(\T)} + 4\mathcal{R}(k)\right]
			\left( \int_0^t \d s\int_{\T}\d y\
			\left| p_{t-s}(x\,,y) - p_{t-s}(z\,,y)\right|^2
			\right)^{1/2}\\
		&\le C_5\sqrt{k} \left[ \|\psi_0\|_{C(\T)} + \mathcal{R}(k)\right]|x-z|^{1/2}.
	\end{align*}
	uniformly for all $t>0$ and $x,z\in\T$.
	The first part of the lemma follows 
	from combining the preceding estimates for $I_1$, 
	$I_2$, and $I_3$. 
	
	Next, suppose additionally that $\psi_0\in C^\alpha_+(\T)$ for some $\alpha\in(0\,,1/2)$.
	Then, the estimate for $I_1$ can be improved upon as follows: We may write
	\[
		I_1 = \left| \E \left[ \psi_0(\beta(t) + x) - \psi_0(\beta(t)+z) \right] \right|,
	\]
	for a standard  Brownian motion $\beta$ on $\T$. In this way we may write
	\[
		I_1 \le \E\left| \psi_0(\beta(t) + x) - \psi_0(\beta(t)+z) \right| \le
		\|\psi_0\|_{C^\alpha(\T)}|x-z|^\alpha,
	\]
	uniformly for all $t\ge0$ and $x,z\in\T$. The estimates for $I_2$ and $I_3$ remain
	the same. Combine things to finish the proof.
\end{proof}

We are ready for the tightness result.

\begin{proposition}\label{pr:tight}
	For every $\tau>0$ and $\alpha\in(0\,,1/2)$,
	there exists a number $L_1=L_1(\|\psi_0\|_{C(\T)}\,,\tau\,,\alpha)>0$ such that
	$a\mapsto L_1(a\,,\tau\,,\alpha)$ is non decreasing and
	\begin{equation}\label{eq:k-moment}
		\sup_{t\geq \tau}\E \left( \|\psi(t)\|_{C^\alpha(\T)}^k
		\right)\le L_1^k\left( \sqrt{k}\, \mathcal{R}(k) + [\mathcal{R}(m_0k)]^{m_0}\right)^k,
	\end{equation}
	uniformly for all real numbers $k\ge2$. Consequently,
	the laws of $\{\psi(t)\}_{t\ge\tau}$ are tight on $C^\alpha_+(\T)$.
	If, in addition, $\psi_0\in C^\alpha_+(\T)$ for some $\alpha\in(0\,,1/2)$,
	then 
	\begin{equation}\label{eq:k-moment:bis}
		\sup_{t\ge0}\E\left( \|\psi(t)\|_{C^\alpha(\T)}^k\right) <\infty
		\quad\text{for all $k\ge2$}.
	\end{equation}
\end{proposition}

\begin{proof}
	Related results appear in \cite{Cerrai03,Cerrai05}; 
	we prefer to give a more detailed proof. 
	
	We can combine Lemmas \ref{lem:moment} and \ref{lem:cont} with a quantitative
	form of Kolmogorov's continuity theorem  
	(see \citet[][Proposition 5.8]{KKMS20} or \citet[][Theorem 4.3]{spde-utah})
	to deduce \eqref{eq:k-moment} and \eqref{eq:k-moment:bis}. 
	On one hand, the Arzel\`a-Ascoli theorem implies that the set
	\[
		\Gamma_n := \{ f\in C^\alpha_+(\T): \|f\|_{C^\alpha(\T)} \le n \}
	\]
	is compact for every $n\in\N$. On the other hand, \eqref{eq:k-moment} and Markov's
	inequality together imply that
	\[
		\adjustlimits
		\lim_{n\to\infty}\limsup_{t\to\infty}
		\P\left\{ \psi(t)\not\in\Gamma_n\right\} =0.
	\]
	This readily implies tightness, and concludes the proof.
\end{proof}

\begin{example}\label{ex:R:1}
	Let us continue with Example \ref{ex:R} and suppose there exist real numbers
	$a,A,\nu>0$ such that
	\[
		Ax^{1+\nu} \ge F(x)\ge ax^{1+\nu}
		\qquad\text{for all $x\ge0$,}
	\]
	so that $\mathcal{R}(k) \le \text{const}\cdot k^{2/\nu}$ and $m_0=1+\nu$.
	According to Proposition \ref{pr:tight}, for every $\alpha\in(0\,,1/2)$
	there exists a positive real number $L\ge 0$ such that
	\[
		\sup_{t\ge1}\E\left(\|\psi(t) \|_{C^\alpha(\T)}^k\right)\le L^k k^{2(1+\nu) k/\nu}
		\quad\text{for all $k\ge2$}.
	\]
	It follows readily from this and Stirling's formula that there exists $q=q(\alpha\,,\nu)>0$ such that
	\[
		\sup_{t\ge1} \E\left[\exp\left( q\|\psi(t) \|_{C^\alpha(\T)}^{\nu/2(1+\nu)}\right)\right]
		<\infty.
	\]
\end{example}


\subsection{Temporal continuity}

The following is very well known for SPDEs with Lipschitz-continuous coefficients. 
We will prove that the following formulation holds  in the present case
that $V$ is not globally Lipschitz continuous exactly as it holds in the case that
$V$ were replaced by a globally Lipschitz function.

\begin{proposition}\label{pr:temporal:cont}
	For every $\theta\in(0\,,1/4)$ and $Q>1$ there exists a number
	$L=L(\theta\,,Q)>0$ such that for all $T>0$, and $k\ge2$,
	\[
		\E\left( \adjustlimits\sup_{t\in[T,T+Q]}\sup_{s\in(0,1]}
		\left\| \frac{\psi(t+s) - \psi(t)}{s^\theta} \right\|_{C(\T)}^k\right)
		\le L^k \left\{ T^{-3/4}\|\psi_0\|_{C(\T)}
		+ \mathcal{A}_k
		+ \mathcal{B}_k\sqrt{k}\right\}^k,
	\]
	where  $\mathcal{A}_k:=\sup_{r\ge0}\sup_{y\in\T}\|V(\psi(r\,,y)\|_k$
	and $\mathcal{B}_k := \sup_{r\ge0}\sup_{z\in\T}\| \psi(r\,,z)\|_k$ are finite; see
	\eqref{||V||} and Lemma \ref{lem:moment}. Also, for every $\theta\in(0\,,1/4)$,
	$Q>1$, and $k\ge2$ there
	exists a number $K=K(\theta\,,Q\,,k)>0$ -- independently of $\psi_0$ -- such that
	\begin{align*}
		&\sup_{T>0}\E\left( \adjustlimits\sup_{t\in[T,T+Q]}\sup_{s\in[0,\varepsilon]}
			\|\psi(t+s) - \psi(t)\|_{C(\T)}^k\right)\\
		&\hskip2.5in\le K \left\{\|\psi_0\|_{C(\T)}
			+1\right\}^k\varepsilon^{\theta k} + K\sup_{s\in[0,\varepsilon]}
			\left\| \mathcal{P}_s\psi_0 - \psi_0\right\|_{C(\T)}^k.
	\end{align*}
	uniformly for every $\varepsilon\in(0\,,1)$.
\end{proposition}

As we shall see from the proof, one can also keep track of the dependence of the constant
$K$ on $(\theta\,,Q\,,k)$. We omit the details.
The proof of Proposition \ref{pr:temporal:cont}
itself will proceed in a standard manner, but uses the Poisson summation formula
at a key juncture, as did the proof of Proposition \ref{pr:tight}.

Recall \eqref{eq:mild} and \eqref{I}, and write
\[
	|\psi(t+\varepsilon\,,x) - \psi(t\,,x)| \le I_1 + I_2 + I_3,
\]
where
\begin{equation*}\label{I1-I4}\begin{split}
	I_1 &:= \left| (\mathcal{P}_{t+\varepsilon}\psi_0)(x) - (\mathcal{P}_t\psi_0)(x)\right|,\\
	I_2 &:= \int_0^t\d s \int_{\T}\d y\
		|V(\psi(s\,,y))| \left| p_{t+\varepsilon-s}(x\,,y) -p_{t-s}(x\,,y) \right|,\\
	I_3 &:= \int_t^{t+\varepsilon}\d s\int_{\T}\d y\  |V(\psi(s\,,y))| p_{t+\varepsilon-s}(x\,,y),\\
	I_4 &:= \left| \mathcal{I}(t+\varepsilon\,,x) - \mathcal{I}(t\,,x)\right|.
\end{split}\end{equation*}
We estimate the $L^k(\Omega)$-norm of $I_1,\ldots,I_4$ in this order.
Proposition \ref{pr:temporal:cont} follows  from Lemmas
\ref{lem:I1}, \ref{lem:I2}, \ref{lem:I3}, and \ref{lem:I4}, Proposition \ref{pr:tight},
and a chaining argument.

\begin{lemma}\label{lem:I1}
	There exists a number $K>0$ such that
	\[
		I_1 \le  K\|\psi_0\|_{C(\T)}\left[\frac{1}{t^{1/4}}
		\min\left( 1\,, \frac{\varepsilon}{t}\right)^{1/2}\wedge 1\right],
	\]
	uniformly for all $x\in\T$ and $\varepsilon,t>0$.
\end{lemma}

\begin{proof}
	By the Cauchy-Schwarz inequality,
	\[
		I_1 \le \|\psi_0\|_{C(\T)} \int_{\T} |p_{t+\varepsilon}(x\,,y) - p_t(x\,,y)|\,\d y
		\le \|\psi_0\|_{C(\T)}\sqrt{2\int_{\T} |p_{t+\varepsilon}(x\,,y) - p_t(x\,,y)|^2\,\d y}.
	\]
	Therefore the result follows from \citet[][Lemma B.6]{KKMS20}.
\end{proof}

\begin{lemma}\label{lem:I2}
	There exists a finite number $K>0$ such that
	$\| I_2\|_k \le K\mathcal{A}_k \sqrt{\varepsilon},$
	uniformly for all $k\ge 2$, $x\in\T$, and $\varepsilon\in(0\,,1)$, and $t>0$.
\end{lemma}

\begin{proof}
	Minkowski's inequality yields
	\begin{equation}\label{eq:I2}
		\|I_2\|_k \le \mathcal{A}_k
		\int_{\T}\d y\int_0^\infty\d s\ |p_{s+\varepsilon}(x\,,y)-p_s(x\,,y)|.
	\end{equation}
	We may use the Poisson summation formula, as we did for Lemma \ref{lem:p-p},
	in order to see that for all $s,\varepsilon>0$ and $x,y\in\T$,
	\[
		\left| p_{s+\varepsilon}(x\,,y) - p_s(x\,,y) \right|
		\le \frac{1}{2} \sum_{k\in\Z\setminus\{0\}}\e^{-\pi^2 k^2 s}\left| 1 -
		\e^{-2\pi^2 k^2\varepsilon}\right|.
	\]
	Because $1-\exp(-q)\le (1\wedge q)$ for all $q>0$, it follows that
	\[
		\sup_{x,y\in\T}
		\int_0^\infty\left| p_{s+\varepsilon}(x\,,y) - p_s(x\,,y) \right|\d s
		\le \sum_{k=1}^\infty\left( k^{-2}\wedge\varepsilon\right),
	\]
	uniformly for all $\varepsilon>0$. This and \eqref{eq:I2} readily imply the lemma.
\end{proof}

\begin{lemma}\label{lem:I3}
	$\|I_3\|_k \le \mathcal{A}_k\varepsilon,$
	uniformly for all $k\ge 2$, $x\in\T$, and $\varepsilon,t>0$.
\end{lemma}

\begin{proof}
	Another appeal to Minkowski's inequality yields
	$\|I_3\|_k \le \mathcal{A}_k \int_t^{t+\varepsilon}\d s\int_{\T}\d y\  p_{t+\varepsilon-s}(x\,,y),$
	which is equal to $\mathcal{A}_k\varepsilon$.
\end{proof}

\begin{lemma}\label{lem:I4}
	There exists a finite number $K>0$ such that
	$\|I_4\|_k \le K\mathcal{B}_k\sqrt{k}\,\varepsilon^{1/4}$
	uniformly for all $k\ge 2$, $x\in\T$, and $\varepsilon,t>0$.
\end{lemma}

\begin{proof}
	We can write
	$I_4  = I_{4,1}+ I_{4,2}$, where
	\begin{align*}
		I_{4,1} &:= \int_{(0,t)\times\T}
			\left( p_{t+\varepsilon-s}(x\,,y) - p_{t-s}(x\,,y)\right) \sigma(\psi(s\,,y))\, W(\d s\,\d y),\\
		I_{4,2} &:= \int_{(t,t+\varepsilon)\times\T} p_{t+\varepsilon-s}(x\,,y)\sigma(\psi(s\,,y))\, W(\d s\,\d y).
	\end{align*}
	We estimate $I_{4,1}$ and $I_{4,2}$ separately, and in this order, using the BDG inequalities
	in the same manner as in the proof of Lemma \ref{lem:moment}. Indeed,
	\begin{align*}
		\|I_{4,1}\|_k^2 &\le A_k^{2/k}
			\left\| \int_0^t\d s\int_{\T}\d y\
			\left( p_{t+\varepsilon-s}(x\,,y) - p_{t-s}(x\,,y)\right)^2 |\sigma(\psi(s\,,y))|^2 \right\|_{k/2}\\
		&\le A_k^{2/k} \int_0^t\d s\int_{\T}\d y\
			\left( p_{t+\varepsilon-s}(x\,,y) - p_{t-s}(x\,,y)\right)^2 \|\sigma(\psi(s\,,y))\|_k^2,
	\end{align*}
	where $A_k$ denotes, as before, the optimal constant of the $L^k(\Omega)$-form of
	the BDG inequality. Thus, Lemma B.6 of \cite{KKMS20} implies that
	\[
		\|I_{4,1}\|_k \le L_1 A_k^{1/k}\adjustlimits\sup_{r\ge0}\sup_{z\in\T}\|\sigma(\psi(r\,,z))\|_k 
		\sqrt{\int_0^\infty \frac{1}{s^{1/2}}\min\left( 1\,,\frac{\varepsilon}{s}\right)\d s}
		\le L_2 A_k^{1/k}\mathcal{B}_k\,\varepsilon^{1/4}.
	\]
	We estimate $I_{4,2}$ using similar techniques. Namely,
	\begin{align*}
		\|I_{4,2}\|_k^2 &\le A_k^{2/k}
			\left\| \int_t^{t+\varepsilon}\d s\int_{\T}\d y\
			\left( p_{t+\varepsilon-s}(x\,,y) \right)^2 |\sigma(\psi(s\,,y))|^2 \right\|_{k/2}\\
		&\le A_k^{2/k}\lip_\sigma^2\mathcal{B}_k^2
			\int_0^\varepsilon\d s\int_{\T}\d y\
			\left( p_s(x\,,y) \right)^2.
	\end{align*}
	To finish our estimate for $I_{4,2}$, we may apply the semigroup property
	and the symmetry of the heat kernel to see that
	\[
		\int_0^\varepsilon\d s\int_{\T}\d y\
		\left( p_s(x\,,y) \right)^2 = \int_0^\varepsilon p_{2s}(0\,,0)\,\d s
		\le \sqrt 2\int_0^\varepsilon\frac{\d s}{\sqrt s}=2\sqrt{2\varepsilon};
	\]
	see \citet[Lemma B.1]{KKMS20} for the last bound. Finally,
	we combine the two bounds for $I_{4,1}$ and $I_{4,2}$ and appeal 
	to an inequality of \cite{CK91} that asserts that $A_k^{1/k}\le 2\sqrt k$
	for every $k\ge2$.
\end{proof}

\begin{proof}[Proof of Proposition \ref{pr:temporal:cont}]
	The first assertion of the proposition follows immediately from
	Lemmas \ref{lem:I1}--\ref{lem:I4} and a chaining argument. 
	We prove the second assertion of the proposition.
	
	Define
	\[
		\varphi(t\,,x) := \psi(t\,,x) - (\mathcal{P}_t\psi_0)(x)\qquad\text{for
		all $t\ge0$ and $x\in\T$}.
	\]
	We repeat the proof of Lemma \ref{lem:cont} in order to see that
	for every $k\ge 2$ there exists $L_k>0$ --
	independently of $\psi_0$ -- such that
	\[
		\sup_{t\ge0}\E\left( |\varphi(t\,,x) - \varphi(t\,,z)|^k\right) \le 
		L_k\left(  \|\psi_0\|_{C(\T)} + 1\right)^k
		|x-z|^{k/2},
	\]
	uniformly for all $x,z\in\T$. The difference between the above
	inequality and that of Lemma \ref{lem:cont} is that we are allowed to
	take a supremum over all $t\ge0$ [and not just $t\ge\tau>0$]; this is because
	$\mathcal{P}_t\psi_0$ has  been subtracted from $\psi$ [to yield
	$\varphi$]. 
	Lemma \ref{lem:moment} tells us that for every $k\ge2$ there exists
	$L_k'>0$ --
	independently of $\psi_0$ -- such that
	$\mathcal{A}_k^k\vee \mathcal{B}_k^k\le L_k'( \|\psi_0\|_{C(\T)} + 1)^k.$
	Therefore, Lemmas \ref{lem:I2}, \ref{lem:I3}, and \ref{lem:I4} tell us that
	for every $k\ge 2$ there exists $L_k''>0$ --
	independently of $\psi_0$ -- such that
	\[
		\adjustlimits\sup_{t\ge0}\sup_{x\in\T}
		\E\left( |\varphi(t+\varepsilon\,,x) - \varphi(t\,,x) |^k\right) \le
		L_k''\left( \|\psi_0\|_{C(\T)}+1\right)^k\varepsilon^{k/4},
	\]
	uniformly for every $\varepsilon\in(0\,,1)$. Apply these inequalities, together
	with a chaining argument, in order to see that for every $\theta\in(0\,,1/4)$
	and $k\ge2$ there exists $L_{k,\theta}>0$ -- independently of
	$\psi_0$ -- such that
	\[
		\sup_{T>0}\E\left( \adjustlimits\sup_{t\in[T,T+Q]}\sup_{s\in[0,\varepsilon]}
		\|\varphi(t+s) - \varphi(t)\|_{C(\T)}^k\right)
		\le L_{k,\theta} \left\{\|\psi_0\|_{C^\alpha(\T)}
		+1\right\}^k\varepsilon^{\theta k},
	\]
	uniformly for every $\varepsilon\in(0\,,1)$. This completes the proof of the proposition
	because
	\[
		\|\psi(t+s) - \psi(  t)\|_{C(\T)}^k
		\le 2^k\|\varphi(t+s) - \varphi(t)\|_{C(\T)}^k + 2^k \|
		\mathcal{P}_{t+s}\psi_0 - \mathcal{P}_t\psi_0\|_{C(\T)}^k,
	\]
	and 
	\begin{align*}
		\| \mathcal{P}_{t+s}\psi_0 - \mathcal{P}_t\psi_0\|_{C(\T)}
			&\le \left\|\mathcal{P}_t\left(  \mathcal{P}_s\psi_0 - \psi_0
			\right) \right\|_{C(\T)}\le \|\mathcal{P}_s\psi_0-\psi_0\|_{C(\T)},
	\end{align*}
	thanks to the semigroup property of $\{\mathcal{P}_r\}_{r\ge0}$.
\end{proof}


\subsection{The Feller property}\label{subsec:Feller}

Let $\psi$ denote the solution to \eqref{eq:RD}.
The existence, as well as regularity, of $\psi$ has been
established already in Theorem \ref{th:exist-unique}. We now study the Markov
properties of the infinite-dimensional process $\psi=\{\psi(t)\}_{t\ge0}$.

As is customary in Markov process theory, let $\P_\mu$
denote the law of the random field $\psi$ starting according to initial measure $\mu$
on $C_+(\T)$,
and let $\E_\mu$ denote the corresponding expectation operator. Until now,
$\P$ and $\E$ referred to $\P_{\psi_0}:=\P_{\delta_{\psi_0}}$ and
$\E_{\psi_0}:=\E_{\delta_{\psi_0}}$ for a given (fixed) function $\psi_0\in C_+(\T)$.
This is customary in Markov process theory, and
we will use both notations without further explanation.

For every  $\Phi\in C_b(C_+(\T))$,  define
\[
	(P_t \Phi)(\psi_0) := \E_{\psi_0} [ \Phi(\psi(t)) ]
	\qquad\text{for every $t\ge0$ and $\psi_0\in C_+(\T)$}.
\]

\begin{proposition} \label{pr:feller}
	The $C_+(\T)$-valued stochastic process $\{\psi(t)\}_{t\ge0}$ is Feller;
	that is, $\{P_t\}_{t\ge0}$ is a continuous semigroup of positive linear operators from
	$C_b(C_+(\T))$ to $C_b(C_+(\T))$.
\end{proposition}

\begin{proof}
	The proof is similar to the derivation of Proposition 5.6 by \cite{Cerrai03}.
	Throughout, we choose and fix some $\Phi\in C_b(C_+(\T))$.
	
	Recall the approximations $\{\psi_N\}_{N=1}^\infty$ of $\psi$ from the proof
	of Theorem \ref{th:exist-unique}. A key feature of every $\psi_N$ is that
	it solves a Walsh-type SPDE of the form \eqref{eq:RD-N} with Lipschitz-continuous coefficients
	\citep[see][]{Dalang1999,wal86}.
	In particular, $\{\psi_N(t)\}_{t\ge0}$ is a
	Feller process for every $N\ge1$ (see, for example the method of
	\cite{NualartPardoux99}). 
	
	Let $\P^N_\mu,\E^N_\mu, P_t^N,\cdots$ 
	denote the same quantities as $\P_\mu,\E_\mu,P_t,\cdots$,
	except with the random field $\psi$ replaced everywhere by
	the random field $\psi_N$. It follows from the
	dominated convergence theorem that 
	\[
		P_t\Phi=\lim_{N\to\infty} P_t^N\Phi
		\qquad\text{for every $t\ge0$.}
	\]
	The Feller property of $\psi_N$ implies, among other things, that 
	\[
		P_{t+s}^N\Phi= P_t^N (P_s^N \Phi)\qquad\text{for all $t,s\ge0$.}
	\]
	Let $N\to\infty$ in order to deduce the semigroup property
	of $\{P_t\}_{t\ge0}$ from the above and the positivity of the operator
	$P_t$ for every $t\ge0$.
	
	We next turn to the more interesting ``Feller property.'' Note that 
	$\psi=\{\psi(t)\}_{t\ge0}$ is a Feller process
	if: (1) $P_t\Phi\in C_b(C_+(\T))$ for every $t>0$ [$P_0$ is manifestly the identity operator]; and 
	(2) the Markovian semigroup $\{P_t\}_{t\ge0}$ 
	is stochastically continuous.
	
	We prove (1) and (2) separately and in this order.
	
	For every $\varphi_0\in C_+(\T)$,
	let $\varphi$ denote the unique solution to \eqref{eq:RD},
	using the same underlying white noise $\dot{W}$, subject
	to the fixed  initial data $\varphi_0$. [The notation is consistent
	with our choice of $(\psi_0\,,\psi)$].
	As first step of the proof, we are going to prove that
	\begin{equation}\label{eq:continuity} 
		\lim_{\varphi_0\to\psi_0}
		\E \left( \| \psi(t)  - \varphi(t) \|_{C(\T)} \right)=0.
	\end{equation}
	This and the bounded convergence theorem together imply that
	$P_t\Phi\in C_b(C_+(\T))$ for every $t\ge0$, which proves (1).
	
	Next, let us recall that $\varphi_N$ denotes the solution to \eqref{eq:RD-N} for every $N\ge1$, started at
	a given $\varphi_0$, and recall that
	\[
		T_N(\varphi_0) := \inf\left\{ t\ge0:\, \|\varphi(t)\|_{C(\T)} \ge N\right\},
	\]
	with $\inf\varnothing:=\infty$.
	[Thus, the stopping time $T_N$ of the proof of Theorem \ref{th:exist-unique} 
	can now also be written as $T_N(\psi_0)$.]
	Then, $\lim_{N\to\infty}T_N(\varphi_0)=\infty$ a.s., thanks to 
	Proposition \ref{pr:tail:T_N}.
	Because $\varphi(t)=\varphi_N(t)$ for all $t\in(0\,,T_N(\varphi_0)]$
	a.s.,
	\begin{align*}
		\E \left( \| \psi(t)  - \varphi(t) \|_{C(\T)} \right) 
			&\le  \E \left( \| \psi(t)  - \psi_N(t) \|_{C(\T)} \right) 
			+ \E \left( \| \varphi(t)  - \varphi_N(t) \|_{C(\T)} \right) 
			+ \E \left( \| \psi_N(t)  - \varphi_N(t) \|_{C(\T)} \right)\\
		&\le K_t\P\{T_N(\varphi_0)\wedge T_N(\psi_0) \le t\}
			+ \E \left( \| \psi_N(t)  - \varphi_N(t) \|_{C(\T)} \right),
	\end{align*}
	where
	\[
		K_t := \E\left( \| \psi(t)\|_{C(\T)}\right) + \E\left( \| \varphi(t)\|_{C(\T)}\right)+
		\sup_{N\ge1}\E\left( \| \psi_N(t)\|_{C(\T)}\right) +
		\sup_{N\ge1}\E\left( \| \varphi_N(t)\|_{C(\T)}\right).
	\]
	On one hand, Proposition \ref{pr:tight} and its proof together show that $K_t<\infty$;
	in fact, the proof shows  that $\sup_{t\ge \tau}K_t<\infty$ for every $\tau>0$.
	On the other hand, since the drift $[V_N]$ and diffusion $[\sigma]$
	terms in \eqref{eq:RD-N} are both Lipschitz continuous, a standard 
	regularity estimate such as the one used in the proof of Proposition \ref{pr:tight} 
	implies that for every integer $N\ge1$ and for all $t>0$
	there exists a number $K_{N,t}>0$ -- independent of $(\varphi_0\,,\psi_0)$ -- such that
	\[
		\E \left( \| \psi_N(t)  - \varphi_N(t) \|_{C(\T)} \right)
		\leq K_{N,t} \|\psi_0 - \varphi_0\|_{C(\T)}.
	\]
	Thus, we find that 
	\begin{align*}
		\E \left( \| \psi(t)  - \varphi(t) \|_{C(\T)} \right) 
		\le K_t\P\{T_N(\varphi_0)\wedge T_N(\psi_0)\le t\}
		+ K_{N,t} \|\psi_0 - \varphi_0\|_{C(\T)}.
	\end{align*}
	We summarize by emphasizing that the preceding holds for every $t\ge 0$,
	$\varphi_0,\psi_0\in C_+(\T)$, and $N\in\N$, and that $K_N$ and $K_{N,t}$
	do not depend on the choice of $(\varphi_0\,,\psi_0)$.
	
	Now let us choose and fix an arbitrary $\varepsilon>0$ and $\varphi_0\in C_+(\T)$
	such that $\|\varphi_0-\psi_0\|_{C(\T)}\le\varepsilon$. 
	Since $T_N(\varphi_0)\wedge T_N(\psi_0)\to\infty$ a.s.\ as $N\to\infty$, 
	we now choose and fix $N\in\N$ large enough to ensure that
	\[
		\P\{T_N(\varphi_0)\wedge T_N(\psi_0)\le t\}\le\varepsilon,
	\]
	whence
	\[
		\E\left( \| \psi(t)  - \varphi(t) \|_{C(\T)} \right) \le  (K_t+K_{N,t})\varepsilon.
	\]
	Since $\varepsilon$ is arbitrary, this proves \eqref{eq:continuity},
	whence also that $P_t\Phi\in C_b(C_+(\T))$ for every $t\ge0$.
	
	Finally, we verify that the Markovian semigroup $\{P_t\}_{t\ge0}$ 
	is stochastically continuous. Choose and fix some 
	$\Phi\in C_b(C(\T))$. Then, for all $\psi_0\in C_+(\T)$ and $s,t\ge0$,
	\[
		\left| (P_{t+s}\Phi)(\psi_0) - (P_s\Phi)(\psi_0)\right|
		\le \E_{\psi_0}\left( \left| \Phi(\psi_N(t+s)) - \Phi(\psi_N(s))  \right| 
		\right) + 2\|\Phi\|_{C(C(\T))}\P\left\{ T_N(\psi_0)\le t+s\right\}.
	\]
	Because the coefficients $V_N$ and $\sigma$ in the SPDE \eqref{eq:RD-N} are Lipschitz
	continuous, the random field $\psi_N$ is uniformly continuous in $L^1(\Omega)$;
	see \cite{Dalang1999}. Thus, the bounded convergence theorem implies that
	\[
		\lim_{t\downarrow0}\left| (P_{t+s}\Phi)(\psi_0) - (P_s\Phi)(\psi_0)\right|
		\le 2\|\Phi\|_{C(C(\T))}\lim_{N\to\infty}\P\left\{ T_N(\psi_0)\le s\right\}=0.
	\]
	This verifies the desired
	stochastic continuity of $t\mapsto P_t$, and completes the proof.
\end{proof}


\subsection{The Krylov-Bogoliubov argument}
\label{subsec:KB}

Define for every number $t\ge0$ and
all Borel sets $\Gamma\subset C_+(\T)$ a probability measure $\nu P_t$ as follows:
\[
	(\nu P_t)(\Gamma) := \P_\nu\{\psi(t)\in\Gamma\} = \int_{C_+(\T)}\nu(\d\psi_0)\ (P_t\bm{1}_\Gamma)(\psi_0),
\]
for every Borel regular probability measure $\nu$ on $C_+(\T)$.
Thus, $(t\,,\nu)\mapsto \nu P_t$ defines
the transition functions of the $C_+(\T)$-valued Markov process $\psi:=\{\psi(t)\}_{t\ge0}$
which solves uniquely the stochastic PDE \eqref{eq:RD}.
Recall that $\nu$ is an \emph{invariant measure} for $\psi$ [equivalently,
for \eqref{eq:RD}] if
\[
	\nu P_t=\nu\qquad\text{for every $t\ge0$.}
\]
Recall also that  $\delta_{\mathbb{0}}$ is always an invariant measure for \eqref{eq:RD},
and of course is concentrated on the trivial solution,
$\psi(t)=\mathbb{0}$ for all $t\ge0$.  

Our next effort is toward proving that if 
$\psi_0\neq\mathbb{0}$, ${\rm L}_\sigma>0$ [see \eqref{LL}], and $\lambda$ is sufficiently small,
then there also exists a non-trivial invariant measure  $\mu_+\perp\delta_{\mathbb{0}}$.   
The standard way to do this sort of thing is to appeal to the Krylov-Bogoliubov argument 
\citep[see][Corollary  3.1.2]{DZ96}, which we shall recall. But first, let us state and prove a simple consequence
of Theorem \ref{th:exist-unique}.

\begin{lemma}\label{lem:non-trivial-inv-meas}
	Let $\nu$ be any probability measure on $C_+(\T)$
	that is invariant  for \eqref{eq:RD}. Then,
	\[
		\nu\left\{\omega\in C_+(\T):\, \inf_{x\in\T}\omega(x)=0<\sup_{x\in\T}\omega(x)
		\right\}=0.
	\]
\end{lemma}

Among other things, Lemma \ref{lem:non-trivial-inv-meas} tells us that if $\nu$ is an invariant
measure for \eqref{eq:RD} that satisfies  $\nu\{\mathbb{0}\}=0$, 
then 
\[
	\nu\left\{\omega\in C_+(\T):\,\inf_{x\in\T}\omega(x)>0\right\}=1.
\]

\begin{proof}[Proof of Lemma \ref{lem:non-trivial-inv-meas}]
	We can always decompose $\nu$ as
	\[
		\nu=\eta\delta_{\mathbb{0}}+(1-\eta)\tilde{\nu},
	\]
	where  $\tilde{\nu}$ is a probability measure on $C_+(\T)$ such that $\tilde{\nu}\{\mathbb{0}\}=0$
	and $\eta\in[0\,,1]$. Therefore, it suffices to prove the lemma with $\nu$ replaced by $\tilde{\nu}$.
	Alternatively, we can relabel $[\tilde{\nu}\to\nu]$ to see that we may -- and will --
	assume without loss of generality that $\nu\{\mathbb{0}\}=0$.
	
	If $\psi(0)=\omega\in C_+(\T)\setminus\{\mathbb{0}\}$, then 
	Theorem \ref{th:exist-unique}  ensures that
	\[
		\P_{\omega}\left\{ \psi(t\,,x)>0\text{ for all $t>0$ and $x\in\T$}\right\}=1,
	\]
	whence
	\[
		\P_\omega\left\{\inf_{x\in\T}\psi(t\,,x)>0\right\}=1\quad\text{for all $t>0$
		and $\omega\in C_+(\T)\setminus\{0\}$}.
	\]
	[Note the exchange of $\P_\omega$ with the ``for all $t>0$'' quantifier.]
	Integrate over all such $\omega$ $[\d\nu]$ and consider $t=1$ in order to see that
	\[
		\P_\nu\left\{ \inf_{x\in\T}\psi(1\,,x)>0 \right\}=1.
	\]
	This proves the lemma, since the $\P_\nu$-law of $\psi(1)$ is $\nu$.
\end{proof}

When $\psi_0\in C_+(\T)\setminus\{\mathbb{0}\}$ and ${\rm L}_\sigma>0$,
Theorem \ref{th:exist-unique} and Proposition \ref{pr:feller} together imply
that $\psi$ a Feller process, taking values in the cone
$C_{>0}(\T)$ of all $f\in C(\T)$ such that $f(x)>0$ for all $x\in\T$,
endowed with relative topology. 

We can now recall the following specialization of the Krylov-Bogoliubov theorem
\cite[Corollary  3.1.2]{DZ96}.

\begin{lemma}[A Krylov--Bogoliubov theorem]\label{lem:DZ-tight}
	Suppose there exists a probability measure $\nu$ on $C_{>0}(\T)$
	such that the probability measures
	\[
		\left\{ \frac{1}{T}\int_0^{T} (\nu P_s) \,\d s\right\}_{T>0}
	\]
	has a tight infinite subsequence in $C_{>0}(\T)$. Then, $\psi$ has an invariant measure
	which is a probability measure on $C_{>0}(\T)$.
\end{lemma}

Because $\mathbb{0}\not\in C_{>0}(\T)$, the preceding invariant measure
cannot be $\delta_{\mathbb{0}}$. This is the desired non-triviality result.
Thus, our next challenge is to verify the conditions of Lemma \ref{lem:DZ-tight}.
We apply that lemma with $\nu := \delta_{\mathbb{1}}$. In order to discuss
tightness, we need a family of compact subsets of $C_{>0}(\T)$, which we build next.

Choose and fix  some $\alpha\in(0\,,1/2)$, 
and define for all $\varepsilon,\delta\in(0\,,1)$,
\[
	A(\varepsilon) := \left\{ f\in C^\alpha(\T):\, \inf_{x\in\T} f(x)\ge\varepsilon\right\}
	\quad\text{and}\quad
	B(\delta) := \left\{ f\in C^\alpha(\T):\, \| f\|_{C^\alpha(\T)}\le 1/\delta\right\}.
\]
According to the Arz\`ela-Ascoli theorem, every $B(\delta)$ is compact
[in the topology of $C(\T)$]. Since every $A(\varepsilon)$ is closed, 
$A(\varepsilon)\cap B(\delta)\subset C(\T)$
is compact, and of course also in $C_{>0}(\T)$.
We propose to prove that if $\psi_0=\mathbb{1}$, ${\rm L}_\sigma>0$,
and $\lambda$ is sufficiently small, then
\begin{equation}\label{claim:DZ-tight}
	\adjustlimits\lim_{\varepsilon,\delta\downarrow0}\limsup_{T\to\infty}
	\frac 1T\, \E_{\mathbb{1}}\left[\int_{0}^T
	\bm{1}_{\{\psi(t)\not\in A(\varepsilon)\cap B(\delta)\}}
	\,\d t\right]=0.
\end{equation}
Given that \eqref{claim:DZ-tight} is true, Lemma \ref{lem:DZ-tight} [with $\nu=\delta_{\mathbb{1}}$] 
readily implies the following.

\begin{proposition}\label{pr:inv:meas:exists}
	If $\psi_0=\mathbb{1}$, ${\rm L}_\sigma>0$,
	and $\lambda$ is sufficiently small, then $\psi$ has an invariant
	measure $\mu_+$ on $C_{>0}(\T)$.
\end{proposition}

Proposition \ref{pr:tight} and Chebyshev's inequality together imply that
\[
	\limsup_{T\to\infty}
	\frac 1T\, \E_{\mathbb{1}}\left[\int_{0}^T
	\bm{1}_{\{\psi(t)\not\in B(\delta)\}}\,\d t\right] \le\sup_{t\ge1}\P_{\mathbb{1}}\left\{ 
	\|\psi(t)\|_{C^\alpha(\T)} > 1/\delta\right\} = o(1)\qquad\text{as $\delta\downarrow0$}.
\]
Therefore, \eqref{claim:DZ-tight} -- whence also Proposition \ref{pr:inv:meas:exists} --
follows as soon as we establish the following: 
\[
	\adjustlimits
	\lim_{\varepsilon\downarrow0}\limsup_{T\to\infty}
	\frac 1T\, \E_{\mathbb{1}}\left[\int_{0}^T
	\bm{1}_{\{\inf_{x\in\T}\psi(t,x) < \varepsilon\}}\,\d t\right] =0.
\]
Recall that Fatou's lemma can be recast as follows: If $\{\xi_T\}_{T>1}$ is a process 
that take values in $[0\,,1]$, then
\[
	\limsup_{T\to\infty}\E(\xi_T)\le\E\left(\limsup_{T\to\infty}\xi_T\right).
\]
We apply this with $\xi_T:= T^{-1}\int_0^T\bm{1}_{\{\inf_\T\psi(t)<\varepsilon\}}\,\d t$,
and then apply the monotone convergence theorem in order to be able to deduce the above,
whence also Proposition \ref{pr:inv:meas:exists}, from the following assertion:
\begin{equation}\label{eq:condition-3}
	\P_{\mathbb{1}}\left\{ \adjustlimits
	\lim_{\varepsilon\downarrow0}\limsup_{T\to\infty}
	\frac 1T \int_{0}^T
	\bm{1}_{\{\inf_{x\in\T}\psi(t,x) < \varepsilon\}}\,\d t =0 \right\}=1.
\end{equation}
We prove \eqref{eq:condition-3} in the next few sections. This will complete our proof
of Proposition \ref{pr:inv:meas:exists}. 




\section{A Random Walk Argument}

For every continuous space-time process $h=\{h(t\,,x)\}_{t\ge0,x\in\T}$ define
\[
	L_t(h)  :=\inf_{x\in\T}h(t\,,x)\quad\text{and}\quad
	U_t(h) :=\sup_{x\in\T}h(t\,,x)\qquad\text{for every $t\ge0$}.
\]
If, in particular, $h\ge0$ then $U_t(h)=\|h(t)\|_{C(\T)}$ for all $t\ge0$.
Our goal is to construct comparison processes which give a series of lower bounds 
for $L_t(\psi)$, all the time controlling $U_t(\psi)$, where 
$\psi$ denote a solution to \eqref{eq:RD} with the initial profile 
$\psi_0=\mathbb{1}$.   
Our argument repeatedly uses the comparison 
principle for SPDEs in the form of Lemma \ref{lem:comparison}. 

\subsection{An associated chain}
By the definition of $F$, we can choose and fix a strictly negative integer $M\in-\N$, sufficiently negative
to ensure that there exists $c\in(0\,,1)$ such that
$(1-c)v\le V(v)=v-F(v)\le v$ for all $v\in (0\,,2^{M+1})$; see Lemma \ref{lem:F}.
Of course the second inequality holds for all $v>0$.  For simplicity, we assume that $c=1/2$. 
It will not be difficult to study the general case $c\in(0\,,1)$ after we adjust the
argument to follow. In other words, we will proceed with the assumption that
\begin{equation}\label{12v<V<v}
	\tfrac12v \leq V(v) = v-F(v) \leq v
	\qquad\text{for all $v\in(0\,,2^{M+1})$}.
\end{equation}
From now on, the symbol ``$M$'' will be used only for this purpose.

As was mentioned in the preamble to this section,
throughout we let $\psi$ denote a solution to \eqref{eq:RD} with the initial profile 
$\psi_0=\mathbb{1}$.  Now we set up our random walk comparison.

We will define $\mathscr{F}$-stopping 
times $0=\tau_0<\tau_1<\cdots$ and comparison processes 
$v_0,v_1,\ldots$ such that
\begin{equation} \label{eq:comparison-reqirement}
	\psi(t)\geq v_n(t)\quad\text{for $t\in\left[\tau_n\,,\tau_{n+1}\right)$,
	everywhere on $\T$}.
\end{equation}
For this reason, we will  define $v_n(t)$ only for $t\geq\tau_n$. 

To start the process, let us define 
\begin{equation*}
	\tau_0  := 0 \quad\text{and}\quad
	v_0(0\,,x) := 2^{M-2}\text{ for all $x\in\T$},
\end{equation*}

Now we proceed inductively on $n$.  Suppose that we have 
defined $\tau_n$ and $v_n(\tau_n)=\{v_n(\tau_n\,,x)\}_{x\in\T}$. 
Let $\{\theta_t\}_{t\ge0}$ denote the standard shift operator on our probability space.
Informally speaking, this means that
$\theta_t\dot{W}(s\,,x) := \dot{W}(s+t\,,x)$ for all $s,t\ge0$ and $x\in\T$.
More precisely, $\theta_t$ is defined via
\[
	\int_{\R_+\times\T}  h(s\,,x)\,\theta_t W(\d s\,\d x) :=
	\int_{(t,\infty)\times\T} h(s-t\,,x)\, W(\d s\,\d x)
	\quad\text{for all $t\ge0$ and $h\in L^2(\R_+\times\T)$}.
\]
With the above shifts in mind, we define $w_n$ for $n\geq 0$ as the unique adapted and continuous
solution to the following SPDE:
\begin{equation*} 
\left[\begin{split}
	&\partial_tw_{n}(t\,,x)=\partial_x^2w_{n}(t\,,x) + 
		\tfrac12 L_{\tau_n}(v_n)
  		+\lambda\sigma(w_{n}(t\,,x))\theta_{\tau_{n}}\dot{W}(t\,,x)
		\quad\text{for $(t\,,x)\in(0\,,\infty)\times\T$},\\
	&\text{subject to }w_{n}(0\,,x)=L_{\tau_{n}}(v_n)\hskip2.25in\text{for all $x\in\T$}.
\end{split}\right. 
\end{equation*}
Define $\tau_{n+1}$
to be the smallest $t+\tau_n>\tau_n$ that at least one of the following occurs:
\begin{compactenum}
\item $L_t(w_{n})=2L_{\tau_{n}}(v_{n})$
\item $L_t(w_{n})=\frac{1}{2}L_{\tau_{n}}(v_{n})$
\item $U_t(w_{n})=4L_{\tau_{n}}(v_{n})$.  
\end{compactenum}
If such a $t$ does not exist, then $\tau_{n+1}:=\infty$.
If such a $t$ does exist, then we let
\begin{equation*}
	v_{n}(\tau_n+t\,,x) := w_{n}(t\,,x)\qquad\text{for all $t\geq 0$ and $x\in\T$}.
\end{equation*}
In case 1, that is a.s.\ on $\{L_{\tau_{n+1}}(v_{n})=2L_{\tau_{n}}(v_{n})\}\cup\{
\tau_{n+1}<\infty\}$, we let
\begin{equation*}
	v_{n+1}(\tau_{n+1},x) :=
	\begin{cases}
		2L_{\tau_{n}}(v_{n})  &\text{if $L_{\tau_n} (v_n) \leq 2^{M-2}$}, \\
		2^{M-2} & \text{if $L_{\tau_n} (v_n) \geq  2^{M-1}$,}
	\end{cases}
\end{equation*}
for every $x\in\T$.
And in cases 2 and 3, that is a.s.\ on
\[
	\left\{L_{\tau_{n+1}}(v_{n})=\tfrac{1}{2}L_{\tau_{n}}(v_{n})\right\}\cup
	\left\{U_t(v_{n})=4L_{\tau_{n}}(v_{n})\right\}\cup
	\left\{\tau_{n+1}<\infty\right\},
\]
we define
\begin{equation*}
	v_{n+1}(\tau_{n+1},x) := \tfrac{1}{2}L_{\tau_{n}}(v_{n})\qquad
	\text{for every $x\in\T$}.
\end{equation*}
If $\tau_{n+1}<\infty$ a.s., then by
\eqref{12v<V<v}, the Markov property, and the comparison theorem for SPDEs,
\eqref{eq:comparison-reqirement} holds
almost surely.  This finishes our inductive construction, provided that we can prove the
following. 

\begin{lemma}
	If ${\rm L}_\sigma>0$, then $\P_{\mathbb{1}}\{\tau_{n+1}<\infty\}=1$ for all $n\in\Z_+$.
\end{lemma}

\begin{proof}
	Since $\tau_0=0$, we may [and will] assume that we have proved that $\P_{\mathbb{1}}
	\{\tau_n<\infty\}=1$ for some $n\in\Z_+$,
	and proceed to prove that $\P_{\mathbb{1}}\{\tau_{n+1}<\infty\}=1$. With this in mind, choose and fix
	some $n\in\Z_+$ and suppose that
	$\tau_n<\infty$ a.s. Then, for every $t>0$ and $x\in\T$,
	\[
		w_{n}(t\,,x) = \left( 1 + \frac t2\right) L_{\tau_n}(v_n) + \lambda
		\int_{(0,t)\times\T} p_{t-s}(x\,,y)\sigma(w_{n}(s\,,y))\,\theta_{\tau_{n}} W(\d s\,\d y)
		\qquad\text{a.s.},
	\]
	and 
	\[
		\tau_{n+1} \le \hat{\tau}_{n+1} := 
		\inf\left\{ t>0:\,  L_t(w_{n}) \text{ or }U_t(w_{n})\notin\left[\tfrac12 L_{\tau_n}(v_n)\,,
		4L_{\tau_n}(v_n)\right]\right\},
	\]
	where $\inf\varnothing=\infty$. Note that $\hat{\tau}_{n+1}$
	is a stopping time with respect to the filtration $\mathscr{F}$.
	It remains to prove that $\hat{\tau}_{n+1}<\infty$
	$\P_{\mathbb{1}}$-a.s. 
	
	Define
	\[
		\hat{w}_{n}(t) := \int_{\T}w_{n}(t\,,x)\,\d x
		\qquad\text{for all $t\ge0$.}
	\]
	A stochastic Fubini argument yields
	\begin{equation}\label{bar(w)}
		\hat{w}_{n}(t) = (2+t)L_{\tau_n}(v_n) + \lambda \mathcal{M}_t
		\qquad\text{for all $t\ge0$},
	\end{equation}
	where 
	\[
		\mathcal{M}_t := \int_{(0,t)\times\T}\sigma(w_{n}(s\,,y))\,\theta_{\tau_{n}} W(\d s\,\d y)
		\qquad\text{for all $t\ge0$}
	\]
	defines a continuous $L^2(\P)$-martingale.
	Since $\sigma$ is Lipschitz and $\sigma(0)=0$, it follows that
	$|\sigma(z)|\le \lip_\sigma|z|$ for all $z\in\R$. Therefore,
	the quadratic variation of $\mathcal{M}$ is given by
	\[
		\< \mathcal{M}\>_t = \int_0^t\d s\int_{\T}\d y\ \sigma^2(w_{n}(s\,,y))
		\le \lip_\sigma^2\int_0^t\d s\int_{\T}\d y\ \left| w_{n}(s\,,y) \right|^2
		\quad\text{for all $t\ge0$.}
	\]
	It follows from this inequality that
	\[
		\< \mathcal{M}\>_t \le 32[\lip_\sigma  L_{\tau_n}(v_n) ]^2t
		\qquad\text{for all $t\ge0$ a.s.\ on $\{\hat{\tau}_{n+1}=\infty\}$.}
	\]
	The law of large numbers for
	continuous $L^2(\P)$-martingales then implies that 
	\begin{equation}\label{<M>:to:0}
		\limsup_{t\to\infty}
		\left| \frac{\mathcal{M}_t}{t} \right| \le 32 \left[ \lip_\sigma L_{\tau_n}(v_n) \right]^2
		\lim_{t\to\infty}\left| \frac{\mathcal{M}_t}{\< \mathcal{M}\>_t}\right| =0\qquad\text{a.s.\
		on $\{\hat{\tau}_{n+1}=\infty\}\cap\mathcal{E}$},
	\end{equation}
	for the event $\mathcal{E} := \{\lim_{t\to\infty}\< \mathcal{M}\>_t=\infty\}$.
	Since ${\rm L}_\sigma>0$, the inequality $|\sigma(z)|\ge{\rm L}_\sigma|z|$ -- valid
	for all $z\in\R$ -- has content, and implies that
	$\< \mathcal{M}\>_t\ge \tfrac12{\rm L}_\sigma^2 L_{\tau_n}^2(v_n)t$ for all $t\ge0$
	a.s.\ on $\{\hat{\tau}_{n+1}=\infty\}$. In particular, 
	\[
		\P_{\mathbb{1}}\left(\{\hat{\tau}_{n+1}=\infty\}\cap\mathcal{E} \right)
		=\P_{\mathbb{1}}\{\hat{\tau}_{n+1}=\infty\}.
	\]
	This fact, \eqref{<M>:to:0}, and \eqref{bar(w)} together imply that
	$\lim_{t\to\infty}t^{-1}\hat{w}_{n}(t) = L_{\tau_n}(v_n)>0$
	a.s.\ on $\{\hat{\tau}_{n+1}=\infty\}$, whence
	\[
		\P_{\mathbb{1}}\left\{ \sup_{t\ge0} \hat{w}_{n}(t)  =\infty
		~,~ \hat{\tau}_{n+1}=\infty\right\}=\P_{\mathbb{1}}\left\{\hat{\tau}_{n+1}=\infty\right\}.
	\]
	Since
	\[
		\sup_{t\ge0} \hat{w}_{n}(t) \le 4L_{\tau_n}(v_n)<\infty\quad\text{a.s.\ on}\quad
		\{\hat{\tau}_{n+1}=\infty\}, 
	\]
	it follows that $\P_{\mathbb{1}}\{\hat{\tau}_{n+1}=\infty\}=0$,
	as was claimed.
\end{proof}

Now we define an embedded ``reflected chain'' $X$, along with 
the length of time $\tau_n$ for step $n$:
\begin{align*}
	X_n&=\log_2L_{\tau_n}(v_{n}),  &\text{for }n\geq0,\\
	\ell_n&=\tau_{n}-\tau_{n-1},   &\text{for }n\geq1.
\end{align*}
Let us pause to explain why we refer to $X$ as a ``reflected chain.''
Firstly, the strong Markov property of every infinite-dimensional diffusion
$\{w_n(t)\}_{t\ge0}$ implies that for every $n\in\Z_+$, and given the value of $X_n$, the random variable
$X_{n+1}$ is conditionally independent of $\mathscr{F}_{\tau_n}$, and clearly $X_0=M-2$; thus,
$X$ is a time-inhomogenous Markov chain that starts at $M-2$. Secondly, the definition of
$X$ implies immediately that $X$ moves in three distinct ways:
$X_{n+1}-X_n=\pm 1$ for all $n\ge0$, and in all cases,
except that $X_{n+1}-X_n= M-2 - X_n\le  -1$ a.s.\
on $\{X_n\ge M-1\}$.  We find that $X_0=M-2$ and from time one onward,
$X$ moves in increments of $\pm1$ and is reflected at $M-1$ by an increment of $-1$ to stay in 
$\Z\cap(-\infty\,,M-1]$ henceforth.
The latter exception ensures that $X_n\le M-1$ for all $n\ge1$, 
and explains the use of ``reflection at $M-1$.'' Finally, 
we will soon demonstrate that, if the parameter $\lambda$ in \eqref{eq:RD} is small enough, then
\[
	\P(X_{n+1}-X_n=+1\mid X_n)\ge \tfrac23\quad
	\text{for all $n\ge1$, $\P_{\mathbb{1}}$-a.s.\ on the event $\{X_n< M-1\}$.}
\]
This will imply that $X$ moves upward at least as fast as a random walk with upward drift,
when it is not being reflected. We now return to the main part of the discussion.

In order to study the walk $X$ we make some definitions.  
For simplicity, define for all $n\in\N$, $t\ge0$, and $x\in\T$,
\begin{equation*}
	\mathcal{I}_n(t\,,x) := \lambda\int_{(0,t)\times\T} p_{t-s}(x-y)
	\sigma(v_n(\tau_n+s\,,y))\,\theta_{\tau_{n}}W(\d y\,\d s).
\end{equation*}
Note that, in the definition of $\mathcal{I}_n(t\,,x)$, the parameter $t$ refers to the 
time beyond $\tau_n$. Thanks to the mild formulation of $w_n$, and hence also
$v_n$, the following is valid a.s.\ for every $t\in[0\,,\ell_{n+1}]$:
\begin{equation} \label{eq:vn-equals}
\begin{split}	
w_n(t\,,x) = v_{n}(\tau_n+t\,,x) &= 2^{X_n}+t2^{X_n-1}+\mathcal{I}_{n}(t\,,x) \\
&= \left(1+\frac{t}{2}\right)L_{\tau_n}(v_n)+\mathcal{I}_{n}(t\,,x).
\end{split}
\end{equation}
Next we estimate $\mathcal{I}_n$. For possibly random numbers $T\geq0$
and $\delta>0$, consider the event
\begin{equation*}
	A_n(T,\lambda\,,\delta)
	:= \left\{\adjustlimits
	\sup_{0\leq t\leq T}\sup_{x\in\T}
	 |\mathcal{I}_n(t\,,x)|\leq\delta\right\}.
\end{equation*}
By \eqref{LL} and by the definition of $\tau_n$ and $\tau_{n+1}$,
\begin{equation} \label{eq:sup-sigma}
	|\sigma(v_n(t\,,x))|\leq 4\text{Lip}_{\sigma}L_{\tau_n}(v_n)
	\quad\text{for every $t\in[\tau_n\,,\tau_{n+1}]$},
\end{equation}
almost surely.
We define a truncated version of $\sigma$ and corresponding versions
of $\mathcal{I}$ and $A$ as follows:
\begin{align}
\label{eq:def-sigma-n}
	\sigma_n(t\,,x) &= \sigma(v_{n}(\tau_n+t\,,x))\bm{1}_{(0,\ell_{n+1})}(t),\\\label{eq:def-I-tr}
	\mathcal{I}^{\text{tr}}_n(t\,,x) & =\lambda
		\int_{(0,t)\times\T}p_{t-s}(x-y)\sigma_n(s\,,y)\,
		\theta_{\tau_{n}}W(\d y\,\d s),\\
	A_n^{\text{tr}}(T,\lambda\,,\delta)
		&=\left\{\adjustlimits
		\sup_{t\in[0,T]}\sup_{x\in\T}
		|\mathcal{I}_n^{\text{tr}}(t\,,x)|\leq\delta\right\}. \nonumber
\end{align}
Since $\sigma_n=0$ beyond time $\ell_{n+1}$, elementary properties of the Walsh stochastic integral
ensure that  a.s.\ on the event $\{\ell_{n+1}<T\}$,
\begin{equation*}
	\sup_{0\leq t\leq T}\sup_{x\in\T}
	|\mathcal{I}_n^{\text{tr}}(t\,,x)|\leq\delta
	\quad\Leftrightarrow\quad
	\sup_{0\leq t\leq\ell_{n+1}}\sup_{x\in\T}
	|\mathcal{I}_n(t\,,x)|\leq\delta.
\end{equation*}
Because the underlying probability space is assumed to be complete, it follows that
\begin{equation*}
	A_n^{\text{tr}}(T,\lambda\,,\delta)\cap\{\ell_{n+1}<T\}
	=A_n(\ell_{n+1}\,,\lambda\,,\delta)\cap\{\ell_{n+1}<T\}\qquad\text{a.s.}
\end{equation*}

Next we give a probabilistic estimate for $\mathcal{I}_n^{\text{tr}}(t\,,x)$.
To keep the flow of the argument, we postpone the proof until the end of 
the paper.

\begin{lemma} \label{lem:large-dev}
	There exist numbers $C_0,C_1>0$ such that
	\begin{equation*}
		\P\left(\left.\adjustlimits
		\sup_{0\leq t\leq T}\sup_{x\in\T}
		|\mathcal{I}_n^{\text{tr}}(t\,,x)| > \rho \ \right|\, \mathscr{F}_{\tau_n}\right) \\
		\leq C_0\exp\left(-\frac{C_1\rho^2}
		{16T^{1/2}\lambda^2 \lip^2_\sigma L^2_{\tau_n}(v_n)}\right),
	\end{equation*}
	$\P_{\mathbb{1}}$-a.s.\ for every $\rho,T>0$ and $n\ge1$.
\end{lemma}

For convenience let us define $\mathscr{P}_n$ to be the conditional measure,
\begin{equation*}
	\mathscr{P}_n(\cdot)=\P(\,\cdot\mid\mathscr{F}_{\tau_n}).
\end{equation*}
We now select
\begin{equation}\label{eq:T_delta}
	\delta := \tfrac{1}{4}L_{\tau_n}(v_n).
\end{equation}
Note that $\delta$ is random, but $\mathscr{F}_{\tau_n}$-measurable.
Moreover, \eqref{eq:vn-equals} assures us that, as long as $T\ge3$,
the following holds a.s.\ on $A_n^{\text{tr}}(T,\lambda\,,\delta)\cap \{ \ell_{n+1} > T\}$:
\begin{equation*}
	v_{n}(\tau_n+T, x) \geq \tfrac{5}{2}L_{\tau_n}(v_n)-\delta 
	> 2L_{\tau_n}(v_n).
\end{equation*}
This proves that $A_n^{\text{tr}}(T,\lambda\,,\delta)=
A_n^{\text{tr}}(T,\lambda\,,\delta)\cap \{ \ell_{n+1} \leq T\}$ a.s.  for (say) $T=3$.
In addition,  \eqref{eq:vn-equals} ensures that
a.s.\ on $A_n^{\text{tr}}(T,\lambda\,,\delta)$,
\begin{equation*}
	\tfrac{1}{2}L_{\tau_n}(v_n) < L_{\tau_n}(v_n)-\delta \leq v_{n} (\tau_n+ t\,,x) 
	\leq \tfrac{5}{2}L_{\tau_n}(v_n) + \delta < 4L_{\tau_n}(v_n),
\end{equation*}
for all $t\in[0\,,\ell_{n+1}]$ and $x\in\T$ a.s.
That is, $\ell_{n+1}\le T$ a.s.\ 
on $A_n^{\text{tr}}(T,\lambda\,,\delta)$,  and we are in 
case 1 (i.e., $L_t(w_n)=2L_{\tau_n}(v_n)$). This is another way to say that the random walk $X$ moves up except possibly at the 
reflecting boundary.

We apply Lemma \ref{lem:large-dev},  conditionally on $\mathscr{F}_{\tau_n}$, in order to see that
\begin{equation} \label{eq:prob-1}
\begin{split}
	\mathscr{P}_n\left(A_n^{\text{tr}}(T,\lambda\,,\delta)^c\right) 
	&\leq C_0\exp\left(- \frac{C_1\delta^2}{T^{1/2}
	\lambda^2(4\text{Lip}_\sigma L_{\tau_n}(v_n))^2}\right)
	=C_0\exp\left(- \frac{C_1}{256\sqrt{T}\,
	 \text{Lip}_\sigma ^2\lambda^2}\right),
\end{split}
\end{equation}
$\P_{\mathbb{1}}$-a.s.\ uniformly for all $n\ge1$, $T>0$, and every initial choice of $\lambda>0$ in \eqref{eq:RD}.
We emphasize that the right-most quantity in \eqref{eq:prob-1} is non random.
In any case, we apply the above with $T=3$ in order to see that  there exists a
non-random number $\lambda_1=\lambda_1(C_1\,,\sigma)\in(0\,,1)$ such that
\begin{equation} \label{eq:walk-moves-up}
	\mathscr{P}_n\big(X_{n+1}=(X_n+1)\big)1_{\{X_n\leq M-2\}}
	\geq \mathscr{P}_n(A_n^{\text{tr}}(3\,,\lambda\,,\delta))1_{\{X_n\leq M-2\}}
	\geq \frac{2}{3}1_{\{X_n\leq M-2\}},
\end{equation} 
$\P_{\mathbb{1}}$-a.s.\ for every $\lambda\in(0\,,\lambda_1)$ and $n\ge1$.
This proves that the random walk $X$ has an upward drift,
thereby concluding our random walk argument.


\subsection{A reduction}

In this subsection we reduce our proof of the existence of
non-trivial invariant measures [Proposition \ref{pr:inv:meas:exists}]
to condition \eqref{eq:goal:KB} below. That condition will be verified
in the following subsection provided additionally that ${\rm L}_\sigma>0$
and $\lambda$ is sufficiently small. 

\begin{proposition}\label{pr:reduction}
	Suppose $\lambda\in(0\,,1)$ is small enough to ensure that \eqref{eq:walk-moves-up} holds.
	Then,
	Condition \eqref{eq:condition-3} --
	whence also Proposition \ref{pr:inv:meas:exists} --
	follows provided that
	\begin{equation}\label{eq:goal:KB}\adjustlimits
		\lim_{k\to\infty}\limsup_{n\to\infty}\frac1n\sum_{j=0}^{n-1}
		\bm{1}_{\{X_{j+1}<-k\}}=0\qquad\text{a.s.}
	\end{equation}
\end{proposition}

Proposition \ref{pr:reduction} hinges on two coupling lemmas that respectively bound 
$\{\ell_n\}_{n=1}^\infty$ from above and from below by ``better behaved'' sequences
of random variables. 

\begin{lemma}\label{lem:ell-underline}
	Suppose $\lambda\in(0\,,1)$ is small enough to ensure that \eqref{eq:walk-moves-up} holds.
	Then, there exists a  sequence $\{\underline{\ell}_n\}_{n=1}^\infty$ of  i.i.d.\
	random variables such that $\ell_n\geq\underline{\ell}_n\ge0$ a.s.\ 
	for all $n\ge1$, and $0<\E_{\mathbb{1}}[\underline{\ell}_1]\leq1$.
\end{lemma}

\begin{proof} 
	Choose and fix a small enough $\lambda$, as designated,
	and recall the $\mathscr{F}_{\tau_n}$-measurable
	random variable $\delta$ from \eqref{eq:T_delta} for every $n\in\Z_+$. For every $T>0$, let 
	\begin{equation*}
		R(T) := C_0\exp\left(- \frac{C_1}{256\sqrt{T}\,
	 	\text{Lip}_\sigma^2}\right)
	\end{equation*}
	denote the supremum of the
	right-hand side of \eqref{eq:prob-1} over all $\lambda\in(0\,,1)$. 
	Throughout, we choose and fix $T\in(0\,,1)$
	such that $R(T)<1$.
	
	Because $A_n^{\text{tr}}(T,\lambda\,,\delta)^c$ is independent of $\mathscr{F}_{\tau_n}$,
	we can enlarge the event 
	$A_n^{\text{tr}}(T,\lambda\,,\delta)^c$ to an event $A_n^{(1)}(T)^c$ whose 
	probability is exactly equal to $R(T)<1$, and such 
	that $A_n^{(1)}(T)$ is independent of $\mathscr{F}_{\tau_n}$.
	By virtue of construction,
	\begin{equation*}
		A_n^{\text{tr}}(T,\lambda\,,\delta) \supseteq A_n^{(1)}(T).
	\end{equation*}
	In addition, using \eqref{eq:vn-equals}, we see that  a.s.\ on 
	$A_{n}^{\text{tr}}(T,\lambda\,,\delta) \cap \{ \ell_{n+1} < T\}$, we have
	\begin{align*}
		\adjustlimits\sup_{t\in[0,\ell_{n+1}]}\sup_{x\in\T}w_{n}(t\,,x)&<2L_{\tau_n}(v_n),  \\
		\adjustlimits\inf_{t\in[0,\ell_{n+1}]}\inf_{x\in\T}w_{n}(t\,,x)&>\tfrac{1}{2}L_{\tau_n}(v_n).
	\end{align*}
	Thus,
	$A_{n}^{\text{tr}}(T,\lambda\,,\delta)=A_{n}^{\text{tr}}(T,\lambda\,,\delta) \cap \{ \ell_{n+1} \geq T\}$,
	and $\ell_{n+1}\ge T$
	a.s.\ on $A_{n}^{\text{tr}}(T,\lambda\,,\delta)$.  With this in mind, we can define
	\begin{equation*}
		\underline{\ell}_{n+1} := T\bm{1}_{A^{(1)}_{n}(T)}
		\leq T\bm{1}_{A^{\text{tr}}_{n}(T,\lambda,\delta)} \leq \ell_{n+1}\wedge T\le
		\ell_{n+1}\wedge 1.
	\end{equation*}
	This is the desired sequence, with the announced properties.
\end{proof}

Lemma \ref{lem:ell-underline} provides a suitable lower sequence for $\{\ell_n\}_{n=1}^\infty$.
The following matches that result with a corresponding upper sequence.

\begin{lemma} \label{lem:ell-bar}
	There exists a  sequence $\{\bar{\ell}_n\}_{n=1}^\infty$ of almost 
	surely nonnnegative i.i.d.\
	random variables such that with probability one for all $n\geq1$ we 
	have $\ell_n\leq\bar{\ell}_n$.  Finally,
	$0<\E_{\mathbb{1}}(|\bar{\ell}_1|^k)<\infty$ for every real number $k\ge 2$.
\end{lemma}

\begin{proof}
	By \eqref{eq:vn-equals}, if $\ell_{n+1}>T$, then
	\begin{equation*}
	2^{X_n}+T2^{X_n-1}+\mathcal{I}_n^{\text{tr}}(T,x) \leq 2^{X_n + 2}.
	\end{equation*}
	Thus, if $\ell_{n+1}>T$ and   $T$ is large enough, we have 
	\begin{equation*}
	\sup_{0\leq t\leq T}\sup_{x\in\T}|\mathcal{I}_n^{\text{tr}}(t\,,x)|
	   > 2^{X_n}+T2^{X_n-1}-2^{X_n+2}.
	\end{equation*}
	That is, $\mathcal{I}_n^{\text{tr}}(t\,,x)$ must counteract the drift 
	which is pushing $v_{n+1}$ out of the interval around 
	$2^{X_n}=L_{\tau_n}(v_n)$.  
	
	Choose $T_0$ so large that $T>T_0$ implies that
	\begin{equation*}
	2^{X_n}+T2^{X_n-1}-2^{X_n+2}>T2^{X_n-2}=\frac{TL_{\tau_n}(v_n)}{4}.
	\end{equation*}
	
	Then using Lemma \ref{lem:large-dev} we can choose $T_0$ so large that for 
	$T>T_0$,  the following holds $\P_{\mathbb{1}}$-a.s.:
	\begin{equation*}\begin{split}
		\P(\ell_{n+1}>T\mid\mathscr{F}_{\tau_n}) &\leq \P\left(\left.\sup_{0\leq t\leq T}\sup_{x\in\T}
			|\mathcal{I}_n^{\text{tr}}(t\,,x)| > \frac{TL_{\tau_n}(v_n)}{4}\ \right|\,\mathscr{F}_{\tau_n}\right) \\
		&\leq \left[ c_0\exp\left( -\frac{c_1 T^{3/2}}{\lambda^2 } \right) \right] \wedge 1;
	\end{split}
	\end{equation*}
	where $c_0,c_1$ are non-random, positive real numbers that do not depend on $(n\,,\lambda)$.
	By the strong Markov property, $\ell_{n+1}$ is independent of $\mathscr{F}_{\tau_n}$ though
	it is measurable with respect to $\mathscr{F}_{\tau_{n+1}}$ by virtue of construction. 
	Therefore, Lemma \ref{lem:coupling} ensures that
	we can find a random variable $\bar{\ell}_{n+1}$, independent 
	of $\mathscr{F}_{\tau_n}$, such that $\ell_{n+1}\leq\bar{\ell}_{n+1}$ and
	\[
		\P_{\mathbb{1}} \left\{\bar{\ell}_{n+1}>t\right\} 
		= c_0\exp\left( -\frac{c_1 t^{3/2}}{\lambda^2} \right) \wedge 1
		\qquad\text{for all $t\ge T_0$},
	\]
	and $\{\bar{\ell}_m\}_{m=1}^\infty$ is i.i.d.\ [$\P_{\mathbb{1}}$].
	Finally, we can study the moments of $\ell_1$ as follows:
	On one hand, Lemma \ref{lem:ell-underline} ensures that
	\[
		\|\bar{\ell}_1\|_2\ge
		\E_{\mathbb{1}}(\bar{\ell}_1)\ge\E_{\mathbb{1}}(\ell_1)\ge\E_{\mathbb{1}}(\underline{\ell}_1)>0.
	\]
	On the other hand,
	\[
		\E_{\mathbb{1}}(|\bar{\ell}_1|^k) =
		k\int_0^\infty t^{k-1}\P_{\mathbb{1}}\{\bar{\ell}_{n+1}>t\} \,\d t
		\le T_0^k + kc_0\int_{T_0}^\infty t^{k-1}\exp\left( -\frac{c_1 t^{3/2}}{\lambda^2} \right) \d t<\infty,
	\]
	for every $k\ge2$.
\end{proof}

We are ready to prove Proposition \ref{pr:reduction}.

\begin{proof}[Proof of Proposition \ref{pr:reduction}]
	Thanks to \eqref{eq:comparison-reqirement} we may observe that for every $n\in\N$
	and $\varepsilon\in(0\,,1/8)$,
	\begin{align*}
		\frac{1}{\tau_n}\int_0^{\tau_n} \bm{1}_{\{\inf_{x\in\T}\psi(t,x)<\varepsilon\}}\,\d t
		&=\frac{1}{\tau_n}\sum_{j=0}^{n-1}\int_{\tau_j}^{\tau_{j+1}} 
			\bm{1}_{\{L_t(\psi)<\varepsilon\}}\,\d t\\
		&\le \frac{1}{\tau_n}\sum_{j=0}^{n-1}\int_{\tau_j}^{\tau_{j+1}} 
			\bm{1}_{\{L_t(v_{j})<\varepsilon\}}\,\d t\\
		&\le \frac{1}{\tau_n}\sum_{j=0}^{n-1}\int_{\tau_j}^{\tau_{j+1}} 
			\bm{1}_{\{U_t(v_{j})<8\varepsilon\}}\,\d t
	\end{align*}
	$\P_{\mathbb{1}}$-almost surely. The very construction of the stopping times
	$\tau_1,\tau_2,\ldots$ ensures that if there exists $t\in[\tau_j\,,\tau_{j+1}]$,
	and a realization of the process $\psi$, such that $L_t(v_{j})<\varepsilon$, 
	then certainly $U_t(v_{j})\le 8\varepsilon<1$
	for all $t\in[\tau_j\,,\tau_{j+1}]$ [for that realization of $\psi$], whence also 
	$L_{\tau_{j+1}}(v_{j+1})\le 8\varepsilon<1$, for the same realization of
	$\psi$. In this way, we find that
	\[
		\frac{1}{\tau_n}\int_0^{\tau_n} \bm{1}_{\{\inf_{x\in\T}\psi(t,x)<\varepsilon\}}\,\d t
		\le \frac{1}{\tau_n}\sum_{j=0}^{n-1} \ell_{j+1}\bm{1}_{\{X_{j+1}\le
		-|\log_2(8\varepsilon)|\}}\qquad\text{$\P_{\mathbb{1}}$-a.s.}
	\]
	If $\lambda$ is sufficiently small,
	then Lemma \ref{lem:ell-underline} and the strong law of large numbers 
	together imply that
	\begin{equation}\label{liminf>0}
		\liminf_{n\to\infty} \frac{\tau_n}{n} = 
		\liminf_{n\to\infty}\frac1n \sum_{j=1}^n\ell_j \ge 
		\lim_{n\to\infty} \frac1n\sum_{j=1}^n\underline{\ell}_j
		= \E(\underline{\ell}_1)>0
		\qquad\text{$\P_{\mathbb{1}}$-a.s.}
	\end{equation}
	Similarly, Lemma \ref{lem:ell-bar} ensures that
	\begin{equation}\label{limsup-finite}
		\limsup_{n\to\infty}\frac{\tau_n}{n}\le\E(\bar{\ell}_1)<\infty
		\qquad\text{$\P_{\mathbb{1}}$-a.s.}
	\end{equation}
	
	The Cauchy-Schwarz inequality and \eqref{liminf>0}
	together yield the following: $\P_{\mathbb{1}}$-a.s.,
	\begin{align*}
		\limsup_{n\to\infty}
			\frac{1}{\tau_n}\int_0^{\tau_n} \bm{1}_{\{\inf_{x\in\T}\psi(t,x)<\varepsilon\}}\,\d t
		&\le \frac{1}{\E_{\mathbb{1}}(\underline{\ell}_1)}\limsup_{n\to\infty}
			\frac1n\sum_{j=0}^{n-1} \ell_{j+1}\bm{1}_{\{X_{j+1}\le
			-|\log_2(8\varepsilon)|\}}\\
		&\le \frac{1}{\E_{\mathbb{1}}(\underline{\ell}_1)}
			\sqrt{\limsup_{n\to\infty}\frac1n\sum_{j=0}^{n-1}\ell_{j+1}^2}
			\sqrt{\limsup_{n\to\infty}\frac1n\sum_{j=0}^{n-1}
			\bm{1}_{\{X_{j+1}\le -|\log_2(8\varepsilon)|}\}}.
	\end{align*}
	Thanks to Lemma \ref{lem:ell-bar}, we may deduce from the above that $\P_{\mathbb{1}}$-a.s.,
	\[
		\limsup_{n\to\infty}
		\frac{1}{\tau_n}\int_0^{\tau_n} \bm{1}_{\{\inf_{x\in\T}\psi(t,x)<\varepsilon\}}\,\d t
		\le \frac{\|\bar{\ell}_1\|_2}{\|\underline{\ell}_1\|_1}
		\sqrt{\limsup_{n\to\infty}\frac1n\sum_{j=0}^{n-1}
		\bm{1}_{\{X_{j+1}\le -|\log_2(8\varepsilon)|}\}}.
	\]
	This proves that \eqref{eq:goal:KB} implies \eqref{eq:condition-3}, except the
	non-random averaging variable $T\to\infty$ is replaced by the random averaging
	variable $\tau_n\to\infty$. In order to complete the proof, let us choose and fix
	2 numbers $a$ and $b$ such that
	\[
		\E(\bar{\ell}_1)>b\ge a>\E(\underline{\ell}_1).
	\]
	For every $T\gg1$ let $n=n(T)=\lceil T/a\rceil$, so that
	$a(n-1) < T \leq an$ and $n\ge 3$, whence $(n-1)^{-1}\le 1/2$. By enlarging
	$n(T)$ further to a finite random number, if need be, \eqref{liminf>0}
	and \eqref{limsup-finite} together ensure that
	\[
		an\le\tau_n\le bn,
	\]
	for all $T$ sufficiently large.
	In this way we find that, for all sufficiently large $T\gg1$,
	\begin{align*}
		\frac1T\int_0^T\bm{1}_{\{\inf_{x\in\T}\psi(t,x)<\varepsilon\}}
			\,\d t
			&\le \frac{1}{a(n-1)}\int_0^{an}\bm{1}_{\{
			\inf_{x\in\T}\psi(t,x)<\varepsilon\}}\,\d t\\
		&\le \frac{2}{an}\int_0^{an}\bm{1}_{\{\inf_{x\in\T}
			\psi(t,x)<\varepsilon\}}\,\d t\\
		&\le \frac{2b}{a\tau_n}\int_0^{\tau_n} \bm{1}_{\{
			\inf_{x\in\T}\psi(t,x)<\varepsilon\}}\,\d t.
	\end{align*}
	Now let $T\to\infty$ first, and then $a\to\E(\underline{\ell}_1)$
	and $b\to\E(\bar{\ell}_1)$ in order to see that $\P_{\mathbb{1}}$-a.s.,
	\begin{equation}\label{eq:lowertail1}
		\limsup_{T\to\infty}
		\frac1T\int_0^T \bm{1}_{\{\inf_{x\in\T}\psi(t,x)<\varepsilon\}}\,\d t
		\le \frac{2\|\bar{\ell}_1\|_2^2}{\|\underline{\ell}_1\|_1^2}
		\sqrt{\limsup_{m\to\infty}\frac1m\sum_{j=0}^{m-1}
		\bm{1}_{\{X_{j+1}\le -|\log_2(8\varepsilon)|}\}}.
	\end{equation}
	This proves Proposition \ref{pr:reduction}.
\end{proof}


\subsection{Proof of Proposition \ref{pr:inv:meas:exists}}

We are ready to begin completing the proof of Proposition \ref{pr:inv:meas:exists},
which assures us of the existence of non-trivial invariant measures  when ${\rm L}_\sigma>0$
and $\lambda$ is sufficiently small.
Our method requires an analysis of the excursions of the chain
$X$ from the level $M-1$. The construction of the chain $X$ ensures that 
\[
	\P_{\mathbb{1}} \{X_0=M-2\}=1.
\]
Moreover, $\P_{\mathbb{1}}\{ |X_{n+1}-X_n|=1\}=1$ for every $n\ge 1$.

Set $\alpha_0:=0$ and for all $n\in\Z_+$ define
\[
	\alpha_{n+1} := \inf\left\{ j>\alpha_n:\, X_j =M-1\right\},
\]
where $\inf\varnothing:=\infty$. Then, the $\alpha_n$'s are stopping times
in the filtration $\{\mathscr{F}_{\tau_n}\}_{n=0}^\infty$.

\begin{lemma}\label{lem:alpha_n-finite}
	$\P_{\mathbb{1}}\{\alpha_n<\infty\}=1$ for every $n\ge1$.
\end{lemma}

\begin{proof}
	The proof works by coupling the inhomogeneous Markov chain $X$ to 
	an infinite family of independent, biased random walks. This coupling is
	the motivation behind the title of this section (``a random walk argument''),
	and will be useful in the sequel as well.
	
	First we prove that $\P_{\mathbb{1}}\{\alpha_1<\infty\}=1$.
Define a random walk $\{Y_n^{(1)}\}_{n=1}^\infty$ on $\Z$ as follows:
	\begin{compactenum}
		\item $Y_0^{(1)}:=X_0=M-2$;
		\item Iteratively define $Y_n^{(1)}$ for every $n\ge1$ as follows:
			\begin{compactenum}
				\item When $X_{n+1}-X_n=-1$, set $Y_{n+1}^{(1)} = Y_n^{(1)}-1$;
				\item Next, let us introduce new variables 
					$\{\Delta_m\}_{m=1}^\infty$ that are independent of
					one another, such that a.s.\ on $\{X_m\le M-2\}$,
					\[
						\P\{\Delta_{m}=+1\}
						= 1 - \P\{\Delta_{m}=-1\}
						= \frac{2}{3\P( X_{m+1}-X_m=+1\mid \mathscr{F}_{\tau_m})},
					\]
					for every $m\ge 0$ such that $\P_{\mathbb{1}}\{X_m\le M-2\}>0$.
					For all other values of $m$, $\P\{\Delta_m=+1\}=0$. The preceding
					is a well-defined construction thanks to \eqref{eq:walk-moves-up}.
					Now we set $Y_{n+1}^{(1)} := Y_n^{(1)}+ \Delta_n$ whenever $X_{n+1}=X_n+1$ and 
					$\alpha_1>n$; and
				\item Let $\{Z_m\}_{m=1}^\infty$ be an independent, biased, simple random walk on 
					$\Z$ whose left-right probabilities given by $\P\{Z_1=+1\}=1-\P\{Z_1=-1\}=2/3$. 
					Finally, define
					$Y_{n+1}^{(1)} := Y_{n}^{(1)} + Z_{n+1-\alpha_1}$
					whenever $X_{n+1}=X_n+1$ and $\alpha_1\le n$.
			\end{compactenum}
	\end{compactenum}
	The above construction shows that $\{Y^{(1)}_n\}_{n=1}^\infty$ is a simple random
	walk on $\Z$ such that: 
	\begin{compactenum}
	\item[(i)] $Y^{(1)}_0=M-2$; 
	\item[(ii)]  $\P_{\mathbb{1}}\{Y^{(1)}_{n+1}-Y^{(1)}_n=+1\}=
		1-\P_{\mathbb{1}}\{Y^{(1)}_{n+1}-Y^{(1)}_n=-1\}=2/3$
		for all $n\ge0$; and 
	\item[(iii)] $\P_{\mathbb{1}}\{ Y^{(1)}_n\le X_n
			\text{ for all $1\le n\le\alpha_1$}\} =1$, where $Y^{(1)}_\infty:=X_\infty:=M-1$
			to make the notation work out correctly in case $\P_{\mathbb{1}}\{\alpha_1=\infty\}>0$
			[which we are about to rule out].
	\end{compactenum}
	Since $Y^{(1)}$ has an upward drift and starts at $M-2$, 
	it almost surely reaches $M-1$ in finite time. Because of item (iii) above, 
	$\alpha_1$ is not greater than the first hitting time of $M-2$ by $Y^{(1)}$.
	This proves that $\P_{\mathbb{1}}\{\alpha_1<\infty\}$; in fact, 
	that
	\[
		\limsup_{m\to\infty}m^{-1}\log\P_{\mathbb{1}}\{\alpha_1>m\}<0.
	\]
			
	To complete the proof, we work by induction.
	Suppose we have proved that
	$\P_{\mathbb{1}}\{\alpha_i<\infty\}=1$ for some $i\ge1$.
	We recycle the preceding random walk construction to produce
	a random walk $Y^{(i+1)}$ on $\Z$ that is 
	\emph{independent} of $Y^{(1)},\ldots,Y^{(i)}$ and:\footnote{In this proof
		we do not use the additional fact that $Y^{(1)},Y^{(2)},\ldots$ are independent
		from one another, but we will use that fact later on.}
	\begin{compactenum}
	\item[(i)] $Y^{(i+1)}_0=M-2$;
	\item[(ii)] $\P_{\mathbb{1}}\{Y^{(i+1)}_{n+1}-Y^{(i+1)}_n=+1\}=
		1-\P_{\mathbb{1}}\{Y^{(i+1)}_{n+1}-Y^{(i+1)}_n=-1\}=2/3$
		for all $n\ge0$; and 
	\item[(iii)] $\P_{\mathbb{1}}\{ Y^{(i+1)}_n\le X_{n+\alpha_i}\text{ for all 
		$1\le n\le\alpha_{i+1}-\alpha_i$}\}=1$
		where $Y^{(i+1)}_\infty:=X_\infty:=M-1$.
	\end{compactenum}
	Since $Y^{(i+1)}$ starts at $M-2$ and
	has an upward drift, it a.s.\ reaches $M-1$ in finite time. Therefore,
	the same argument that proved
	that $\P_{\mathbb{1}}\{\alpha_1<\infty\}=1$ now
	implies that $n\mapsto X_{n+\alpha_i}$ reaches $M-1$ in a.s.-finite
	time, whence $\P_{\mathbb{1}}\{\alpha_{i+1}<\infty\}=1$.
\end{proof}

Next we prove that $\alpha_n$ is asymptotically of sharp order $n$ as $n\to\infty$.
We will state and prove the complete result, though we need only the following 
asymptotic lower bound on $\alpha_n/n$.

\begin{lemma}\label{lem:alpha_n/n}
	$\P_{\mathbb{1}}$-almost surely,
	\[
		\frac23\le\liminf_{n\to\infty}\frac{\alpha_n}{n}\le
		\limsup_{n\to\infty}\frac{\alpha_n}{n}\le3.
	\]
\end{lemma}

\begin{proof}
	Recall the independent biased random walks $Y^{(1)},Y^{(2)},\ldots$
	of the proof of Lemma \ref{lem:alpha_n-finite}, and define for every $k\in\N$,
	\[
		\beta_k:= \inf\left\{j\ge 1:\, Y^{(k)}_j = M-1
		\right\},
	\]
	where $\inf\varnothing=\infty$. Choose and fix an integer $k\ge1$.
	Evidently, $\beta_1,\beta_2,\ldots$ are i.i.d.\ random variables. And
	since $Y^{(k)}_1=M-2$, $Y^{(k)}$ has
	positive upward drift, and $Y^{(1)}_n - (n/3)$ defines a mean-zero martingale,
	a gambler's ruin computation shows that $\E_{\mathbb{1}}(\beta_1)=3$.
	
	Define $\alpha_0:=0$.
	Because $Y^{(k)}_n \le X_{n+\alpha_{k-1}}$ $\P_{\mathbb{1}}$-a.s.\ for all $n\in\N$,
	it follows from a little book keeping that
	\[
		\beta_k\ge\inf\{n\ge1:\, X_{n+\alpha_{k-1}}=M-1\}.
	\]
	Apply induction on $k$ to see that $\beta_k\ge\alpha_k-\alpha_{k-1}$
	for all $k\in\N$, $\P_{\mathbb{1}}$-a.s. Thus, the  strong law of large
	numbers implies that
	\[
		\limsup_{n\to\infty}\frac{\alpha_n}{n} = 
		\limsup_{n\to\infty}\frac1n\sum_{k=1}^n (\alpha_k-\alpha_{k-1}) \le 
		\lim_{n\to\infty}\frac1n\sum_{k=1}^n\beta_k=\E_{\mathbb{1}}(\beta_1)
		=3\qquad\text{$\P_{\mathbb{1}}$-a.s.}
	\]
	For the converse bound we might observe that, if
	$\beta_k=1$, then $X_{1+\alpha_{k-1}}=M-1$ and hence
	$\alpha_k-\alpha_{k-1}=1$. Therefore,
	\[
		\liminf_{n\to\infty}\frac{\alpha_n}{n} = 
		\liminf_{n\to\infty}\frac1n\sum_{k=1}^n (\alpha_k-\alpha_{k-1})
		\ge \lim_{n\to\infty}\frac1n\sum_{k=1}^n\bm{1}_{\{\beta_k=1\}}
		=\frac23\qquad\text{$\P_{\mathbb{1}}$-a.s.,}
	\]
	thanks to the strong law of large numbers. The lemma follows.
\end{proof}

We are ready to conclude this subsection by verifying Proposition \ref{pr:inv:meas:exists};
namely, that if ${\rm L}_\sigma>0$
and $\lambda$ is sufficiently small (which we assume is the case), then 
there exists a non-trivial invariant measure.

\begin{proof}[Proof of Proposition \ref{pr:inv:meas:exists}]
	Recall the i.i.d.\ random walks $Y^{(1)},Y^{(2)},\ldots$
	from the proof of Lemma \ref{lem:alpha_n-finite}. The very
	construction of the $Y^{(i)}$'s implies that
	\[
		\sum_{j=1}^{\alpha_n}\bm{1}_{\{X_j\le -k\}}
		= \sum_{\ell=1}^n \sum_{j=1}^{\alpha_\ell-\alpha_{\ell-1}}
		\bm{1}_{\{X_{j+\alpha_{\ell-1}}\le-k\}}
		\le \sum_{\ell=1}^n \sum_{j=1}^\infty\bm{1}_{\{Y_j^{(\ell)}\le-k\}}
		=: \sum_{\ell=1}^n\chi_\ell,
	\]
	notation being clear. Now, $\chi_1,\chi_2,\ldots$ are i.i.d., and a standard
	computation shows that
	\[
		\E_{\mathbb{1}}[\chi_1] \le 17\cdot 2^{-(k-M+2)/2};
	\]
	see Lemma \ref{lem:E-num} for example. Therefore, Kolmogorov's strong law of large numbers
	implies that 
	\[
		\limsup_{n\to\infty} n^{-1}\sum_{j=1}^{\alpha_{2n}}\bm{1}_{\{X_j\le -k\}}
		\le 34\cdot 2^{-(k-M+2)/2},
	\]
	$\P_{\mathbb{1}}$-a.s.
	Lemma \ref{lem:alpha_n/n} ensures that $\alpha_{2n}\ge n$ for all $n$ large.
	Therefore,
	\begin{equation}\label{I(X)}
		\limsup_{n\to\infty}\frac1n\sum_{j=1}^n\bm{1}_{\{X_j\le -k\}}
		\le 17\cdot 2^{-(k-M)/2}\qquad\text{$\P_{\mathbb{1}}$-a.s.}
	\end{equation}
	This implies \eqref{eq:goal:KB}. Therefore, Proposition \ref{pr:inv:meas:exists}
	follows from Proposition \ref{pr:reduction}.
\end{proof}




\section{A Support Theorem}

In general, a ``support theorem'' for a probability measure $\nu$ is a full, or sometimes a partial,
description of the support of the measure $\nu$. In this section we provide a partial support
theorem that describes the support of the law of $\psi(t)$, at least for small values of $t$,
where $\psi$ solves the SPDE \eqref{eq:RD} starting from a given function $\psi_0\in C_{>0}(\T)$.

\begin{proposition}\label{pr:support}
	Choose and fix non-random number $A > A_0>0$, and $\alpha\in(0\,,1/2)$. 
	Then, for every non-random $\psi_0\in C_+(\T)$ with 
	$\frac12 A_0\le \inf_{x\in\T}\psi_0(x)\le
	\|\psi_0\|_{C^\alpha(\T)}\le A$,
	and for all $\delta>0$, there exists $t_0=t_0(A\,,A_0\,,\alpha\,,\delta)>0$
	and a strictly positive number $\mathfrak{p}_{A,A_0}(t_0\,,\alpha\,,\delta)$
	-- dependent on $(A\,,A_0\,,t_0\,,\alpha\,,\delta)$ but otherwise
	independently of $\psi_0$  --
	such that the solution $\psi$ to \eqref{eq:RD} with initial profile $\psi_0$ satisfies
	\[
		\P\left\{ \sup_{x\in\T}  |\psi(t_0\,,x)-A_0|  \le
		\delta \,, \|\psi(t_0)\|_{C^{\alpha/2}(\T)}\le
		A+1\right\} \ge \mathfrak{p}_{A,A_0}(t_0\,,\alpha\,,\delta).
	\]
\end{proposition}

\begin{proof}
	It clearly suffices to prove the result when $\delta$ is small. Therefore, we may [and will]
	assume without any loss in generality that
	\begin{equation}\label{cond:delta}
		\delta < \frac{1}{16}\wedge \frac{A_0}{2} \wedge
		\left( \frac{A_0}{2A\mathscr{C}_\alpha}\right)^{2/\alpha},
	\end{equation}
	is sufficiently small (but fixed), where
	\begin{equation}\label{C_alpha}
		\mathscr{C}_\alpha := \frac{1}{\sqrt{2\pi}}\int_{-\infty}^\infty |x|^\alpha
		\e^{-x^2/2}\,\d x = \frac{2^{\alpha/2}}{\sqrt\pi}\Gamma\left( \frac{1+\alpha}{2}\right).
	\end{equation}
	
	The essence of the idea is quite natural: By regularity estimates on the paths of the solution
	[Propositions \ref{pr:tight} and \ref{pr:temporal:cont}] we may choose $t$ sufficiently
	small to ensure that
	\[
		\P\left\{ \|\psi(t)-\psi_0\|_{C(\T)}\le\delta\,, \|\psi(t)\|_{C^{\alpha/2}(\T)}\le A_0+1\right\}>0.
	\]
	Then, we apply Girsanov's theorem \citep{Allouba,DZ96}
	to shift the center of the above radius-$\delta$ ball in $C(\T)$. 
	Because of the multiplicative
	nature of the noise in \eqref{eq:RD}, and since $\sigma$ vanishes 
	at zero, the said appeal to Girsanov's theorem is somewhat
	non trivial. Therefore, we write a detailed proof.
	
	Fix a real number $k\ge 2$, and recall the random field $\mathcal{I}$ from \eqref{I}.
	We may apply the BDG inequality, in a manner  similar to our method of proof of Lemma \ref{lem:cont},
	in order to see that there exists a real number $c_k>0$
	such that simultaneously for every $x\in\T$ and $t>0$,
	\begin{align*}
		\|\mathcal{I}(t\,,x)\|_k^2 &\le c_k\int_0^t\d s\int_{\T}\d y\
			[p_{t-s}(x\,,y)]^2 \|\sigma(\psi(s\,,y))\|_k^2\\
		&\le c_k\lip_\sigma^2\sup_{r\ge0}\sup_{z\in\T}\|\psi(r\,,z)\|_k^2\int_0^t\d s\int_{\T}\d y\
			[p_s(x\,,y)]^2\\
		&= c_k\lip_\sigma^2\sup_{r\ge0}\sup_{z\in\T}\|\psi(r\,,z)\|_k^2\int_0^t p_{2s}(0\,,0)\,\d s;
	\end{align*}
	we have appealed to the semigroup property of the heat kernel in order to deduce the last inequality.  
	Now, Lemma B.1 of \citet{KKMS20}
	tells us that $p_{2s}(0\,,0)\le 2\max\{ s^{-1/2}\,,1\}$, and Proposition \ref{pr:tight}
	implies that $\sup_{r\ge0}\sup_{z\in\T}\|\psi(r\,,z)\|_k<\infty$. Therefore, there exists 
	$c_k'>0$ such that
	\begin{equation}\label{I(t,x):1}
		\sup_{x\in\T}\|\mathcal{I}(t\,,x)\|_k \le c_k' \max\left\{ t^{1/4}\,, t^{1/2}\right\}
		\qquad\text{for all $t>0$}.
	\end{equation}
	Next, we might observe from Lemma \ref{lem:cont} and the proof of 
	Proposition together  that there exists $c_k''>0$ such that
	\begin{equation}\label{I(t,x):2}
		\|\mathcal{I}(t\,,x) - \mathcal{I}(s\,,x')\|_k \le c_k''\left\{ |x-x'|^{1/2} + 
		|t-s|^{1/4}\right\},
	\end{equation}
	uniformly for all  $s,t>0$ and $x,x'\in\T$. Therefore, we may apply chaining
	together with \eqref{I(t,x):1} and \eqref{I(t,x):2} in order to find that
	there exists $C_k>0$ such that
	\begin{equation}\label{||I||}
		\E\left( \sup_{s\in(0,t)}\|\mathcal{I}(s)\|_{C^{\alpha/2}(\T)}^k\right) \le
		C_k t^{k/5}\qquad\text{for all $t\in(0\,,1]$}.
	\end{equation}
	The careful reader might find that we have made a few arbitrary choices here:
	The $C^{\alpha/2}(\T)$-norm can be replaced by a $C^\beta(\T)$-norm
	for any $\beta\in(0\,,\alpha)$, and $t^{k/5}$ can be replaced by $t^{\theta k}$
	for any $\theta\in(0\,,\frac14)$. Of course, in that case, $C_k=C_k(\theta\,,\beta)$.
	
	Next, we consider the events
	\[
		\bm{\mathcal{G}}_t := \left\{ 
		\sup_{s\in(0,t)}\|\mathcal{I}(s)\|_{C^{\alpha/2}(\T)} \le \frac{\delta}{10\lambda}\right\},
	\]
	as $t$ roams over $(0\,,1]$. According to Chebyshev's inequality, and thanks to \eqref{||I||},
	\[
		\P\left(\bm{\mathcal{G}}_t^c\right) 
		\le \frac{10^k\lambda^k}{\delta^k}C_k t^{\theta k}
		\qquad\text{for all $t\in(0\,,1]$}.
	\]
	Since $\sup_{w\ge0}V(w)<\infty$ [see Lemma \ref{lem:F}],
	it follows that,  for all sufficiently small values of $t_0 \in(0\,,\delta)$,
	\begin{equation}\label{P(G)}\begin{aligned}
		&\P\left( \bm{\mathcal{G}}_{t_0} \right) \ge \frac{1}{2},&
			2\sup_{w\ge0}V(w)t_0 < \frac{\delta}{10},\\
		&2\sup_{|w|<A+1}|V(w)| t_0\le \frac{\delta}{10},\qquad&
			\sup_{|w|<A+1}|V(w)|\sum_{k=1}^\infty \frac{k^2t_0\wedge 1}{k^{(4-\alpha)/2}}\le\frac{1}{20},\\
			& (A\mathscr{C}_\alpha {t_0})^{\alpha/2} \leq \frac{\delta}{10}.
	\end{aligned}\end{equation}
	From now on, we select and fix such a $t_0 = t_0(\delta\,,A) \in (0\,,\delta)$.
	
	According to \eqref{eq:mild}, with probability one,
	\[
		\psi(s\,,x) = (\mathcal{P}_s\psi_0)(x) + \int_0^s\d r\int_{\T}\d y\
		p_{s-r}(x\,,y) V(\psi(r\,,y)) + \lambda \mathcal{I}(s\,,x)\qquad\text{for
		all $s>0$ and $x\in\T$}.
	\]
	Let $W$ denote the Brownian sheet that is naturally associated to the white noise
	$\dot{W}$; that is,
	\[
		W(s\,,x) := \int_{(0,s)\times[-1,x]}\dot{W}(\d r\,\d y)\qquad\text{for
		all $s>0$ and $x\in\T$},
	\]
	where we recall $\T$, as a set, is identified with the interval $[-1\,,1]$. 
	
	Define
	\[
		D_s := \e^{-M_s -\frac12\< M\>_s} \qquad\text{for
		all $s>0$},
	\]
	where $\{M_s\}_{s\ge0}$ is the continuous local martingale defined by
	\[
		M_s :=  \frac{1}{\lambda t_0}  \int_{(0,s)\times\T}
		\left\{ \psi_0(y) - \frac{A_0}{2}\right\}
		\frac{\bm{1}_{[\delta/2,A+1]}(\psi(r\,,y))}{\sigma(\psi(r\,,y))}\, W(\d r\,\d y).
	\]
	Because ${\rm L}_\sigma>0$ [see \eqref{LL}],
	the quadratic variation of $M$ satisfies
	\begin{align*}
		\< M\>_s &= \frac{1}{\lambda^2 t_0^2}
			\int_{(0,s)\times\T} \left\{ \psi_0(y) - \frac{A_0}{2}\right\}^2
			\frac{\bm{1}_{[\delta/2,A+1]}(\psi(r\,,y))}{\sigma^2(\psi(r\,,y))}
			\,\d r\,\d y\\
		&\le\frac{1}{\lambda^2 t_0^2{\rm L}_\sigma^2}
			\left\{ A+ \frac{A_0}{2}\right\}^2\int_{(0,s)\times\T} 
			\frac{\bm{1}_{[\delta/2,A+1]}(\psi(r\,,y))}{|\psi(r\,,y)|^2}
			\,\d r\,\d y\\
		&\le\frac{4}{\lambda^2 t_0^2{\rm L}_\sigma^2\delta^2}
			\left\{ A+ \frac{A_0}{2}\right\}^2 \int_{(0,s)\times\T} \,\d r\,\d y=:Cs
			\quad\text{for every $s>0$}.
	\end{align*}
	This inequality shows
	that the exponential local martingale
	$\{D_s\}_{s\ge0}$ is in fact a continuous $L^2(\P)$-martingale. 
	The DDS martingale representation theorem ensures the existence of a Brownian motion
	$B$ such that $M_s=B(\<M\>_s)$ for all $s\ge0$, whence we learn from the 
	Cauchy-Schwarz inequality and the reflection principle
	that, for every $s>0$,
	\[
		\E(D_s^2) \le  \E\left( \e^{2M_s}\right) 
		=\E\left( \e^{ 2B(\<M\>_s) }\right)
		\le \E\left[ \exp\left\{ 2\sup_{r\in[0, Cs]}B(r)\right\} \right]
		\le 2\e^{2Cs}.
	\]
	Define
	\[
		\underline{W}(s\,,x) := W(s\,,x) + \frac{1}{\lambda t_0} \int_{(0,s)\times\T}
		\left\{ \psi_0(y) - \frac{A_0}{2}\right\}\frac{\bm{1}_{[\delta/2,A+1]}(\psi(r\,,y))}{\sigma(\psi(r\,,y))}
		\,\d r\,\d y,
	\]
	for all $s>0$ and $x\in\T$.
	Girsanov's theorem
	ensures that $\underline{W}$ is a Brownian sheet under the
	measure $\mathrm{Q}$ defined via
	\[
		D_s :=
		\left.\frac{\d\mathrm{Q}}{\d\P}\right|_{\mathscr{F}_s}\qquad\text{for
		all $s>0$};
	\]
	see \cite{Allouba}  for the precise version of the Girsanov theorem that we need, and 
	Chapter 10 of \cite{DZ14} for the general theory.
	Among other things, the Cauchy-Schwarz inequality implies that
	\begin{equation}\label{Q<P}
		\mathrm{Q}(\Lambda) \le \|D_t\|_2\sqrt{\P(\Lambda)}
		\le \e^{3Ct_0/2}\sqrt{2\P(\Lambda)}
		\qquad\text{for all $\Lambda\in\mathscr{F}_{t_0}$}.
	\end{equation}
	And a similar estimate holds that bounds $\P(\Lambda)$ by  a
	[large] multiple of $\sqrt{Q(\Lambda)}$ for every $\Lambda\in\mathscr{F}_{t_0}$.
	In particular, it follows that ${\rm Q}$ and $\P$ are mutually absolutely continuous
	probability measures on the sigma algebra $\mathscr{F}_{t_0}$.
	
	The above application of Girsanov's theorem ensures that 
	$\psi$ solves the following SPDE driven by $\underline{\dot{W}}$:
	$\mathrm{Q}$-almost surely,
	\[
		\partial_t\psi(s\,,x) = \partial^2_x\psi(s\,,x) + V(\psi(s\,,x)) 
		- \frac{1}{ t_0}\left(\psi_0(x)- \frac{A_0}{2}\right) \bm{1}_{[\delta/2,A+1]}(\psi(s\,,x)) 
		+ \lambda\sigma(\psi(s\,,x))\,\underline{\dot{W}}(s\,,x).
	\]
	We can write this in mild form [see \eqref{eq:mild}] in order to see that
	with probability one $[\mathrm{Q}]$,
	\begin{equation}\label{psi:dec}
		\psi(s\,,x) = (\mathcal{P}_s\psi_0)(x) + J_1(s\,,x) - J_2(s\,,x) + \lambda\underline{\mathcal{I}}(s\,,x),
	\end{equation}
	where 
	\begin{align*}
		J_1(s\,,x) & :=  \int_{(0,s)\times\T}p_{s-r}(x\,,y)V(\psi(r\,,y))\,\d r\,\d y,\\
		J_2(s\,,x) & :=  \frac{1}{t_0}
			\int_{(0,s)\times\T} p_{s-r}(x\,,y)
			\left\{ \psi_0(y) - \frac{A_0}{2}\right\}\bm{1}_{[\delta/2,A+1]}(\psi(r\,,y))\,\d r\,\d y,
	\end{align*}
	for every $s>0$ and $x\in\T$,
	and $\underline{\mathcal{I}}$ is defined exactly as was $\mathcal{I}$,
	but with $W$ replaced by $\underline{W}$. 
	
	Next, consider the events, 
	\[
		\underline{\bm{\mathcal{G}}}_s :=  \left\{ 
		\sup_{r\in(0,s)}\|\underline{\mathcal{I}}(r)\|_{C^{\alpha/2}(\T)} \le \frac{\delta}{10\lambda}\right\}.
	\]
	That is, $\underline{\bm{\mathcal{G}}}_s$ is defined exactly as was
	$\bm{\mathcal{G}}_s$, but with $\mathcal{I}$ replaced by $\underline{\mathcal{I}}$.
	Recall from Lemma \ref{lem:F}
	that
	\[
		K=\sup_{w\in\R}V(w)<\infty,
	\]
	and since $\psi\ge0$, observe that
	$\sup_{x\in\T}J_1(s\,,x) \le 2Ks\le 2Kt_0$ for every $s\in(0\,,t_0]$ a.s.
	Because $J_2(s\,,x)\ge 0$ a.s., we combine these statements 
	with \eqref{P(G)} and \eqref{psi:dec} in order to see that
	\begin{equation}\label{eq:duck1}
		\adjustlimits\sup_{s\in(0,t_0]}\sup_{x\in\T}\psi(s\,,x) \le A + 2Kt_0 + \frac{\delta}{10}
		<A+\frac12\qquad\text{${\rm Q}$-a.s.\ on $\underline{\bm{\mathcal{G}}}_{t_0}$}.
	\end{equation}
	Now, \eqref{eq:duck1} and \eqref{P(G)} together imply that, for all $s\in(0\,,t_0]$ and $x\in\T$,
	\begin{equation}\label{J_1<}
		|J_1(s\,,x)| \le \sup_{|w|\le A+1}|V(w)|
		\int_{(0,s)\times\T}p_{s-r}(x\,,y)\,\d r\,\d y \le \frac{\delta}{10}
		\qquad\text{${\rm Q}$-a.s.\ 
		on $\underline{\bm{\mathcal{G}}}_{t_0}$}.
	\end{equation}
	Next, we observe that
	\begin{align*}
		(\mathcal{P}_s\psi_0)(x) - J_2(s\,,x) &\ge (\mathcal{P}_s\psi_0)(x)
			-  \frac{1}{t_0}\int_{(0,s)\times\T} p_{s-r}(x\,,y) \left\{ \psi_0(y) - 
			\frac{A_0}{2}\right\}\,\d r\,\d y\\
		&  = (\mathcal{P}_s\psi_0)(x)
			-  \frac{1}{t_0}\int_0^s (\mathcal{P}_r\psi_0)(x)\,\d r  + \frac{s}{t_0}\frac{A_0}{2},
	\end{align*}
	for all $s\in(0\,,t_0]$ and $x\in\T$. Let $\{\beta(s)\}_{s\ge0}$ denote a Brownian motion
	on $\T$ and observe that
	\begin{equation}\label{eq:diff0}
	\begin{aligned}
		\left| \left(\frac st_0\right)(\mathcal{P}_s\psi_0)(x)
			-  \frac{1}{t_0}\int_0^s (\mathcal{P}_r\psi_0)(x)\,\d r\right|
			&= \frac{1}{t_0}\left|\int_0^s\left\{  \E\left[ \psi_0(\beta(s)+x) \right] 
			-\E\left[ \psi_0(\beta(r)+x) \right] \right\}\d r\right|\\
		&\le \frac{\|\psi_0\|_{C^\alpha(\T)}}{t_0}\int_0^s
			\E\left( |\beta(s)-\beta(r)|^\alpha\right)\d r\\
		&\le \frac{A\mathscr{C}_\alpha}{t_0}\int_0^s r^{\alpha/2}\,\d r&\text{[see \eqref{C_alpha}]}\\
		&<A\mathscr{C}_\alpha s^{\alpha/2}\qquad\text{for all $s\in(0\,,t_0]$ and $x\in\T$}.
	\end{aligned}
	\end{equation}
	Since $A\mathscr{C}_\alpha s^{\alpha/2} \le A\mathscr{C}_\alpha t_0^{\alpha/2}
	\leq \delta/10$, it follows from the preceding and from \eqref{cond:delta} that
	\[
		(\mathcal{P}_s\psi_0)(x)
		-  \frac{1}{t_0}\int_0^s (\mathcal{P}_r\psi_0)(x)\,\d r
		\ge \left(1-\frac{s}{t_0}\right)\frac{A_0}{2} -\frac{\delta}{10}
		\qquad\text{for all $s\in(0\,,t_0]$ and $x\in\T$}.
	\]
	and hence
	\[
		(\mathcal{P}_s\psi_0)(x)
		-  J_2(s\,,x) \ge \frac{A_0}{2} -\frac{\delta}{10}.
	\]
	Thus,  \eqref{P(G)},  \eqref{psi:dec}, and \eqref{J_1<} together ensure that
	\begin{equation}\label{eq:duck2}
		\adjustlimits\inf_{s\in(0,t_0]}\inf_{x\in\T}\psi(s\,,x)
		\ge \frac{A_0}{2} -\frac{3\delta}{10} > \frac{\delta}{2}\qquad\text{${\rm Q}$-a.s.\ 
		on $\underline{\bm{\mathcal{G}}}_{t_0}$},
	\end{equation}
	and \eqref{eq:duck1} and \eqref{eq:duck2} together yield the following
	${\rm Q}$-a.s.\ on $\underline{\bm{\mathcal{G}}}_{t_0}$: For
	all  $x\in\T$,
	\begin{equation}\label{LHS=A_0}\begin{split}
		(\mathcal{P}_{t_0}\psi_0)(x) - J_2(t_0\,,x) &= (\mathcal{P}_{t_0}\psi_0)(x)
			-  \frac{1}{t_0}\int_{(0,t_0)\times\T} \left\{ 
			\psi_0(y) - \frac{A_0}{2}\right\}p_{t_0-r}(x\,,y) \,\d r\,\d y\\
		&=A_0 + \mathcal{P}_{t_0}\psi_0(x) - \frac{1}{t_0} \int_0^{t_0} (\mathcal{P}_r\psi_0)(x) \, \d r.
	\end{split}\end{equation}
	Thus, it follows from \eqref{psi:dec}, \eqref{J_1<}, \eqref{eq:diff0}
	and the definition of the event $\underline{\bm{\mathcal{G}}}_{t_0}$ that,
	${\rm Q}$-a.s.\ on $\underline{\bm{\mathcal{G}}}_{t_0}$,
	\begin{equation}\label{good:stuff:1}
		\sup_{x\in\T}|\psi(t_0\,,x) - A_0 | \le \|(\mathcal{P}_s\psi_0 - J_2(s)\|_{C(\T)}+
		\|J_1(t_0)\|_{C(\T)} + \lambda
		\left\| \underline{\mathcal{I}}(t_0)\right\|_{C(\T)}
		\le \frac{3\delta}{10}<\delta.
	\end{equation}
	Also, \eqref{LHS=A_0} tells us that
	${\rm Q}$-a.s.\ on $\underline{\bm{\mathcal{G}}}_{t_0}$, 
	\begin{equation}\label{psi:good}\begin{split}
		|\psi(t_0\,,x) - \psi(t_0\,,y)| &\le |J_1(t_0\,,x)-J_1(t_0\,,z)| + 
			\lambda|\underline{\mathcal{I}}(t_0\,,x)-
			\underline{\mathcal{I}}(t_0\,,z)| +2\|(\mathcal{P}_s\psi_0 - J_2(s)\|_{C(\T)}\\
		&\le |J_1(t_0\,,x)-J_1(t_0\,,z)| + 
			\frac{\delta}{10}|x-z|^{\alpha/2} +\frac{\delta}{5},
	\end{split}\end{equation}
	simultaneously for all $x,z\in\T$. We estimate the remaining term
	as follows:
	Because of \eqref{eq:duck1}, the following holds ${\rm Q}$-a.s.\ on $\underline{\bm{\mathcal{G}}}_{t_0}$:
	\[
		|J_1(t_0\,,x)-J_1(t_0\,,z)|\le \sup_{w\le A+1}|V(w)|\int_{(0,t_0)\times\T}
		\left| p_r(x\,,y)-p_r(z\,,y)\right|\,\d r\,\d y,
	\]
	simultaneously for all $x,z\in\T$. Now we apply \eqref{p-p} to see that
	${\rm Q}$-a.s.\ on $\underline{\bm{\mathcal{G}}}_{t_0}$,
	\[
		|J_1(t_0\,,x)-J_1(t_0\,,z)| \le 2\sqrt 2\sup_{w\le A+1}|V(w)|
		\sum_{k=1}^\infty \left( |x-z|k\wedge 1\right)\int_0^{t_0}
		\e^{-\pi^2 k^2r}\,\d r,
	\]
	simultaneously for every $x,z\in\T$. Since $(|a|\wedge1)\le|a|^{\alpha/2}$ for all $a\in\R$,
	it follows that ${\rm Q}$-a.s.\ on $\underline{\bm{\mathcal{G}}}_{t_0}$,
	\begin{align*}
		|J_1(t_0\,,x)-J_1(t_0\,,z)| & \le 2\sqrt 2\sup_{w\le A+1}|V(w)|
			|x-z|^{\alpha/2} 
			\sum_{k=1}^\infty k^{\alpha/2}\left( \frac{1 - \e^{-\pi^2 k^2t_0}}{\pi^2k^2}\right)\\
		&\le \frac{2\sqrt 2}{\pi^2}\sup_{w\le A+1}|V(w)||x-z|^{\alpha/2} 
			\sum_{k=1}^\infty \frac{k^2t_0\wedge 1}{k^{(4-\alpha)/2}}\\
		&\le \frac{\sqrt 2}{20}\, |x-z|^{\alpha/2}\qquad\text{for all $x,z\in\T$};
	\end{align*}
	see \eqref{P(G)}.  This and \eqref{psi:good} together yield
	\[
		\sup_{\substack{x,z\in\T\\x\neq z}}
		\frac{|\psi(t_0\,,x) - \psi(t_0\,,z)|}{|x-z|^{\alpha/2}} \le 
		\frac{\sqrt 2}{20} + \frac{3\delta}{10} \qquad\text{
		${\rm Q}$-a.s.\ on $\underline{\bm{\mathcal{G}}}_{t_0}$.}
	\]
	Thus, we may deduce from \eqref{eq:duck1} that
	\begin{equation}\label{good:stuff:2}
		\|\psi(t_0)\|_{C^\alpha(\T)} \le 
		A + \frac12 + \frac{\sqrt 2}{20} + \frac{3\delta}{10} < A+1\qquad\text{
		${\rm Q}$-a.s.\ on $\underline{\bm{\mathcal{G}}}_{t_0}$.}
	\end{equation}
	
	Thanks to \eqref{P(G)} and Girsanov's theorem,
	${\rm Q}( \underline{\bm{\mathcal{G}}}_{t_0}) =
	\P( \bm{\mathcal{G}}_{t_0}) \ge 1/2$. Therefore,
	\eqref{Q<P}, \eqref{good:stuff:1}, and \eqref{good:stuff:2} together imply that
	\begin{align*}
		&\P\left\{  \left\| \psi(t_0) - A_0 \right\|_{C(\T)} \le \delta \,, \|\psi(t_0)\|_{C^{\alpha/2}(\T)}\le
			A+1\right\}\\
		&\ge \e^{-3Ct_0}\left|{\rm Q}
			\left\{ \left\| \psi(t_0) - A_0 \right\|_{C(\T)} \le \delta \,, \|\psi(t_0)\|_{C^{\alpha/2}(\T)}\le
			A+1 \right\}\right|^2 \ge \e^{-3Ct_0}\left|{\rm Q}(\underline{\bm{\mathcal{G}}}_{t_0})\right|^2
			\ge \frac14\e^{-3Ct_0}.
	\end{align*}
	This has the desired result.
\end{proof}




\section{Natural, Independent, and AM/PM Couplings}\label{sec:coupling}

The principal aim of this section is proof of the statement that if ${\rm L}_\sigma>0$
and $\lambda$ is small, then there can 
only exist one probability measure $\mu_+$ on $C_+(\T)$ such that $\mu_+\{\mathbb{0}\}=0$ and $\mu_+$
is invariant for the SPDE \eqref{eq:RD}. 
We have demonstrated already  in Proposition \ref{pr:inv:meas:exists}
that if ${\rm L}_\sigma>0$ and $\lambda$ is sufficiently small, then
at least one invariant measure $\mu_+$ exists such that $\mu_+\{\mathbb{0}\}=0$. 
The main point of this section is that $\mu_+$ is the only invariant measure of the type
outlined. In order to do this, we build on coupling ideas of \cite{Mueller-Coupling}.
Let $\psi_1$ and $\psi_2$ denote the  solutions to the SPDE
\eqref{eq:RD} starting respectively, given respective initial data $\psi_{1,0},\psi_{2,0}\in C_{>0}(\T)$.
We will not assume that they are driven by the same noise, or even are defined on the same probability space.
With this in mind, recall that a \emph{coupling} of $(\psi_1\,,\psi_2)$ is a construction of $(\psi_1\,,\psi_2)$
jointly on the same probability space such that $\psi_1$ and $\psi_2$ have the correct respective
marginals. In other words, a coupling of $(\psi_1\,,\psi_2)$ involves the construction of 
two space-time white noises $\dot{\mathcal{W}}_1$ and  $\dot{\mathcal{W}}_2$ such that
the following stochastic integral equations
\begin{equation}\label{psi_j}\begin{split}
	\psi_j(t\,,x) &= (\mathcal{P}_t\psi_{j,0})(x) + \int_{(0,t)\times\T}
		p_{t-s}(x\,,y) V(\psi_j(s\,,y))\,\d s\,\d y\\
	&\hskip2in +\lambda
		\int_{(0,t)\times\T} p_{t-s}(x\,,y) \sigma(\psi_j(s\,,y))\,\mathcal{W}_j(\d s\,\d y)
\end{split}\end{equation}
are valid for all $t>0$, $x\in\T$, and $j\in\{1\,,2\}$. The novelty here can be in the fact that
$\mathcal{W}_1$ and $\mathcal{W}_2$ might be correlated with one another, and even
constructed {\it a priori} using the solution $(\psi_1\,,\psi_2)$ to the above. This is of course
a pairwise coupling. One can imagine also a more general $N$-wise coupling of $N\ge2$
solutions to \eqref{eq:RD}, etc.

Next we devote some time to describe
four notions of coupling, all of which are used in this paper. We call them
\emph{natural}, \emph{independent},  \emph{pairwise monotone} (PM, for short), 
and \emph{anchored monotone} (AM, for short) couplings
for the sake of comparison and ease of later reference. The first two coupling methods
are standard; the more subtle PM and the AM couplings of this paper were introduced
in \cite{Mueller-Coupling}.\\

\noindent\textbf{\itshape (i) Natural coupling.}
By a \emph{natural coupling} of $\psi_1$ and $\psi_2$ we simply mean the construction of
$\psi_1$ and $\psi_2$ using the same underlying white noise. This is the coupling that we have
tacitly used so far in the paper. The natural coupling has a number of obvious advantages.
For example, if $\psi_{1,0}\le\psi_{2,0}$, then $\psi_1\le\psi_2$ a.s.; see Lemma \ref{lem:comparison}.
Another attractive feature of natural couplings is that they are not limited to pairwise couplings, or even
$N$-wise couplings. One can in fact solve \eqref{eq:RD} simultaneously for every non-random
initial profile $\psi_0\in C_+(\T)$.\\

\noindent\textbf{\itshape (ii) Independent coupling.}
By an \emph{independent} coupling of $\psi_1$ and $\psi_2$ we simply mean that 
the underlying noises $\dot{\mathcal{W}}_1$ and $\dot{\mathcal{W}}_2$ in \eqref{psi_j}
are independent from one another. This is the most naive form of coupling, but as we shall
see has its uses.\\

\noindent\textbf{\itshape (iii) Pairwise monotone {\bf (PM)} coupling.}
\emph{PM coupling} refers to the first step of a two-step coupling method
that was introduced in \cite{Mueller-Coupling}. In order to recall
that method, and adapt it to the present setting, let us first 
define $\dot{\mathcal{W}}_1$ and $\dot{\mathcal{W}}_2$ to be two independent
space-time white noises.  Also, consider the real-valued functions
\begin{equation}\label{fg}
	f(y) := \sqrt{|y|\wedge 1}\quad\text{and}\quad
	g(y) := \sqrt{1 - |f(y)|^2} 
	=\sqrt{ 1 - \left( |y|\wedge 1\right)}\qquad\text{for $y\in\R$}.
\end{equation}
Now, we first let $\psi_1$ solve \eqref{eq:RD}, driven by $\dot{W}_1$; that is,
\begin{equation}\label{eq:psi_1}\begin{split}
	\psi_1(t\,,x) &= (\mathcal{P}_t\psi_{1,0})(x) + \int_{(0,t)\times\T}
		p_{t-s}(x\,,y) V(\psi_1(s\,,y))\,\d s\,\d y \\
	&\hskip2in + \lambda \int_{(0,t)\times\T} p_{t-s}(x\,,y)
		\sigma(\psi_1(s\,,y))\,\mathcal{W}_1(\d s\,\d y),
\end{split}\end{equation}
for every $t>0$ and $x\in\T$. Next, we let $\psi_2$ define the solution to the
coupled SPDE,
\begin{equation}\label{eq:psi_2}\begin{split}
	\psi_2(t\,,x) &= (\mathcal{P}_t\psi_{2,0})(x) + \int_{(0,t)\times\T}
		p_{t-s}(x\,,y) V(\psi_2(s\,,y))\,\d s\,\d y \\
	&\hskip.7in + \lambda \int_{(0,t)\times\T} p_{t-s}(x\,,y)
		\sigma(\psi_2(s\,,y)) g\left(\psi_1(s\,,y))-\psi_2(s\,,y)\right)
		\mathcal{W}_1(\d s\,\d y) \\
	&\hskip.7in + \lambda \int_{(0,t)\times\T} p_{t-s}(x\,,y)
		\sigma(\psi_2(s\,,y)) f\left(\psi_1(s\,,y))-\psi_2(s\,,y)\right)
		\mathcal{W}_2(\d s\,\d y).
\end{split}\end{equation}
Soon, we will elaborate on the existence of the PM coupling briefly,
following the work of \cite{Mueller-Coupling}, and adapting that
work to the present setting. For now, let us make a few remarks:
\begin{compactitem}
\item[-] As was mentioned by \cite{Mueller-Coupling} in a similar setting,
	we do not make a statement about the pathwise uniqueness of the solution to the
	SPDE system that defines $(\psi_1\,,\psi_2)$ in the PM coupling.
	Nor does pathwise uniqueness affect us. We care only about weak existence and
	uniqueness (in the probabilistic sense), which we shall establish soon.
\item[-] We can write the PM coupling of $(\psi_1\,,\psi_2)$, in differential notation,
	as the following interacting pair of SPDEs:
	\begin{equation*}\begin{split}
		\text{\ding{192}}&\left[\begin{split}
				&\partial_t\psi_1 = \partial^2_x\psi_1 + V(\psi_1) 
					+ \lambda\sigma(\psi_1)\dot{\mathcal{W}}_1
					\hskip2.3in \text{on $(0\,,\infty)\times\T$},\\
				&\text{subject to }\psi_1(0)=\psi_{1,0}
					\hskip2.95in \text{on $\T$},\\
			\end{split}\right.\\
		\text{\ding{193}}&\left[\begin{split}
				&\partial_t\psi_2 = \partial^2_x\psi_2 + V(\psi_2) 
					+ \lambda\sigma(\psi_2)\left[ g\left( \psi_1-\psi_2\right)\dot{\mathcal{W}}_1
					+ f\left( \psi_1-\psi_2\right)\dot{\mathcal{W}}_2\right]
					\hskip.15in \text{on $(0\,,\infty)\times\T$},\\
				&\text{subject to }\psi_2(0)=\psi_{2,0}
					\hskip2.93in \text{on $\T$}.\\
			\end{split}\right.{}.
	\end{split}\end{equation*}
	As long as the solution $(\psi_1\,,\psi_2)$ 
	exists as a 2-D predictable random field in the sense of \cite{wal86},
	and because $f^2+g^2=1$, the
	random distribution
	$g( \psi_1-\psi_2)\dot{\mathcal{W}}_1
	+ f( \psi_1-\psi_2)\dot{\mathcal{W}}_2$
	is {\it a priori} a space-time white noise; see Corollary \ref{cor:WN} of
	the appendix. This proves that if \ding{192} and \ding{193} jointly have a 
	random field solution
	$(\psi_1\,,\psi_2)$, then that solution is {\it a fortiori} a coupling 
	of $\psi_1$ and $\psi_2$. 
	
\item[-] If $|\psi_1-\psi_2|\ll 1$, then 
	$g(\psi_1-\psi_2)\approx1$ and $f(\psi_1-\psi_2)\approx 0$, and
	if $|\psi_1-\psi_2|\gg 1$, then 
	$g(\psi_1-\psi_2)\approx0$ and $f(\psi_1-\psi_2)\approx 1$. This suggests 
	somewhat informally that the PM coupling of $(\psi_1\,,\psi_2)$ ought to behave 
	roughly as follows:
	\[
		\partial_t\psi_2\approx\begin{cases}
			\partial^2_x\psi_2 + V(\psi_2) + \lambda \sigma(\psi_2)
			\dot{\mathcal{W}}_1&\text{when $|\psi_1-\psi_2|\ll1$},\\
		\partial^2_x\psi_2 + V(\psi_2) + \lambda \sigma(\psi_2)
			\dot{\mathcal{W}}_2 &\text{when $|\psi_1-\psi_2|\gg1$}.
		\end{cases}
	\]
	Of course, these remarks are not rigorous, in part because SPDEs are not local equations.
	Still, the preceding serves as a reasonable heuristic to
	suggest that the PM coupling of $(\psi_1\,,\psi_2)$ ought to behave as independent coupling when 
	$\psi_1$ and $\psi_2$ are far apart, and it works as natural coupling when
	$\psi_1$ and $\psi_2$ are close.
\end{compactitem}

Before we go on to describe AM coupling, let us pause and state and prove an existence
result [Proposition \ref{pr:coupling}], and a ``successful coupling'' result
[Lemma \ref{lem:pm coupling}],
for PM couplings. In particular, part 2 of the following proposition justifies the terminology
``pairwise monotone,'' or ``PM.''

\begin{proposition}\label{pr:coupling}
	Choose and fix two non-random functions
	$\psi_{1,0},\psi_{2,0}\in C^\alpha_+(\T)$
	for some $\alpha\in(0\,,1/2)$.
	After possibly enlarging the underlying probability space,
	one can contruct a pair $(\dot{\mathcal{W}}_1\,,\dot{\mathcal{W}}_2)$ of
	two independent space-time white noises for which \eqref{eq:psi_1} and
	\eqref{eq:psi_2} have random-field solutions $(\psi_1\,,\psi_2)$. Moreover:
	\begin{compactenum}
	\item For every $i\in\{1\,,2\}$, the law of  
		$\psi_i$ is the same as the law of \eqref{eq:RD} started at $\psi_{i,0}$;
	\item  If, in addition, $\psi_{1,0}\ge\psi_{2,0}$, then
		\[
			\P\{ \psi_1(t\,,x) \ge \psi_2(t\,,x)\text{ for all $t\ge0$ and $x\in\T$}\}=1;
			\text{ and}
		\]
	\item $\{(\psi_1(t)\,,\psi_2(t))\}_{t\ge0}$ is a Feller process with values in the space
		$C(\T\,;\R^2)$.
	\end{compactenum}
\end{proposition}

\begin{proof}
	If the functions $x\mapsto V(x)=x-F(x)$ and $\sigma$ were replaced by bounded,
	globally Lipschitz functions, then parts 1 and 3 of this proposition reduce to the construction of
	\cite{Mueller-Coupling} with our $(\psi_1\,,\psi_2)$
	being replaced by $(u\,,v)$ of Mueller ({\it ibid.}). 
	We adapt Mueller's arguments, and fill in some additional details
	to cover the present setting. 
	
	Let us start with two independent space-time white noises $\dot{W}_1$ and $\dot{W}_2$.
	Theorem \ref{th:exist-unique} ensures that the process $\psi_1$ of \eqref{eq:psi_1}
	is well defined on any probability
	space that supports a space-time white noise $\dot{W}_1$. However, the non-Lipschitz
	behavior of $f$ and $g$ at the origin prevent us from using standard SPDE machinary to
	produce a strong solution $\psi_2$. We overcome this, as in \cite{Mueller-Coupling},
	by producing instead a weak solution (in the sense of probability).
	
	We follow \cite{Mueller-Coupling} and define, for every $n\in\N$ and $y\in\R$,
	\[
		f_n(y) := \left( \left[|y| + \frac1n\right]
		\wedge 1 \right)^{1/2} - \left( \frac1n\right)^{1/2}\quad\text{and}\quad
		g_n(y) := \sqrt{1 - |f_n(x)|^2}.
	\]
	Then every $f_n$ and $g_n$ is a Lipschitz continuous function, 
	and $\lim_{n\to\infty} f_n=f$ and $\lim_{n\to\infty} g_n=g$, 
	both limits holding uniformly on $\R$.
	
	Recall that $\psi_1$ has already been constructed via \eqref{eq:psi_1} using
	the space-time white noise $\dot{W}_1$. Theorem \ref{th:exist-unique} of course
	justifies the existence and uniqueness of this construction.
	
	For every $n,N\in\N$, let $\psi_{2,n,N}$ 
	denote the solution to the  SPDE,
	\begin{equation*}\left[\begin{split}
		&\partial_t \psi_{2,n,N}=\partial^2_x \psi_{2,n,N}  + 
			V_N(\psi_{2,n,N})  + \lambda \sigma(\psi_{2,n,N})  
			\left\{ g_n\left( \psi_1-\psi_{2,n,N}\right)\dot{W}_1
			+ f_n\left( \psi_1-\psi_{2,n,N}\right) \dot{W}_2\right\},\\
		&\hskip3in\text{on $(0\,,\infty)\times\T$},\\
		&\text{subject to}\quad \psi_{2,n,N}(0)=\psi_{2,0},
	\end{split}\right.\end{equation*}
	where $V_N$ denotes our existing truncation of $V$ from \eqref{V_N}.
	This is a standard SPDE with Lipschitz-continuous coefficient as in \cite{wal86},
	and hence has a strong solution on any probability space that supports
	two independent copies $\dot{W}_1$ and $\dot{W}_2$ of a space-time white
	noise. Because $f_n^2+g_n^2=1$,
	\[
		g_n\left( \psi_1-\psi_{2,n,N}\right)\dot{W}_1
		+ f_n\left( \psi_1-\psi_{2,n,N}\right) \dot{W}_2
	\]
	defines a space-time white noise; see Corollary \ref{cor:WN} of the appendix.
	Therefore, $\psi_{2,n,N}$ has the same law as $\psi_{2,N}$, started at
	$\psi_{2,0}$, where $\psi_{2,N}$ denotes the solution to \eqref{eq:RD} with $V$ replaced by
	$V_N$. The proof of Theorem \ref{th:exist-unique} shows that
	there exist stopping times $T_1,T_2,\ldots$ (depending on $n$) such that
	$\lim_{N\to\infty}T_N=\infty$ a.s.\ and 
	$\psi_{2,n,N}(t)=\psi_{2,n,N+1}(t)$ for all $t\in[0\,,T_N]$.
	In this way, we obtain a predictable random field $\psi_{2,n}$ such that
	$\psi_{2,n}(t)=\psi_{2,n,N}(t)$ for all $t\in[0\,,T_N]$,
	and $\psi_{2,n}$ is the strong solution to the SPDE,
	\begin{equation*}\left[\begin{split}
		&\partial_t \psi_{2,n}=\partial^2_x \psi_{2,n}  + 
			V(\psi_{2,n})  + \lambda \sigma(\psi_{2,n})  
			\left\{ g_n\left( \psi_1-\psi_{2,n}\right)\dot{W}_1
			+ f_n\left( \psi_1-\psi_{2,n}\right) \dot{W}_2\right\}\ 
			\text{on $(0\,,\infty)\times\T$},\\
		&\text{subject to}\quad \psi_{2,n}(0)=\psi_{2,0}.
	\end{split}\right.\end{equation*}
	Once again, Theorem \ref{th:exist-unique} ensures that
	this SPDE can be solved on any probability space that supports 
	$(\dot{W}_1\,,\dot{W}_2)$. 
	
	Since 
	\begin{equation}\label{W}
		\dot{w}_n := g_n\left( \psi_1-\psi_{2,n}\right)\dot{W}_1
		+ f_n\left( \psi_1-\psi_{2,n}\right) \dot{W}_2
	\end{equation}
	is a space-time white noise [Corollary \ref{cor:WN}], Theorem \ref{th:exist-unique}
	ensures that the law of $\psi_{2,n}$ is the same as the law of
	$\psi_2$, any solution to \eqref{eq:RD} started at $\psi_{2,0}$,
	and is in particular does not depend on $n\in\N$.
	
	Next, we use Proposition \ref{pr:tight} to see that 
	\[
		\sup_{t\ge 0}\E\left( \left\| \psi_1(t)
		\right\|^k_{C^\alpha(\T)}\right)<\infty\qquad\text{for 
		every $k\ge2$}.
	\]
	Propositions \ref{pr:tight} and \ref{pr:temporal:cont}, and a  chaining argument
	together imply that for every $t_0>0$,
	\[
		\E\left( \sup_{t\in(0,t_0)}\left\| \psi_1(t)
		\right\|^k_{C^\alpha(\T)}\right)<\infty\qquad\text{for 
		every $k\ge2$}.
	\]
	Let $\psi_2$ denote any solution to \eqref{eq:RD} starting at $\psi_{2,0}$.
	Since every $\psi_{2,n}$ has the same law as
	$\psi_2$, an appeal to the Arzel\`a-Ascoli theorem
	(see the proof of Proposition \ref{pr:tight} for details) shows that
	the random fields $[0\,,t_0]\ni t\mapsto \psi_{2,n}(t)$ --
	as $n$ varies in $\N$ -- are tight in the space $C([0\,,t_0]\times\T)$.
	Therefore, the laws of the vector-valued random fields
	\begin{equation}\label{DV}
		[0\,,t_0]\ni t\mapsto \left( \psi_1(t)\,,
		\psi_{2,n}(t)\,,\mathcal{S}(t)\,,\mathcal{T}_n(t)\right)
	\end{equation}
	are tight in $\mathscr{C} := C([0\,,t_0];\,C(\T\,;\R^4))$
	as $n$ roams over $\N$, where $\dot{w}_n$ is the white noise
	of \eqref{W}, and
	\begin{align*}
		\mathcal{S}(t\,,x) &:= \int_{(0,t)\times\T} p_{t-s}(x\,,y)
			\sigma(\psi_1(s\,,y))\,\dot{W}_1(\d s\,\d y),\\
		\mathcal{T}_n(t\,,x) &:= \int_{(0,t)\times\T} p_{t-s}(x\,,y)
			\sigma(\psi_{2,n}(s\,,y))\,\dot{w}_n(\d s\,\d y),
	\end{align*}
	are the stochastic integrals used to define $\psi_1$ and $\psi_{2,n}$
	in their mild form. Because tight probability laws on $\mathscr{C}$ 
	have weak subsequential limits, \eqref{DV}
	has a subsequence (as $n\to\infty$) that converges weakly to a vector-value random field\footnote{%
		The notation is deliberately slightly inconsistent, as the recently derived random field
		$\psi_1$ is not defined on the same probability space that the earlier-defined
		$\psi_1$ was. It does, however, have the same law of course.}
	\[
		[0\,,t_0]\ni t\mapsto \left( \psi_1(t)\,,
		\psi_2(t)\,,\mathcal{S}(t)\,,\mathcal{T}(t)\right),
	\]
	and
	\begin{align*}
		\mathcal{S}(t\,,x) &:= \int_{(0,t)\times\T} p_{t-s}(x\,,y)
			\sigma(\psi_1(s\,,y))\,\dot{W}_1(\d s\,\d y),\\
		\mathcal{T}(t\,,x) &:= \int_{(0,t)\times\T} p_{t-s}(x\,,y)
			\sigma(\psi_{2}(s\,,y))\,\dot{w}(\d s\,\d y),
	\end{align*}
	for a space-time white noise $w$. And we obtain, additionally, that
	this new pair $(\psi_1\,,\psi_2)$ solve respectively
	\eqref{eq:psi_1} and \eqref{eq:psi_2} for $t\in(0\,,t_0)$. This proves
	the existence assertion of the proposition.
	
	The above construction also readily yields part 1 of the proposition
	since any subsequential limit of $\psi_{2,n}$'s, as $n\to\infty$ has the
	same law as $\psi_2$ because each $\psi_{2,n}$ does.
	
	In the case that $V$ were replaced by a Lipschitz-continuous and bounded function,
	\citet[Lemma 3.1]{Mueller-Coupling} includes part 2. In other words, 
	the latter result shows that $\psi_1\ge\psi_{2,n,N}$. Let $N$ and
	$n$ tend to infinity -- as we did previously -- in order 
	to deduce part 2. 
	
	Finally, we observe that $(\psi_1\,,\psi_{2,n,N})$
	is a Feller process for the very same reasons that $\psi_1$ and
	$\psi_2$ are individually Feller.
	Moreover, the estimates required for the Feller property can all be made 
	to hold uniformly in $(n\,,N)$; see the proof of Proposition \ref{pr:feller}.
	Let $n,N\to\infty$ as above to deduce part 3 and hence the proposition.
\end{proof}

The following is the second, and final, result of this section about PM couplings. After
this, we shall move on to describe the fourth [and final] example of couplings for SPDEs,
which is our AM coupling.

\begin{lemma}\label{lem:pm coupling}
	Choose and fix non-random numbers $C_0>c_0>0$ and $\alpha,\varepsilon\in(0\,,1/2)$.
	Also consider two
	functions $\psi_{1,0},\psi_{2,0}\in C_+^\alpha(\T)$ such that
	$\psi_{2,0}\le\psi_{1,0}$,
	$\max_{i\in\{1,2\}}\|\psi_{i,0}\|_{C^\alpha(\T)} \le C_0$,
	and $\inf_{x\in\T}\psi_{2,0}(x) \ge c_0.$
	Let $(\psi_1\,,\psi_2)$ denote a PM coupling of 
	two solutions to \eqref{eq:RD} with respective
	initial profiles $\psi_{1,0}$ and $\psi_{2,0}$, and consider
	the stopping time,
	\[
		\tau:= \inf\left\{ s>0:\, \psi_1(s) = \psi_2(s) \right\}
		\qquad[\inf\varnothing:=\infty].
	\]
	Then there exists non-random numbers $t_1,\delta_1\in(0\,,1)$
	-- depending only on $(c_0\,,C_0\,,\alpha)$ -- such that
	\[
		 \P\left\{ \psi_1(\tau+s) = \psi_2(\tau+s)
		\text{ for all $s\ge0$} \text{ and }\tau\le t_1\right\} \ge 1-\varepsilon,
	\]
	provided that $\|\psi_{1,0}-\psi_{2,0}\|_{L^1(\T)}\le\delta_1$.
\end{lemma}

Let us make two brief remarks first. We will prove the lemma afterward.

\begin{remark}
	In the context of PM couplings, when we say that $\tau$ is a stopping time
	we mean that $\tau$ is a stopping time with respect to the filtration 
	$\{\mathscr{G}_t\}_{t\ge0}$ generated
	by the underlying two noise $\dot{\mathcal{W}}_1$ and $\dot{\mathcal{W}}_2$
	used in the PM coupling. That is, for every $t\ge0$,
	we first let $\mathscr{G}_t$ define the sigma algebra
	generated by all random variables of the form
	$\int_{(0,t)\times\T}\phi(s\,,x)\,\mathcal{W}_i(\d s\,\d x)$ as $i$ ranges over
	$\{1\,,2\}$ and $\phi$ roams over non-random elements of $L^2([0\,,t]\times\T)$.
	This forms a filtration $\{\mathscr{G}_t\}_{t\ge0}$, which we then augment by making
	it  right continuous, and then by completing the sigma algebra $\mathscr{G}_0$ 
	with respect to the measure $\P$, as is done in  martingale theory.
\end{remark}

\begin{definition}
	Let $(\psi_1\,,\psi_2)$ denote a given coupling, using any coupling method,
	of the solutions to \eqref{eq:RD} with respective initial profiles $\psi_{1,0}$ and
	$\psi_{2,0}$. Also, choose and fix some number $t>0$. 
	We say that the coupling $(\psi_1\,,\psi_2)$ is \emph{successful by time $t$}
	if  $\psi_1(t+s) = \psi_2(t+s)\text{ for all $s\ge0$}.$
\end{definition}

We are ready to prove Lemma \ref{lem:pm coupling}.

\begin{proof}
	Define
	\[
		\Delta(t\,,x) := \psi_1(t\,,x) - \psi_2(t\,,x)\qquad
		\text{for all $t\ge0$ and $x\in\T$}.
	\]
	The same reasoning that led to eq.\ (3.4) of \cite{Mueller-Coupling}
	leads us to the assertion that $\Delta$ solves the SPDE
	\begin{equation*}\left[\begin{split}
		&\partial_t\Delta =\partial^2_x\Delta + V(\psi_1) - 
			V(\psi_2)
			+ \lambda\left[ \left| \sigma(\psi_1) - \sigma(\psi_2) \right|^2
			+ 2\sigma(\psi_1)\sigma(\psi_2)
			\frac{f^2(\Delta)}{1+g(\Delta)} \right]^{1/2}\dot{F},\\
		&\text{subject to }\Delta(0)=\psi_{1,0}-\psi_{2,0},
	\end{split}\right.\end{equation*}
	where $\dot{F}$ is a space-time white noise, and $f$ and $g$
	were defined in \eqref{fg}. 
	Since $\psi_{1,0}\ge\psi_{2,0}$, part 2 of Proposition \ref{pr:coupling} ensures that
	$\Delta\ge0$ a.s., and hence
	$\| \psi_1(t) - \psi_2(t) \|_{L^1(\T)}
	=\int_{\T} \Delta(t\,,x)\,\d x$ for all $t\ge0$.	The main portion of the proof is to
	demonstrate that 
	\begin{equation}\label{goal:coupling}
		\P\{\tau>t\}=\P\left\{\inf_{s\in(0,t)}X(s)>0\right\}<1
		\quad\text{for all sufficiently small $(t\,,\delta)\in(0\,,\infty)^2$,}
	\end{equation}
	where
	\[	
		X(t) := \int_{\T} \Delta(t\,,x)\,\d x
		\qquad[t\ge0].
	\]
	Since $\psi_1$ and $\psi_2$ are continuous random
	fields, this proves that $\P\{\tau> t\} <1$ for $(t\,,\delta)$ small,
	which is the more challenging portion of this proof. Once we prove this,
	we will easily complete the proof of the remainder of the proposition at the end.

	Lemma 3.2 of \cite{Mueller-Coupling} [with our $X$
	playing the role of that lemma's $U$] implies that
	\begin{equation}\label{XCM}
		X(t) = \delta + \int_0^t C(s)\,\d s +  M(t)\qquad\text{for all $t>0$},
	\end{equation}
	where $\delta = \| \psi_{1,0} - \psi_{2,0} \|_{L^1(\T)} \in (0\,,1)$, 
	\[
		C(t) = \int_{\T} 
		\left[V(\psi_1(t\,,x)) - V(\psi_2(t\,,x)) \right]\d x
		= X(t) - \int_{\T} 
		\left[F(\psi_1(t\,,x)) - F(\psi_2(t\,,x)) \right]\d x,
	\]
	and $M=\{M(t)\}_{t\ge0}$ is a continuous $L^2(P)$-martingale
	with quadratic variation,
	\begin{align*}
		\< M\>(t) &= \lambda^2 \int_0^t\d s\int_{\T}\d x\left[
			\left( \sigma(\psi_1(s\,,x)) - \sigma(\psi_2(s\,,x))\right)^2
			+ 2 \sigma(\psi_1(s\,,x))\sigma(\psi_2(s\,,x))
				\frac{f^2(\Delta(s\,,x))}{1 + g(\Delta(s\,,x))}\right],
	\end{align*}
	for every $t>0$. In particular, we use the facts that:
	(a) $g(z)\le 1$ for all $z\in\R$; and (b) $\psi_2\le\psi_1$ [Proposition
	\ref{pr:coupling}] in order to see that a.s.\ for all $t>0$,
	\begin{equation}\label{<M>:coupling}\begin{split}
		\frac{\d\<M\>(t)}{\d t} &\ge 2\lambda^2 {\rm L}_\sigma^2 \adjustlimits\inf_{r\in(0,t)}\inf_{y\in\T}
			\left| \psi_2(r\,,y)\right|^2 \int_{\T} 
			\frac{f^2(\Delta(s\,,x))}{1 + g(\Delta(s\,,x))}\,\d x\\
		&\ge \lambda^2 {\rm L}_\sigma^2 \adjustlimits\inf_{r\in(0,t)}\inf_{y\in\T}
			\left| \psi_2(r\,,y)\right|^2
			\int_{\T}\min\left\{ \Delta(s\,,x)\,,1\right\}\d x.
	\end{split}\end{equation}
	
	Because of \textbf{(F1)}, $F$ is non decreasing. Therefore, we can infer that
	$C(t)\le X(t)$ a.s.\
	for all $t\ge0$, and hence \eqref{XCM} ensures that $X$ satisfies the stochastic differential inequality,
	\[
		\d X(t) \le X(t)\,\d t + \d M(t).
	\]

	Next, we choose and fix some $\varepsilon\in(0\,, \frac12\{c_0\wedge (1-\delta)\})$, 
	and consider the stopping time,
	\[
		\mathcal{H} := \min_{i\in\{1,2\}} \inf\left\{ s>0:\, \sup_{x\in\T}
		\left| \psi_i(s\,,x) -\psi_{i,0}(x)\right| \ge \varepsilon\right\},
	\]
	where $\inf\varnothing:=\infty$. 
	
	For every $t>0$, the following holds almost surely on $\{\mathcal{H}>t\}$:
	\[
		\adjustlimits\sup_{s\in(0,t)}
		\sup_{x\in\T}\Delta(s\,,x)<  \delta+2\varepsilon<1
		\text{ and }
		\adjustlimits\inf_{r\in(0,t)}\inf_{y\in\T}\psi_2(r\,,y)> c_0-\varepsilon > \frac{c_0}{2}>0.
	\]
	Therefore, \eqref{<M>:coupling} implies that
	\[
		\frac{\d\< X\>(t)}{\d t} =
		\frac{\d\< M\>(t)}{\d t} \ge \left( \frac{\lambda{\rm L}_\sigma c_0}{4}
		\right)^2 X(t) 
		\qquad\text{a.s.\ on $\{\mathcal{H}>t\}$},
	\]
	for every non-random real number $t>0$. In particular,
	\[
		\int_0^t\e^{-s}\,\frac{\d\<X\>(s)}{X(s)}
		\ge \left( \frac{\lambda{\rm L}_\sigma c_0\sqrt{1-\e^{-t}}}{4}
		\right)^2\qquad\text{a.s.\ on $\{\mathcal{H}>t\}$}.
	\]
	Combine these facts together with Proposition \ref{pr:SDI}, and recall
	that $X(0)=\delta$, in order
	to see that
	\begin{align*}
		\P\left\{ \tau>t \,,\, \mathcal{H}>t\right\}
			&\le \P\left\{ \inf_{s\in(0,t)} X(s)>0 \,,\,
			\int_0^t\e^{-s}\,\frac{\d\<X\>(s)}{X(s)}
			\ge \left( \frac{\lambda{\rm L}_\sigma c_0 t \sqrt{1-\e^{-t}}}{4}
			\right)^2\right\}\\
		&\le 2\P\left\{|\mathcal{Z}| < \frac{8\sqrt{\delta}}{\lambda{\rm L}_\sigma
			c_0\sqrt{1-\e^{-t}}}\right\}
			\le \frac{16\sqrt{\delta}}{\lambda{\rm L}_\sigma
			c_0\sqrt{t}},
	\end{align*}
	where $\mathcal{Z}$ has a standard
	normal distribution. We have appealed to the simple bound $1-\exp(-t)\ge t$ and the fact that the probability density
	function of $\mathcal{Z}$ is at most $(2\pi)^{-1/2}<1/2$ for the last inequality.
	
	Next, we observe that
	\[
		\P\{\mathcal{H} \le t\} \le \sum_{i=1}^2\P\left\{ \adjustlimits\sup_{s\in(0,t)}\sup_{x\in\T}
		|\psi_i(s\,,x) - \psi_{i,0}(x) | \ge \varepsilon\right\}.
	\]
	Recall that, for a standard 1-D Brownian motion $\beta$, $i\in\{1\,,2\}$, and $s>0$,
	\[
		\| \mathcal{P}_s\psi_{i,0} - \psi_{i,0}\|_{C(\T)} 
		\le \sup_{x\in\T}\E\left| \psi_{i,0}(\beta(s)+x) - \psi_{i,0}(x)\right|
		\le \|\psi_{i,0}\|_{C^\alpha(\T)}\E\left( |\beta(s)|^\alpha\right)
		\le KC_0 s^{\alpha/2},
	\]
	where $K=\E(|\beta(1)|^\alpha)= 2^{(1+\alpha)/2}\pi^{-1/2}\Gamma((2+\alpha)/2)$.
	Therefore, Proposition \ref{pr:temporal:cont} and Chebyshev's inequality
	together imply that  for every $k\ge2$ there exists a real number
	$K_1=K_1(k\,,\alpha\,,A)>0$ such that
	$\P\{ \mathcal{H}\le t\} \le K_1 t^{\alpha k/2}$ for all $t\in[0\,,1]$. 
	Choose $k$ large enough to ensure that
	$\P\{\mathcal{H}\le t\}\le K_1\sqrt t$ and hence
	\[
		\inf_{t\in(0,1)}
		\P\left\{ \tau >t \right\}
		\le K_2\inf_{t\in(0,1)}\left( \sqrt{\frac{\delta}{ t}} + \sqrt t\right)
		\le 2K_2\delta^{1/4},
	\]
	where $K_2>0$ does not depend on $\delta\in(0\,,1)$, even though it 
	might depend on $(c_0\,,C_0\,,\alpha)$.
	This implies \eqref{goal:coupling}. In order to complete the proof, it
	suffices to show that
	\[
		\psi_1(\tau+s\,,x)=\psi_2(\tau+s\,,x)
		\text{ for all $s>0$ and $x\in\T$, almost surely on $\{\tau<\infty\}$}.
	\]
	Equivalently, it remains to prove that
	\[
		X(\tau+t) =0
		\text{ for all $t>0$, almost surely on $\{\tau<\infty\}$}.
	\]
	Define $Y(t) := \exp(-t)X(t)$ and apply It\^o's formula to \eqref{XCM}
	in order to see that $Y$ is a continuous, non-negative supermartingale.
	Since $\tau$ denotes the first time $Y$ hits zero, it follows from
	a classical exercise in elementary martingale theory that
	\[
		Y(\tau+t)=0
		\text{ for all $t>0$, almost surely on $\{\tau<\infty\}$}.
	\]
	Because $Y(s)=0$ iff $X(s)=0$ for any and every $s\ge0$, this has the desired effect.
\end{proof}

\noindent\textbf{\itshape (iv) Anchored monotone {\bf (AM)} coupling.}
The AM coupling is a more attractive variation of the PM coupling, where
the qualifier ``attractive'' is used in the same vein as it is used in particle systems.

Before we describe the AM coupling, let us mention the following ``attractive''
property of the AM coupling.

\begin{lemma}\label{lem:am coupling}
	Choose and fix non-random numbers $C_0>c_0>0$ and $\alpha,\varepsilon\in(0\,,1/2)$,
	and consider $\psi_{1,0},\psi_{2,0}\in C_+^\alpha(\T)$ such that
	$\max_{i\in\{1,2\}}\|\psi_{i,0}\|_{C^\alpha(\T)} \le C_0$,
	and $\min_{i\in\{1,2\}}\inf_{x\in\T}\psi_{i,0}(x) \ge c_0.$
	Let $(\psi_1\,,\psi_2)$ denote an AM coupling of 
	two solutions to \eqref{eq:RD} with respective
	initial profiles $\psi_{1,0}$ and $\psi_{2,0}$, and consider
	the stopping time,
	$\tau:= \{ s>0:\, \psi_1(s) = \psi_2(s) \}$,
	where $\inf\varnothing:=\infty$.
	Then, for the same numbers $t_1,\delta_1\in(0\,,1)$
	that arise in Lemma \ref{lem:pm coupling}, 
	\[
		 \P\left\{ \psi_1(\tau+s) = \psi_2(\tau+s)
		\text{ for all $s\ge0$} \text{ and }\tau\le t_1\right\} \ge 1-2\varepsilon,
	\]
	provided that $\|\psi_{1,0}-\psi_{2,0}\|_{L^1(\T)}\le\delta_1/2$.
\end{lemma}

This is exactly the same assertion as the one in Lemma \ref{lem:pm coupling},
except we no longer need to assume that  $\psi_{1,0}\le\psi_{2,0}$. Next, we describe 
the AM coupling which accomplishes this generalization.  
Lemma \ref{lem:am coupling} will be an immediate consequence
of that description.

Suppose $\psi_{1,0},\psi_{2,0}\in C_+(\T)$ are fixed non-random initial profiles.
Let us introduce
three independent space-time white noises $\dot{\mathcal{W}}$, 
$\dot{\mathcal{W}}_1$, and $\dot{\mathcal{W}}_2$, and 
let $\psi$ denote the solution to the SPDE,
\[
	\partial_t \psi = \partial^2_x\psi + V(\psi) + \sigma(\psi)
	\dot{\mathcal{W}},\quad\text{subject to}\quad \psi(0)=
	\psi_0 := \psi_{1,0}\vee\psi_{2,0}.
\]
Then, we use $\dot{\mathcal{W}}$ and $\dot{\mathcal{W}}_i$ [$i=1,2$] 
to construct a PM coupling $(\psi\,,\psi_i)$, where the initial distribution of
$\psi_i$ is $\psi_{i,0}$. We refer to this construction of $(\psi_1\,,\psi_2)$
as an \emph{AM coupling} of the solutions to \eqref{eq:RD} with
respect initial data $(\psi_{1,0}\,,\psi_{2,0})$,
and to $\psi$ as the \emph{anchor} process for $\psi_1$ and $\psi_2$. 
The following is a ready consequence
of the proof of Proposition \ref{pr:feller}.

\begin{lemma}\label{lem:am-feller}
	Let $(\psi_1\,,\psi_2)$ denote an AM coupling of the solutions to
	the SPDE \eqref{eq:RD} with respective initial profiles $\psi_{1,0},\psi_{2,0}
	\in \cup_{\alpha\in(0,1/2)}C^\alpha_+(\T)$, and let $\psi$ denote
	the associated process. Then,
	$\{(\psi(t)\,,\psi_1(t)\,,\psi_2(t))\}_{t\ge0}$
	is a Feller process with values in $C(\T\,;\R^3)$.
\end{lemma}

In order to be guaranteed that we can perform this
construction, we might of course have to enlarge the underlying probability space;
see Proposition \ref{pr:coupling}. It follows immediately from Proposition
\ref{pr:coupling} that the marginal laws of $\psi_i$
is the same as the law of the solution to the SPDE \eqref{eq:RD} starting at $\psi_{i,0}$.
So this produces a coupling indeed. And Lemma \ref{lem:am-feller} follows from
the method of proof of Proposition \ref{pr:feller}. We skip the details of the proof,
and merely refer to the comments made about the proof of part 3 of Proposition
\ref{pr:coupling}.

Now we prove Lemma \ref{lem:am coupling}.

\begin{proof}[Proof of Lemma \ref{lem:am coupling}]
	This is basically the argument that appears at the very last portion of the paper by
	\cite{Mueller-Coupling}; see the paragraphs surrounding eq.\
	therein ({\it ibid.}). We repeat the proof here for the convenience
	of the reader.
	
	Proposition \ref{pr:coupling}
	insures that $\psi\ge \max\{\psi_1\,,\psi_2\}$ a.s., and Lemma \ref{lem:pm coupling}
	ensures that the coupling $(\psi\,,\psi_i)$ is successful for either choice
	of $i\in\{1\,,2\}$, with probability $\ge1-\varepsilon$, by the same time $t_1\in(0\,,1)$
	as was given in Lemma \ref{lem:pm coupling}, 
	provided that the condition $\|\psi_0-\psi_{i,0}\|_{L^1(\T)}\le\delta_1$ is met for
	either $i\in\{1\,,2\}$. Because $\|\psi_{1,0}-\psi_{2,0}\|_{L^1(\T)}\le \delta_1/2$, 
	we find that $\|\psi_0-\psi_{i,0}\|_{L^1(\T)}\le\delta_1$ for both $i=1,2$, in fact.
	Thus, it follows that 
	\[
		\P\left\{ \text{the PM coupling of $(\psi\,,\psi_i)$ is successful by time $t_1$}\right\}
		\ge 1-\varepsilon,
	\]
	for both $i=1,2$. If the coupling of $(\psi\,,\psi_1)$ is successful by time $t_1$
	and the coupling of $(\psi\,,\psi_2)$ is successful by time $t_1$, then certainly
	the AM coupling of $(\psi_1\,,\psi_2)$ is successful by time $t_1$.
	Therefore, 
	\begin{align*}
		&\P\left\{\text{the AM coupling of $(\psi_1\,,\psi_2)$ is not successful by time $t_1$}\right\}\\
		&\hskip1in\le \sum_{i=1}^2\P\left\{
			\text{the PM coupling of $(\psi\,,\psi_i)$ is not successful by time $t_1$}\right\}
			\le 2\varepsilon.
	\end{align*}
	Thanks to the strong Markov property of $(\psi\,,\psi_1\,,\psi_2)$
	[Lemma \ref{lem:am-feller}] if the above couplings are successful then the first 
	time to succeed is a stopping time. This completes the proof.
\end{proof}

\section{Uniqueness of a non-trivial invariant measure via coupling}

The main result of this section is the following uniqueness result.

\begin{theorem}\label{th:inv-unique}
	If ${\rm L}_\sigma>0$ and $\lambda$ is small enough 
	to ensure that the conclusion of Proposition \ref{pr:inv:meas:exists}
	is valid, then there is at most one invariant
	measure $\mu$ for \eqref{eq:RD} that satisfies $\mu\{\mathbb{0}\}=0$.
	That measure is $\mu_+$ of Proposition \ref{pr:inv:meas:exists}.
\end{theorem}

We shall combine coupling ideas from the previous section in order to prove Theorem
\ref{th:inv-unique}. As first step in that proof, we offer the following technical result,
valid for any coupling method, including those that are possibly not mentioned in this paper.

\begin{lemma}\label{lem:strip}
	Choose and fix two non-random functions $\psi_{1,0},\psi_{2,0}\in C_{>0}(\T)$,
	and let $(\psi_1\,,\psi_2)$ denote any coupling of 
	two solutions to \eqref{eq:RD} with respective
	initial profiles $\psi_{1,0}$ and $\psi_{2,0}$. Choose an arbtirary non-random number
	$q>0$, and let $\mathcal{T}_0$ denote
	any a.s.-finite stopping time with respect to the underlying noises of $(\psi_1\,,\psi_2)$.
	Then, there exist non-random numbers $C_0>c_0>0$ such that the stopping times
	\[
		\mathcal{T}_n := \inf\left\{ s> \mathcal{T}_{n-1} + q:\, 
		\max_{i\in\{1,2\}}\|\psi_i(s)\|_{C^\alpha(\T)} \le C_0
		\text{\ \ or }\min_{i\in\{1,2\}}\inf_{x\in\T}\psi_i(s\,,x) \ge c_0\right\}
	\]
	are a.s.\ finite for every $n\in\N$. The constants $c_0$ and $C_0$ do not depend
	on the particular coupling method used.
\end{lemma}

\begin{proof}
	Since $\psi_1$ and $\psi_2$ are continuous random fields
	[see Theorem \ref{th:exist-unique}],
	the random mappings $t\mapsto\min_{i\in\{1,2\}}\psi_i(t\,,x)$ and $t\mapsto\max_{i\in\{1,2\}}
	\|\psi_i(t)\|_{C^\alpha(\T)}$ define continuous and adapted processes. This proves that
	every $\mathcal{T}_n$ is indeed a stopping time. We now prove the more
	interesting statement that these $\mathcal{T}_n$'s
	are a.s.\ finite for suitable non-random choices of $C_0\gg1$ and $c_0\ll1$ that do not depend
	on the particular details of the coupling.
	
	Recall that our proof of \eqref{claim:DZ-tight} hinged on proving that if $\psi_0=\mathbb{1}$
	and $\psi$ solves \eqref{eq:RD} starting from $\psi_0$ then 
	\[
		\adjustlimits\lim_{c\downarrow0}\liminf_{T\to\infty}\frac1T
		\int_0^T\bm{1}_{\{ c \le \inf_{x\in\T} \psi(t,x) \le \|\psi(t)\|_{C^\alpha(\T)}
		\le 1/c\}}\,\d t=1\qquad\text{a.s.}
	\]
	See the paragraphs that follow Proposition \ref{pr:inv:meas:exists},
	as well as \eqref{eq:condition-3}.
	A brief inspection of the random walk argument shows that the same fact holds for
	every $\psi_0\in C_{>0}(\T)$ [not just $\psi_0\equiv1$].\footnote{
		Indeed, the fact that Proposition \ref{pr:tight} holds equally well for non-constant
		initial data implies that the same proof that follows Proposition \ref{pr:inv:meas:exists}
		goes to show that, in the present setting,
		\[
			\adjustlimits\lim_{c\downarrow0}\limsup_{T\to\infty}\frac1T
			\int_0^T\bm{1}_{\{\|\psi(t)\|_{C^\alpha(\T)}
			> 1/c\}}\,\d t=0\qquad\text{a.s.}
		\]
		It therefore remains to prove that
		\begin{equation}\label{gee}
			\adjustlimits\lim_{c\downarrow0}\limsup_{T\to\infty}\frac1T
			\int_0^T\bm{1}_{\{\inf_{x\in\T}\psi(t)
			<c\}}\,\d t=0\qquad\text{a.s.}
		\end{equation}
		Let $a := \inf_{x\in\T}\psi_0(x).$ Since $a>0$, we can appeal to Lemma 
		\ref{lem:comparison} and compare $\psi$ to the solution of \eqref{eq:RD}
		in order to reduce the problem to proving the above in the case that
		$\psi_0 = a\mathbb{1}$. Now, $a^{-1}\psi$ solves \eqref{eq:RD}
		but with $\Theta:=(\sigma\,,F)$ replaced by $\Theta_a:=(a^{-1}\sigma(a\cdot)\,,
		a^{-1}F(a\cdot))$, starting from $\mathbb{1}$. 
		This proves \eqref{gee} since $\Theta_a$ has the same analytic properties that
		we required of $\Theta$.
	} 
	In particular,
	we can find non-random numbers $0<c_i<C_i$ such that
	\[
		\liminf_{T\to\infty}\frac1T
		\int_0^T\bm{1}_{\{ c_i \le \inf_{x\in\T} \psi_i(t,x) \le \|\psi_i(t)\|_{C^\alpha(\T)}
		\le C_i\}}\,\d t\ge \frac{3}{4}\qquad\text{a.s.\
		for $i\in\{1\,,2\}$}.
	\]
	Clearly, $(c_1\,,c_2\,,C_1\,,C_2)$ depend only on 
	the marginal laws of $\psi_1$ and $\psi_2$. Therefore, $(c_1\,,c_2\,,C_1\,,C_2)$
	does not depend
	on how $\psi_1$ and $\psi_2$ are coupled.
	
	Define
	\[
		\mathcal{E}_i := \left\{ t\ge0:\, c_i \le \inf_{x\in\T} \psi_i(t,x) \le \|\psi_i(t)\|_{C^\alpha(\T)}
		\le C_i\right\},
	\]
	and let $m_T$ denote  the measure defined by
	\[
		m_T(F) := \frac1T\int_0^T \bm{1}_F(t)\,\d t\qquad\text{for all $T>0$
		and Borel sets $F\subset\R_+$}.
	\]
	Since $m_T([\mathcal{E}_1\cap\mathcal{E}_2]^c)\le m_T(\mathcal{E}_1^c)
	+ m_T(\mathcal{E}_2^c)$, it follows that
	$\liminf_{T\to\infty} m_T(\mathcal{E}_1\cap\mathcal{E}_2) \ge 1/2>0$
	whence it follows that $\mathcal{E}_1\cap\mathcal{E}_2$ is unbounded
	with probability one. The a.s.-finiteness of the $\mathcal{T}_n$'s is
	now immediate.
\end{proof}

We now use Lemma \ref{lem:strip} as a ``regeneration result,'' in order to
prove the main step in the proof of Theorem \ref{th:inv-unique}.
In anticipation of future potential applications, we record that regeneration result as the
following theorem.

\begin{theorem}\label{th:successful-coupling}
	Choose and fix two non-random functions $\psi_{1,0},\psi_{2,0}
	\in C_+(\T)\setminus\{\mathbb{0}\}$. Then there exists a successful coupling
	$(\psi_1\,,\psi_2)$ of solutions to \eqref{eq:RD} with respective
	initial profiles $(\psi_{1,0}\,,\psi_{2,0})$.
\end{theorem}

\begin{proof}
	We first prove the theorem in the case that $\psi_{1,0},\psi_{2,0}
	\in C_{>0}(\T)$, a condition which we assume until further notice.
	Throughout, we choose and fix some $\alpha,\varepsilon\in(0\,,1/2)$;
	to be concrete, 
	\[
		\alpha=\varepsilon=\frac14.
	\]

	We build a hybrid coupling that makes appeals to natural coupling,
	independent coupling, and AM coupling. Our coupling is performed
	inductively, and in stages. 
	
	Throughout, let us choose and fix
	the non-random numbers $0<c_0<C_0$ --
	depending on $(\psi_{1,0}\,,\psi_{2,0})$ --
	whose existence (and properties)
	are guaranteed by Lemma \ref{lem:strip}. It should be clear from Lemma
	\ref{lem:strip} that the assertion of that lemma continues to remain valid 
	if instead of $C_0$ we choose a larger number. Therefore, we increase
	$C_0$ once and for all, if need be, in order to ensure additionally that
	\begin{equation}\label{c0C0}
		C_0 > 3c_0.
	\end{equation}
	
	We begin with the natural coupling of $(\psi_1\,,\psi_2)$ until stopping time,
	\[
		\mathcal{T}_1 := \inf\left\{ s>0:\,
		\max_{1\in\{1,2\}}\|\psi_i(s)\|_{C^\alpha(\T)} \le C_0
		\quad\text{or}\quad\adjustlimits
		\min_{i\in\{1,2\}}\inf_{x\in\T}\psi_i(s\,,x) \ge c_0
		\right\}.
	\]
	Lemma \ref{lem:strip} ensures that $\mathcal{T}_1<\infty$ a.s. This yields a coupling of
	the sort that we want, but only until time $\mathcal{T}_1$. 
	
	To extend our coupling
	beyond time $\mathcal{T}_1$ we first apply Proposition \ref{pr:support}
	with $A := C_0$ and $A_0 := 2c_0$
	in order to obtain a non-random strictly positive
	number  $t_0$. Then, starting 
	from $(\psi_1(\mathcal{T}_1)\,,\psi_2(\mathcal{T}_1))$ we run an independent coupling
	for $t_0$ units of time. By the strong Markov property,
	this yields a coupling of the sort that we want until stopping time $\mathcal{T}_1+t_0$.
	
	In order to continue our construction beyond time $\mathcal{T}_1+t_0$,
	we first let $t_1$ denote the number that was defined in Lemma \ref{lem:pm coupling}.
	Then, conditionally independently from the construction so far, we run an AM coupling
	starting from $(\psi_1(\mathcal{T}_1+t_0)\,,\psi_2(\mathcal{T}_1+t_0))$ for
	$t_1$ units of time. By the strong Markov property, this yields a coupling that we want
	until stopping time $\mathcal{T}_1+ t_0+t_1$. Thanks to Lemma \ref{lem:am coupling}
	if the two processes have merged some time in $(\mathcal{T}_1+t_0\,,\mathcal{T}_1+t_0+t_1)$,
	then from that time until time $\mathcal{T}_1+t_0+t_1$ they are equal. In this case, we
	just continue running our AM coupling to see that we have a successful coupling, as desired. If
	the two processes have not merged by time $\mathcal{T}_1+t_0+t_1$, then we continue
	our AM coupling until time
	\[
		\mathcal{T}_2 := \inf\left\{ s>\mathcal{T}_1+t_0+t_1:\,
		\max_{1\in\{1,2\}}\|\psi_i(s)\|_{C^\alpha(\T)} \le C_0
		\quad\text{or}\quad\adjustlimits
		\min_{i\in\{1,2\}}\inf_{x\in\T}\psi_i(s\,,x) \ge c_0
		\right\}.
	\]
	Then, run an independent coupling for $t_0$ units of time, and then an AM coupling
	for another $t_1$ units of time [all conditionally independently of the past
	in order to maintain the strong Markov property]. This yields a coupling up to time
	$\mathcal{T}_2+t_0+t_1$. If the two processes have merged some time between
	$\mathcal{T}_2+t_0$ and $\mathcal{T}_2+t_0+t_1$ then continue running
	the final AM coupling {\it ad infinitum}. Lemma \ref{lem:am coupling} ensures that
	this is the desired successful coupling. Else, we continue inductively. 
	
	Choose and fix some $n\in\N$, and let $\delta_1$ by the number given by
	Lemma \ref{lem:pm coupling}. Thanks to Lemma \ref{lem:pm coupling},
	we may (and will) assume without loss of generality that 
	\begin{equation}\label{cond:delta_1}
		\delta_1 < 1 \wedge \frac{c_0}{10}.
	\end{equation}
	By the strong Markov property, 
	Proposition \ref{pr:support} ensures that, almost surely on the event
	that the coupling is not successful by time $\mathcal{T}_n$,
	the conditional probability of the event
	\[
		\mathcal{E}_n := \left\{
		\|\psi_1(\mathcal{T}_n+t_0)-\psi_2(\mathcal{T}_n+t_0)\|_{C(\T)} \le 
		\frac{\delta_1}{2}\right\}\cap \bigcap_{i=1}^2
		\left\{ \|\psi_i\|_{C^{\alpha/2}(\T)}\le C_0+1\right\}
	\]
	is at least $[\mathfrak{p}_{C_0,2c_0}(t_0\,,\alpha\,,\delta_1)]^2>0.$ In order for
	this assertion to be true, we need to know additionally that 
	$[2c_0-\delta_1\,,2c_0+\delta_1]\subset[c_0\,,C_0]$; this is so because of \eqref{c0C0}
	and \eqref{cond:delta_1}.
	
	We emphasize that $\delta_1$ is deterministic (as it did not depend on the initial condition in 
	Proposition \ref{pr:support}) as well as independent of $n$. And 
	Lemma \ref{lem:am coupling} [with $\alpha$ replaced by $\alpha/2$]
	ensures that, a.s.\ on $\mathcal{E}_n$
	the conditional probability that the coupling has succeeded by time $\mathcal{T}_n+t_0+t_1$
	given everything by time $\mathcal{T}_n+t_0$ is at least $3/4$. Thus,
	\[
		\P\left(\text{the coupling succeeds by time
		$\mathcal{T}_n+t_0+t_1$}\mid \text{no success by time $\mathcal{T}_n$}\right)
		\ge \frac{3\left[\mathfrak{p}_{C_0,2c_0}(t_0\,,\alpha\,,\delta_1)\right]^2}{4}.
	\]
	Since the right-hand side does not depend on $n$, the tower property of conditional 
	probabilities yield the following for every $n\in\N$:
	\[
		\P\left(\text{the coupling does not succeed by time
		$\mathcal{T}_n+t_0+t_1$}\right) \le\left( 1-
		\frac{3\left[\mathfrak{p}_{C_0,2c_0}(t_0\,,\alpha\,,\delta_1)\right]^2}{4}\right)^{n-1}.
	\]
	Since $\mathcal{T}_{n+1}-\mathcal{T}_n > t_0+t_1$ a.s.\ for every $n\in\N$,
	it follows that
	$\lim_{n\to\infty}\mathcal{T}_n=\infty$ a.s. Thus,
	we let $n\to\infty$ to conclude the proof, from preceding display, in the case that
	$\psi_{1,0},\psi_{2,0}\in C_{>0}(\T)$.
	
	In order to prove the general result, we first run our natural coupling for one unit of time,
	starting from $(\psi_{1,0}\,,\psi_{2,0})$. Theorem \ref{th:exist-unique} assures us
	that with probability one, $\psi_i(1\,,x)>0$ for every $x\in\T$ and $i\in\{1\,,2\}$.
	Condition on everything by time one, and run our hybrid coupling from then on,
	conditionally independently of the first one time unit, starting from
	$(\psi_1(1)\,,\psi_2(1))$. Apply the strong Markov property and the first portion of
	the proof to finish.
\end{proof}

We are in position to prove Theorem \ref{th:inv-unique}.

\begin{proof}[Proof of Theorem \ref{th:inv-unique}]
	Let $\phi\in C_+(\T)\setminus\{\mathbb{0}\}$ be non-random.
	Theorem \ref{th:successful-coupling} ensures that, after we possibly enlarging 
	the underlying probability space, we can construct a successful
	coupling $(\psi_0\,,\psi_1)$ of solutions to \eqref{eq:RD}
	that start respectively from $(\phi\,,\mathbb{1})$. In particular,
	\[
		\lim_{T\to\infty}
		\sup_{\substack{\Gamma\subset C(\T)\\\Gamma\text{ Borel}}}
		\left| \frac 1T\int_0^T \bm{1}_\Gamma(\psi_0(t))\,\d t
		- \frac 1T\int_0^T \bm{1}_\Gamma(\psi_1(t))\,\d t \right| =0\qquad\text{a.s.}
	\]
	Take expectations and appeal to the bounded convergence theorem
	in order to deduce that
	\[
		\lim_{T\to\infty}
		\sup_{\substack{\Gamma\subset C(\T)\\\Gamma\text{ Borel}}}
		\left| 
		\E\left( \frac 1T\int_0^T \bm{1}_\Gamma(\psi_0(t))\,\d t\right) 
		- \E\left( \frac 1T\int_0^T \bm{1}_\Gamma(\psi_1(t))\,\d t\right) 
		\right|=0.
	\]
	Notice that the preceding is a statement about probability laws and does not
	depend on the coupling construction that was devised in order  to prove it. Therefore,
	it is convenient to set $F:=\bm{1}_\Gamma$ and revert to the notation of Markov process theory
	(see \S\ref{subsec:Feller} and especially \S\ref{subsec:KB}), and rewrite the above in terms of the Feller semigroup
	$\{P_t\}_{t\ge0}$ associated to \eqref{eq:RD} as follows:
	\[
		\lim_{T\to\infty} \frac1T\int_0^T (\delta_\phi P_t)\,\d t 
		= \mu_+\qquad\text{in total variation.}
	\]
	Let $\mu$ denote any probability measure 
	on $C_+(\T)$ that satisfies $\mu\{\mathbb{0}\}=0$ and is invariant
	for \eqref{eq:RD}. Proposition \eqref{pr:inv:meas:exists} ensures that 
	there is at least one such measure  $\mu_+$ and that, in fact,
	$\mu_+$ concentrates on $C_{>0}(\T)$. 
	Because $\mu$ is invariant, Tonelli's theorem ensures that
	\[
		\int_{C(\T)}  \left( \frac1T\int_0^T (\delta_\phi P_t)(\Gamma)\,\d t \right)\mu(\d\phi)
		=   \frac1T\int_0^T \left(\int_{C(\T)} (\delta_\phi P_t)(\Gamma) \mu(\d\phi)\right)\d t
		= \mu(\Gamma),
	\]
	for all $T>0$ and Borel sets $\Gamma\subset C(\T)$.
	Therefore, the bounded convergence theorem yields
	\[
		\lim_{T\to\infty} \frac1T\int_0^T (\delta_{\mathbb{1}}P_t)\,\d t
		=\mu\qquad\text{in total variation.}
	\]
	Our proof of Proposition \ref{pr:inv:meas:exists} consisted of proving that
	$\mu_+$ is a subsequential weak limit of the probability measures
	$\{T^{-1}\int_0^T (\delta_{\mathbb{1}}P_t)\,\d t\}_{T>0}$. Therefore,
	the above shows that $\mu=\mu_+$, which completes the uniqueness
	of $\mu_+$.
\end{proof}

Let us conclude this section with an ergodic theorem for the solution to
\eqref{eq:RD}.

\begin{corollary}[Ergodic theorem]\label{co:ergodic}
	Suppose ${\rm L}_\sigma>0$ and $\lambda$ is small enough to ensure
	that the conclusion of Proposition \ref{pr:inv:meas:exists} is valid.
	Then, for every probability measure $\nu$ on $C_+(\T)$
	that satisfies $\nu\{\mathbb{0}\}=0$, 
	\[
		\lim_{T\to\infty}
		\frac1T \int_0^T (\nu P_t) \,\d t=
		\mu_+
		\quad\text{in total variation,}
	\]
	where  $\mu_+$ is the invariant measure produced by Proposition \ref{pr:inv:meas:exists}.
\end{corollary}

\begin{proof}
	The proof of Theorem \ref{th:inv-unique} readily implies
	that for every  $\phi\in C_+(\T)\setminus\{\mathbb{0}\}$,
	\[
		\frac1T \int_0^T  (\delta_\phi P_t)\,\d t - 
		\frac1T \int_0^T  (\delta_{\mathbb{1}} P_t)\,\d t 
		\quad\text{in total variation, as $T\to\infty$}.
	\]
	In particular, we use the above twice [once for $\phi_1$, and once for $\phi_2$,
	in place of $\phi$] in order to see that for
	every $\phi_1\,,\phi_2\in C_+(\T)\setminus\{\mathbb{0}\}$,
	\[
		\frac1T \int_0^T  (\delta_{\phi_1} P_t)\,\d t - 
		\frac1T \int_0^T  (\delta_{\phi_2} P_t)\,\d t 
		\quad\text{in total variation, as $T\to\infty$}.
	\]
	Integrate over all such $\phi_1$
	$[\d\nu]$ and all such $\phi_2$ $[\d\mu_+]$ to deduce the corollary from 
	the bounded convergence theorem and the invariance of $\mu_+$.
\end{proof}

\section{Proofs of Theorem \ref{th:RD} and Remark \ref{rem:RD}}\label{sec:pf-RD}

Parts 1(a) and 1(d) of Theorem \ref{th:RD} were proved respectively in Theorem 
\ref{th:inv-unique} and Corollary \ref{co:ergodic}. 

Part 1(b) of Theorem \ref{th:RD} is now easy to prove.
Indeed, any probability measure $\mu$ on $C_+(\T)$ is a linear combination of $\delta_{\mathbb{0}}$
and some other probability measure $\mu_1$ on $C_+(\T)\setminus\{\mathbb{0}\}$. 
Since $\delta_{\mathbb{0}}$ is always
invariant, it follows immediately that so is $\mu_1$.  Part 1(a) now shows that $\mu=\mu_+$.
This proves part 1(b) because the converse is obvious, as every element
of $\mathscr{I\hskip-.5em{M}}$ is manifestly invariant.

Thanks to part 1(d) of Theorem \ref{th:RD} -- which we have already verified --
and a standard approximation theorem,
\[
	\int F\,\d\mu_+ \le \limsup_{T\to\infty}\frac1T\int_0^T F(\psi(t))\,\d t,
\]
for every lower semicontinuous function $F:C_+(\T)\to\R_+$. This is basically
a restatement of Fatou's theorem of classical integration theory.  Since
$F(\omega) := \|\omega\|_{C^\alpha(\T)}^k$ defines a lower semicontinuous
function on $C(\T)$, it follows that
\begin{equation}\label{eq:inv-meas-moments}
	\int \|\omega\|_{C^\alpha(\T)}^k\mu_+(\d\omega)
	\le \sup_{t\ge1}\E\left( \|\psi(t)\|_{C^\alpha(\T)}^k\right).
\end{equation}
This and Proposition \ref{pr:tight} yield 1(c) of Theorem \ref{th:RD}; see 
\eqref{tail:mu+}. 

Next we prove part 2 of Theorem \ref{th:RD} and conclude its proof. That is,
we plan to show that $\delta_{\mathbb{0}}$ is the only invariant
measure for \eqref{eq:RD} when $\lambda$ is large.

Because $V(w)\le w$ for all $w\ge0$,
the comparison theorem for SPDEs
[Lemma \ref{lem:comparison}] shows that $\psi(t\,,x) \le u(t\,,x)$
for all $t\ge0$ and $x\in\T$, where $u$ solves the SPDE,
\[
	\partial_t u(t\,,x) = \partial^2_x u(t\,,x) + u(t\,,x) + \lambda\sigma(u(t\,,x))\dot{W}(t\,,x)\qquad\text{
	for $(t\,,x)\in(0\,,\infty)\times\T$},
\]
subject to $u(0)=\psi_0$. Define $v(t\,,x) := \exp(-t) u(t\,,x)$, and observe
that $v$ is the solution to the SPDE,
\[
	\partial_t v(t\,,x) = \partial^2_x v(t\,,x) + \lambda\sigma(t\,,v(t\,,x))\dot{W}(t\,,x)\qquad\text{
	for $(t\,,x)\in(0\,,\infty)\times\T$},
\]
subject to $v(0)=\psi_0$, where
\[
	\sigma(t\,,w) := \e^{-t}\sigma\left( \e^t w\right)
	\qquad\text{for all $t\ge0$ and $w\in\R$}.
\]
Evidently, $\sigma(t)$ is Lipschitz continuous, uniformly for all $t\ge0$. In fact,
\[
	\sup_{t\ge0}
	|\sigma(t\,,w) - \sigma(t\,,z)| \le \lip_\sigma|w-z|
	\qquad\text{for all $w,z\in\R$}.
\]
Therefore, the proof of Theorem 1.2 of \cite{KKMS20} [with $\sigma$
replaced everywhere by $\sigma(t)$] works verbatim to imply
the existence of a real number $c>0$ -- independent of $\lambda$ --
such that
\[
	\limsup_{t\to\infty} \frac1t \log \|v(t)\|_{C(\T)}\le -c\lambda^2\qquad\text{a.s.}
\]
The above statements together show that, with probability one,
\[
	\limsup_{t\to\infty}\frac1t\log\sup_{x\in\T}\psi(t\,,x)
	\le \limsup_{t\to\infty}\frac1t\log\sup_{x\in\T}u(t\,,x)
	= 1 + \limsup_{t\to\infty}\frac1t\log\sup_{x\in\T}v(t\,,x)
	\le 1-c\lambda^2,
\]
which is $<0$ when $\lambda$ is sufficiently large. Since $\psi$ is positive,
this proves that $\psi(t)\to\mathbb{0}$ as $t\to\infty$ a.s.\
when $1-c\lambda^2<0$, when convergence takes place in $C(\T)$. Therefore, it remains to
prove that $\delta_{\mathbb{0}}$ is the only invariant measure
for \eqref{eq:RD} when $1-c\lambda^2<0$. Thankfully, this is easy to do as all of the harder
work is done by now. Indeed, suppose $\mu$ is invariant for \eqref{eq:RD} and
$1-c\lambda^2<0$ so that $\lim_{t\to\infty}\psi(t)=\mathbb{0}$ a.s.

Choose and fix $\varepsilon\in(0\,,1)$,
and consider the following relatively open subsets of $C(\T)$:
\[
	S_r := \left\{ \omega\in C_+(\T): \, \|\omega\|_{C(\T)} > r\right\},\quad
	L_R :=  \left\{ \omega\in C_+(\T): \, \|\omega\|_{C(\T)} < R\right\}
	\qquad\text{for all $r,R>0$}.
\]
We may select $R>0$ large enough to insure that
$\mu(L_R^c)\le \varepsilon$. 
Next, we use the notation of Markov process theory to write, for every $r,t>0$,
\[
	\mu(S_r)  =\int \P_{\psi_0}\{\psi(t)\in S_r\}\,\mu(\d\psi_0)
	\le \varepsilon + \int_{L_R}\P_{\psi_0}\{\psi(t)\in S_r\}\,\mu(\d\psi_0).
\]
Let $\phi$ be the solution to \eqref{eq:RD} starting from $\psi(0)= R\mathbb{1}$,
using the same noise that was used for $\psi$.
Lemma \ref{lem:comparison} implies that
$\phi(t\,,x) \ge \psi(t\,,x)$ for all $t\ge0$ and $x\in\T$, $\P_{\psi_0}$-a.s.\
for every $\psi_0\in L_R$. In other words, continuing to write using Markov process
theory notation, we have
\[
	\mu(S_r)  \le \varepsilon + \mu(L_R)\P_{R\mathbb{1}}\{\psi(t)\in S_r\}
	\le \varepsilon + \P_{R\mathbb{1}}\{\psi(t)\in S_r\},
\]
for all $r,t>0$.
Since $R\mathbb{1}\in C_{>0}(\T)$, the portion of part 2 that we proved already
tells us that $\lim_{t\to\infty}\psi(t)=\mathbb{0}$, $\P_{R\mathbb{1}}$-a.s.
In particular, $\P_{R\mathbb{1}}\{\psi(t)\in S_r\}\to0$ as $t\to\infty$.
Because $\mu(S_r)$ does not depend on $(t\,,\varepsilon)$, it follows that
$\mu(S_r)=0$ for every $r>0$.
This proves that $\mu=\delta_{\mathbb{0}}$, and completes the proof
of Theorem \ref{th:RD}.\qed\\

We conclude this section with a proof of Remark \ref{rem:RD}.
Thanks to \eqref{eq:inv-meas-moments} and Proposition
\ref{pr:tight},
\[
	\int\|\omega\|_{C^\alpha(\T)}^k\,\mu_+(\d\omega) \le 
	L_1^k\left( \sqrt{k}\, \mathcal{R}(k) + [\mathcal{R}(m_0k)]^{m_0}\right)^k,
\]
for all $k\ge 2$, where $\mathcal{R}$ is the function defined in Lemma \ref{lem:moment},
$L_1$ is described in Proposition \ref{pr:tight},
and $m_0$ comes from hypothesis {\bf (F3)} from the beginning portions
of this paper. Now we appeal to Example \ref{ex:R:1} to see that,
in the context of Remark \ref{rem:RD}, there exists a constant $L>0$ such that
\[
	\int\|\omega\|_{C^\alpha(\T)}^k\,\mu_+(\d\omega) \le 
	L^k k^{2(1+\nu) k/\nu},\quad\text{uniformly for all $k\ge2$}.
\]
This and Stirling's formula together imply Remark \ref{rem:RD}.\qed

\section{On the support of $\mu_+$}\label{sec:support}

For the remainder of the paper we assume that $\lambda\in(0\,,\lambda_0)$,
so that Theorem \ref{th:RD} ensures the existence of a
non-trivial invariant measure $\mu_+$. We have seen already that
$\mu_+$ has finite moments on $C^\alpha(\T)$. In this section we derive a
few additional properties of $\mu_+$. 

Throughout this section, we write
$f\lesssim g$ for two nonnegative, real-valued
functions $f$ and $g$ when there exists a number $c>0$
such that $f(x)\le cg(x)$ uniformly for all $x$ in the common domain of
definition of $f$ and $g$.

We have seen already that $\mu_+$ lives on the strictly positive functions
in $C(\T)$. Our first result is a quantitative bound that complements this fact.

\begin{proposition}\label{pr:lower-tail}
	$\mu_+\{ \omega\in C_+(\T):\,
	\inf_{x\in\T}\omega(x) \le \varepsilon\} \lesssim\varepsilon^{1/4}$
	for all $\varepsilon\in(0\,,1)$.
\end{proposition}

It is in fact possible to adapt the proof to see that for every $\theta\in(0\,,1)$
there exists $\lambda_\theta$ such that if $\lambda\in(0\,,\lambda_\theta)$, then
\[
	\mu_+\left\{ \omega\in C_+(\T):\,
	\inf_{x\in\T}\omega(x) \le \varepsilon\right\}\lesssim\varepsilon^\theta
	\qquad\text{for all $\varepsilon\in(0\,,1)$}.
\]
We skip the details and prove only Proposition \ref{pr:lower-tail}.

\begin{proof}
	We plan to prove that $\mu_+(\Gamma_\varepsilon)\lesssim \varepsilon^{1/4}$
	uniformly for all $\varepsilon\in(0\,,1)$, where
	\[
		\Gamma_\varepsilon := \left\{ \omega\in C_+(\T):\,\inf_{x\in\T}\omega(x) \le \varepsilon\right\}.
	\]
	According to the random walk argument [see \eqref{eq:lowertail1} ],
	with probability one,
	\[
		\limsup_{T\to\infty}\frac1T\int_0^T\bm{1}_{\Gamma_\varepsilon}(\psi(t))\,\d t
		\lesssim \sqrt{\limsup_{m\to\infty}
		\frac1m\sum_{j=0}^{m-1}\bm{1}_{\{X_{j+1}\le-|\log_2(8\varepsilon)|\}}}.
	\]
	Apply \eqref{I(X)} to deduce the result.
\end{proof}

From now on, we assume additionally that
\begin{equation}\label{(F4)}
	F'(x) = O(x^{m_0-1})\qquad\text{as $x\to\infty$}.
\end{equation}
This condition implies {\bf (F3)}, and is only a little stronger than
{\bf (F3)} for many examples. Recall that $\mu_+(C^\alpha(\T))=1$
for every $\alpha\in(0\,,1/2)$. The following shows that
$\mu_+$ does not charge the critical case.

\begin{theorem}\label{th:mu_+(C^1/2)}
	For $\mu_+$-almost all $\omega\in C(\T)$,
	\begin{equation}\label{eq:mu_+(C^1/2)}\begin{split}
		\limsup_{r\downarrow0}\frac{1}{\sqrt{r\log(1/r)}}
			\adjustlimits\sup_{|y|<r}\sup_{x\in\T}|\omega(x+y)
			-\omega(x)| &= \lambda \sup_{x\in\T}|\sigma(\omega(x))|,\\
	\liminf_{r\downarrow0} \sqrt{\frac{16\log(1/r)}{\pi^2 r}}
		\adjustlimits\inf_{|y|<r}\sup_{x\in\T}|\omega(x+y)-\omega(x)|
		&= \lambda \sup_{x\in\T}|\sigma(\omega(x))|.
	\end{split}\end{equation}
	In particular, $\mu_+(C^{1/2}(\T))=0$.
\end{theorem}

Before we prove Theorem  \ref{th:mu_+(C^1/2)}
let us state a result about the fractal nature of the functions in
the support of $\mu_+$. Recall that the \emph{Hausdorff dimension} of
a Borel set $G\subset\T$ is
\[
	\dim_{_{\rm H}} (G) = \sup\left\{ s>0:\,
	I_s(m)<\infty\text{ for some probability measure $m$ on $G$}\right\},
\]
where $I_s(m)$ denotes the $s$-dimensional ``energy integral,''
\[
	I_s(m) := \iint \frac{m(\d x)\, m(\d y)}{|x-y|^s}.
\]
Then, we have the following property.

\begin{theorem}\label{th:dimh}
	Choose and fix an arbitrary Borel set $G\subset\T$. Then,
	\[
		\dim_{_{\rm H}} \omega(G) = 1\wedge 2\dim_{_{\rm H}}(G)
		\qquad\text{for $\mu_+$-almost all $\omega\in C(\T)$}.
	\]
\end{theorem}

Theorems \ref{th:mu_+(C^1/2)} and \ref{th:dimh} suggest that
the functions in the support of $\omega$ ``look'' like
Brownian paths. Of course, this cannot be interpretted too strongly
as $\mu_+$ is singular with respect to
Wiener measure $\mathbb{W}$ on $C(\T)$;
in fact, $\mathbb{W}(C_+(\T))=0,$ yet
$\mu_+(C_+(\T))=1$ by Theorem \ref{th:RD}. 

We begin the proofs of Theorems  \ref{th:mu_+(C^1/2)} and \ref{th:dimh}.
As a simple first step we offer the following real-variable consequence of \eqref{(F4)}.

\begin{lemma}\label{lem:V(x)-V(y)}
	There exists $c>0$ such that
	\[
		|V(x) - V(y) | \le c(1\vee x\vee y)^{m_0-1}|x-y|
		\qquad\text{for all $x,y\ge0$}.
	\]
\end{lemma}

\begin{proof}
	Without loss of generality, $x\ge y\ge0$, in which case,
	\[
		0\le |V(x) - V(y) |\le x-y + \int_y^x |F'(w)| \,\d w 
		\le \left[1+\sup_{w\le x} |F'(w)|\right](x-y).
	\]
	This yields the result.
\end{proof}

From now on we consider \eqref{eq:RD} with initial value $\psi(0)=\mathbb{1}$.
Also, let $\mathcal{J}$ denote the solution to the
following linearized version of \eqref{eq:RD} with $\lambda=1$,
\[
	\partial_t \mathcal{J} = \partial^2_x\mathcal{J} + \dot{W},
\]
subject to $\mathcal{J}(0)=\mathbb{0}$. That is,
\begin{equation}\label{J}
	\mathcal{J}(t\,,x) = \int_{(0,t)\times\T} p_{t-s}(x\,,y)\,W(\d s\,\d y)
\end{equation}
We have the following second-order regularity result. Related results have been 
found by \cite{FoondunKM} and \cite{HairerPardoux}.

\begin{lemma}\label{lem:localize}
	For every $\varepsilon\in(0\,,3/4)$ and $k\ge2$,
	\[
		\sup_{t\ge 2}
		\E\left( \sup_{\substack{x,z\in\T\\
		x\neq z}} \frac{\left| \psi(t\,,x) - \psi(t\,,z)-\lambda\sigma(\psi(t\,,z))\{\mathcal{J}(t\,,x)
		-\mathcal{J}(t\,,z)\}\right|^k }{|x-z|^{\left(\frac34-\varepsilon\right)k}}\right)<\infty.
	\]
\end{lemma}

\begin{proof}
	 Throughout,  let us choose and fix $t\ge 2$ and $k\ge2$. Thanks to \eqref{eq:mild},
	\begin{equation}\label{psi-psi}\begin{split}
		&\psi(t\,,x) - \psi(t\,,z) \\
		&\hskip.8in= \int_{(0,t)\times\T} \left[ p_{t-s}(x\,,y) - p_{t-s}(z\,,y)
		\right] V(\psi(s\,,y))\,\d s\,\d y + \lambda\left\{
		\mathcal{I}(t\,,x) - \mathcal{I}(t\,,z) \right\},
	\end{split}\end{equation}
	almost surely for every $x,z\in\T$. We estimate the two quantities on the right-hand
	side of \eqref{psi-psi} separately and in order. 
	
	First of all, note that the $L^k(\Omega)$-norm of the first term on the right-hand
	side of \eqref{psi-psi} can be written as
	\begin{align*}
		&\left\|\int_{(0,t)\times\T} \left[ p_{t-s}(x\,,y) - p_{t-s}(z\,,y)
			\right] V(\psi(s\,,y))\,\d s\,\d y - 
			\int_{(0,t)\times\T} \left[ p_{t-s}(x\,,y) - p_{t-s}(z\,,y)
			\right] V(\psi(t\,,x )) \,\d s\,\d y\right\|_k\\
		&\le\int_{(0,t)\times\T}\left| p_{t-s}(x\,,y) - p_{t-s}(z\,,y)\right|
			\|V(\psi(s\,,y)) - V(\psi(t\,,x))\|_k\,\d s\,\d y\\
		&\lesssim\int_{(0,t)\times\T}\left| p_{t-s}(x\,,y) - p_{t-s}(z\,,y)\right|
			\left\|\{ \psi(s\,,y) - \psi(t\,,x)\}
			\left( 1 + |\psi(s\,,y)| + |\psi(t\,,x)|\right)^{m_0-1}\right\|_k\,\d s\,\d y;
	\end{align*}
	see Lemma \ref{lem:V(x)-V(y)} for the last line. We emphasize that the implied
	constant does not depend on the choice of $(t\,,x\,,z)\in[2\,,\infty)\times\T^2$.
	In any case, it follows readily
	from Lemma \ref{lem:cont} and Proposition \ref{pr:temporal:cont} that for any $\alpha \in (0\,, 1/2)$ and $\beta \in (0\,, 1/4)$,
	\begin{equation}\label{eq:long:p-p}\begin{split}
		&\left\|\int_{(0,t)\times\T} \left[ p_{t-s}(x\,,y) - p_{t-s}(z\,,y)
			\right] V(\psi(s\,,y))\,\d s\,\d y\right\|_k\\
		&\hskip1.7in\lesssim\int_{(0,t)\times\T}\left| p_{t-s}(x\,,y) - p_{t-s}(z\,,y)\right|
			\left((t-s)^{\beta} + |x-y|^{\alpha}\right)\,\d s\,\d y,
	\end{split}\end{equation}
	where once again the implied
	constant does not depend on the choice of $(t\,,x\,,z)\in[2\,,\infty)\times\T^2$.
	By \eqref{p-p},
	\[
		|p_{t-s}(x\,,y) - p_{t-s}(z\,,y)| (t-s)^{\beta}\lesssim(t-s)^{\beta}\sum_{k=1}^\infty
		\e^{-\pi^2 k^2(t-s)}\left( |x-z|k\wedge 1\right),
	\]
	for similar constant dependencies as above.
	Therefore,
	\[
		\int_{(0,t)\times\T}
		|p_{t-s}(x\,,y) - p_{t-s}(z\,,y)| (t-s)^{\beta}\,\d s\,\d y
		\lesssim\int_0^ts^{\beta}\,\d s\sum_{k=1}^\infty \left( |x-z|k\wedge 1\right)
		\e^{-\pi^2 k^2s},
	\]
	valid uniformly for all $(t\,,x\,,z)\in[2\,,\infty)\times\T^2$.
	We split the above sum in two parts according to whether or not $k\le 1/|x-z|$.
	First,
	\begin{align*}
		\int_0^t s^{\beta}\,\d s\sum_{\substack{k\in\N:\\
			k \le 1/ |x-z|}} \left( |x-z|k\wedge 1\right)
			\e^{-\pi^2 k^2s}
			&= |x-z|\int_0^t s^{\beta}\,\d s\sum_{k=1}^\infty \e^{-\pi^2 k^2s}\\
		&= \frac{|x-z|}{\pi^{2\beta+2}}
			\sum_{k=1}^\infty  k^{-2-2\beta} \int_0^{\pi^2 k^2t}
			r^{\beta }\e^{-r}\,\d r \lesssim |x-z|.
	\end{align*}
	Next, we observe that
	\begin{align*}
		\int_0^t s^{\beta}\,\d s\sum_{\substack{k\in\N:\\
			k > 1/ |x-z|}} \left( |x-z|k\wedge 1\right)
			\e^{-\pi^2 k^2s} &=
			\int_0^t s^{\beta}\,\d s\sum_{\substack{k\in\N:\\
			k > 1/ |x-z|}}  \e^{-\pi^2 k^2s}
			\le  \int_0^t s^{\beta}\,\d s \int_{1/|x-z|}^\infty\d w\ \e^{-\pi^2 w^2s}\\
		&\lesssim\int_0^t s^{\beta-1/2}\exp\left( -\frac{\pi^2 s}{|x-z|^2}\right)\d s
			\lesssim |x-z|^{1+2\beta} \lesssim|x-z|.
	\end{align*}
	Combine to deduce the uniform estimate
	\[
		\int_{(0,t)\times\T}
		|p_{t-s}(x\,,y) - p_{t-s}(z\,,y)| (t-s)^{\beta}\,\d s\,\d y
		\lesssim |x-z|,
	\]for all $\beta \in (0,1/4)$. 
	Therefore, Lemma \ref{lem:p-p} and \eqref{eq:long:p-p} together yield
	\begin{equation}\label{Term1}
		\left\|\int_{(0,t)\times\T} \left[ p_{t-s}(x\,,y) - p_{t-s}(z\,,y)
		\right] V(\psi(s\,,y))\,\d s\,\d y\right\|_k
		\lesssim |x-z| \log_+(1/|x-z|),
	\end{equation}
	where the implied
	constant does not depend on the choice of $(t\,,x\,,z)\in[2\,,\infty)\times\T^2$.
	
We now define 
\begin{align*}
A_1&:=(0, t-|x-z|)\times \T,\\
A_2&:=(t-|x-z|, t) \times\left[z-2|x-z|^\gamma, z+2|x-z|^\gamma\right]^{c},\\
A_3&:=(t-|x-z|, t)\times \left[z-2|x-z|^\gamma, z+2|x-z|^\gamma\right],
\end{align*}
where $\gamma \in (0\,,1/2)$ is a fixed constant. Here, we assume that $|x-z|$ is small enough so that $(z-2|x-z|^\gamma, z+2|x-z|^\gamma) \subset (-1\,, 1)$. If $|x-z|$ is not small --
i.e., if $|x-z|\geq c$ for some constant $c>0$ -- 
then we may use  Lemma \ref{lem:cont} (more precisely, the proof of of Lemma \ref{lem:cont}) to prove  the lemma. Since  we assume $t\geq 2$, we have $0<t-|x-z|<t$. Therefore, we may define 
\[
\mathcal{I}(t\,,x) - \mathcal{I}(t\,,z) -
			\sigma(\psi(t\,,z))\{\mathcal{J}(t\,,x) - \mathcal{J}(t\,,z)\}:=\sum_{i=1}^5 Q_i,
\]			
where
\begin{align*}
&Q_1:= \int_{A_1}  \left[ p_{t-s}(x\,,y) - p_{t-s}(z\,,y)\right]
			\sigma(\psi(s\,,y))\, W(\d s\,\d y)  \\
&Q_2:=  \int_{A_2} \left[ p_{t-s}(x\,,y) - p_{t-s}(z\,,y)\right]
			\sigma(\psi(s\,,y))\, W(\d s\,\d y) \\ 
&Q_3:= \int_{A_3} \left[ p_{t-s}(x\,,y) - p_{t-s}(z\,,y)\right]
			\left[\sigma(\psi(s\,,y)) - \sigma(\psi(t-|x-z|\,,z)\right]   \, W(\d s\,\d y)  \\
&Q_4:=  \left[\sigma(\psi(t-|x-z|\,,z)- \sigma(\psi(t\,,z))\right]   \int_{A_3} \left[ p_{t-s}(x\,,y) - p_{t-s}(z\,,y)\right]\, W(\d s\,\d y)   \\
&Q_5:=  \sigma(\psi(t\,,z))   \int_{A_1\cup A_2} \left[ p_{t-s}(x\,,y) - p_{t-s}(z\,,y)\right]\, W(\d s\,\d y) . 
\end{align*}
We repeatedly use $|\sigma(a)|\leq \lip_\sigma |a|$ and the boundedness of the moments of $\psi$.   First consider $Q_1$.	We can appeal to the Burkholder-Davis-Gundy inequality (see the proof of Theorem \ref{th:exist-unique}) and \eqref{eq:p2}
	to see that
	\begin{align*}
		\|Q_1\|_k^2 &\lesssim \int_0^{t-|x-z|}\d s\int_{\T}\d y\
			\left[ p_{t-s}(x\,,y) - p_{t-s}(z\,,y)\right]^2\\
		&\propto \sum_{n=1}^\infty\left[ 1-\cos(\pi n|x-z|)\right]\int_{|x-z|}^t
			\e^{-2\pi^2 n^2 s}\,\d s\\
		&\lesssim \sum_{n=1}^\infty\left( |x-z|^2n^2\wedge1\right)\left( 
			\frac{\e^{-2\pi^2 n^2|x-z|}}{n^2}\right)\\
		&= \sum_{n=1}^\infty\left( |x-z|^2  \wedge \frac{1}{n^2}\right)\e^{-2\pi^2 n^2|x-z|}\\
		&= |x-z|^2\sum_{n\le 1/|x-z|}\e^{-2\pi^2 n^2|x-z|}
			+ \sum_{n>1/|x-z|} n^{-2}\e^{-2\pi^2 n^2|x-z|}\\
		&\lesssim |x-z|^{3/2},
	\end{align*}
	uniformly for all $t\geq 2$.
	
	Next, consider $Q_2$. A similar appeal to the BDG inequality yields 
	\begin{align*}
	\|Q_2\|_k^2 & \lesssim \int_{t-|x-z|}^t \d s  \int_{\left[z-2|x-z|^\gamma, z+2|x-z|^\gamma\right]^c}\d y\ \left[ p_{t-s}(x\,,y) - p_{t-s}(z\,,y)\right]^2 \\
	& \lesssim \int_0^{|x-z|} \d s \int_{\left[z-2|x-z|^\gamma, z+2|x-z|^\gamma\right]^c}\d y\ \left(  [ p_{s}(x\,,y) ]^2  +  [p_{s}(z\,,y)]^2 \right) \\
	&\lesssim \int_0^{|x-z|} \d s \int_{\left[z-2|x-z|^\gamma, z+2|x-z|^\gamma\right]^c}\d y\ \left(  \frac{1}{4\pi s}\exp\left\{ - \frac{(y-z)^2}{2s}\right\} +  \frac{1}{4\pi s}\exp\left\{ - \frac{(y-x)^2}{2s}\right\} \right), 
	\end{align*}
	where the last inequality comes from the simple fact that 
	\[
		p_t(x\,,y)\lesssim  \frac{1}{\sqrt{t}}\exp\left\{ - \frac{(y-z)^2}{4t}\right\},
	\]
	for all $x, y \in \T$ and for all $t \leq 1$. Since 
	$\min\{|y-z|, |y-x|\} \geq |x-z|$ when $y\in \left[z-2|x-z|^\gamma, z+2|x-z|^\gamma\right]^c$, 
	it follows that 
	\begin{align*}
	\|Q_2\|_k^2 & \lesssim  \int_0^{|x-z|} \d s\  s^{-1/2} \int_{|y-z|\geq |x-z|}\d y\ \left(  \frac{1}{\sqrt{s}}\exp\left\{ - \frac{(y-z)^2}{2s}\right\} \right)  \\
	&\lesssim |x-z|^{1/2} \exp\left( -|x-z|^{2\gamma -1}  \right)\\
	&\lesssim |x-z|^{1+\gamma},
	\end{align*}
uniformly for all $t\geq 2$. We
used the facts that $\gamma<1/2$ and $|x-z|\ll1$ in
order to guarantee the last inequality above. 
	
	Now we consider $Q_3$. Apply the BDG inequality once more  to see that 
	
	\begin{align*}
		\|Q_3\|_k^2&:=\int_{t-|x-z|}^t  \d s \int_{|y-z| \leq |x-z|^\gamma} \d y  \left[ p_{t-s}(x\,,y) - p_{t-s}(z\,,y)\right]^2  \left\|\sigma(\psi(s\,,y)) - \sigma(\psi(t-|x-z|\,,z)\right\|_k^2 \\
		&\lesssim \int_{t-|x-z|}^t  \d s \int_{|y-z| \leq |x-z|^\gamma} \d y  \left[ p_{t-s}(x\,,y) - p_{t-s}(z\,,y)\right]^2 \left\{  (s-t+|x-z|)^{1/2} + |y-z|\right\} \\
		&\lesssim \left\{ |x-z|^{1/2}+|x-z|^\gamma\right\} \int_0^{|x-z|} \d s \int_{\T} \d y     \left[ p_{s}(x\,,y) - p_{s}(z\,,y)\right]^2  \\
		&\lesssim \left\{ |x-z|^{1/2}+|x-z|^\gamma\right\} \left\{  \sum_{n=1}^\infty\left( \frac{1-\e^{-2\pi^2n^2|x-z|}}{n^2}\right)  \left[ 1-\cos(\pi|x-z|n)\right] \right\}\quad (\text{see \eqref{eq:p2}}) \\
		&\lesssim \left\{ |x-z|^{1/2}+|x-z|^\gamma\right\} \left\{  \sum_{n=1}^\infty\left( \frac{1-\e^{-2\pi^2n^2|x-z|}}{n^2}\right)  \left[ 1\wedge (|x-z|n)^2 \right] \right\}    \\
		&\lesssim |x-z|^{3/2}+|x-z|^{1+\gamma},
	\end{align*} uniformly for all $t\geq 2$. 
	
	Next, we consider $Q_4$. 
	H\"older's inequality,  the BDG inequality, and the fact that $\sigma$ is Lipchitz
	together imply that 
	\begin{align*}
	\|Q_4\|_k^2 &\lesssim \left\| \psi(t- |x-z|\,, z)- \psi(t\,, z) \right\|_{2k}^2
	\int_{t-|x-z|}^t  \d s \int_{|y-z| \leq |x-z|^\gamma} \d y  \left[ p_{t-s}(x\,,y) - p_{t-s}(z\,,y)\right]^2.
	\end{align*}
	We now apply Proposition \ref{pr:temporal:cont} and use a similar calculation as for $Q_3$ to get that 
	\begin{align*}
	\|Q_4\|_k^2 &\lesssim |x-z|^{3/2},
	\end{align*}
	uniformly for all $t\geq 2$. 
	
	Lastly, we consider $Q_5$. Since the moments of $\psi(t\,, x)$ are uniformly bounded, 
	we may use the fact that $|\sigma(a)|\leq |a|$ and follow the  calculations for $Q_1$ and $Q_2$
	to see that 
	\[ 
	\|Q_5\|_k^2 \lesssim |x-z|^{3/2} + |x-z|^{1+\gamma},
	\]uniformly for all $t\geq 2$. 
	Combine the preceding estimates to find that for every $\gamma \in (0\,,1/2)$ 
	\[
		\left\| \mathcal{I}(t\,,x) - \mathcal{I}(t\,,z) - \sigma(\psi(t\,,z))
		\{ \mathcal{J}(t\,,x) - \mathcal{J}(t\,,z) \} \right\|_k \lesssim |x-z|^{(1+\gamma)/2},
	\]
	uniformly for all $t\geq 2 $. This, \eqref{Term1}, and \eqref{psi-psi} together imply that
	 for every $\gamma \in (0\,,1/2)$ 
	\[
		\left\| \psi(t\,,x) - \psi(t\,,z) - \lambda\sigma(\psi(t\,,z))
		\{ \mathcal{J}(t\,,x) - \mathcal{J}(t\,,z) \} \right\|_k \lesssim |x-z|^{(1+\gamma)/2},
	\]
	uniformly for all $t\geq 2 $. 
	The remainder
	of the result follows from the above and a standard chaining argument
	that uses Lemma \ref{lem:cont} and Proposition \ref{pr:temporal:cont};
	we skip the details as they are routine.
\end{proof}

\begin{proof}[Proof of Theorem \ref{th:mu_+(C^1/2)}]
	We plan to prove only \eqref{eq:mu_+(C^1/2)}.
	The first assertion of \eqref{eq:mu_+(C^1/2)} immediately also implies
	that $\mu_+(C^{1/2}(\T))=0$.
	
	Recall \eqref{J}.
	We study the Gaussian process $\mathcal{J}$ by studying its incremental
	variance using \eqref{eq:p2} 
	as follows: For all $t>0$ and $x,z\in\T$,
	\begin{equation}\label{J-J}\begin{split}
		\E\left( |\mathcal{J}(t\,,x) - \mathcal{J}(t\,,z)|^2\right)
			&= \int_0^t\d s\int_{\T}\d y\ \left[ p_s(x\,,y)-p_s(z\,,y)\right]^2\\
		&=2\int_0^t\d s\sum_{n=1}^\infty \e^{-2\pi^2n^2s}\left[ 1-\cos(\pi|x-z|n)\right]\\
		&= \frac{1}{\pi^2}\sum_{n=1}^\infty\left( \frac{1-\e^{-2\pi^2n^2t}}{n^2}\right)
			\left[ 1-\cos(\pi|x-z|n)\right]
	\end{split}\end{equation}
	Let $\hat{W}$ denote an independent space-time Brownian sheet and define
	a new, independent, Gaussian process $\mathcal{K}$ as follows:
	\[
		\mathcal{K}(x) := \int_{(0,\infty)\times\T} \left[ p_{t+s}(x\,,y) - p_{t+s}(0\,,y)\right]
		\hat{W}(\d s\,\d y).
	\]
	The above is well defined, as
	\begin{align*}
		\E\left(|\mathcal{K}(x)|^2\right) &= \int_t^\infty\d s\int_{\T}\d y\
			\left[ p_s(x\,,y)-p_s(z\,,y)\right]^2\\
		&\leq 2\int_t^\infty\d s\sum_{n=1}^\infty \e^{-2\pi^2n^2s}
			= \frac{1}{\pi^2}\sum_{n=1}^\infty\frac{\e^{-2\pi^2n^2t}}{n^2};
	\end{align*}
	see \eqref{eq:p2}. Moreover, we apply \eqref{eq:p2} yet again to see that
	\begin{equation}\label{K-K}\begin{split}
		\E\left(|\mathcal{K}(x) - \mathcal{K}(z)|^2\right) &= \int_t^\infty\d s\int_{\T}\d y\
			\left[ p_s(x\,,y)-p_s(z\,,y)\right]^2\\
		&= 2\int_t^\infty\d s\sum_{n=1}^\infty \e^{-2\pi^2n^2s}\left[ 1-\cos(\pi|x-z|n)\right]\\
		&= \frac{1}{\pi^2}\sum_{n=1}^\infty\left( \frac{\e^{-2\pi^2n^2t}}{n^2}\right)
			\left[ 1-\cos(\pi|x-z|n)\right].
	\end{split}\end{equation}
	Now, define
	\begin{equation}\label{beta:J:K}
		\beta(x) :=  \sqrt{2}\left[ \mathcal{J}(t\,,x) - \mathcal{J}(t\,,0) + 
		\mathcal{K}(x) -\mathcal{K}(0) \right] +\frac{x}{\sqrt{2}}
		\qquad\text{for all $x\in\T$}.
	\end{equation}
	Then, \eqref{J-J} and \eqref{K-K} together yield
	\[
		\E\left( |\beta(x) - \beta(z)|^2\right) 
		= |x-z|
		\qquad\text{for all $x,z\in\T$}.
	\]
	Since $\beta(0)=0$, it follows that $\beta$ is a two-sided Brownian motion
	indexed by $\T\simeq[-1\,,1]$. Also, we apply \eqref{p=exp}
	yet another time to find that
	\begin{align*}
		\Cov\left(\mathcal{K}(x) \,, \mathcal{K}(z)\right) &= \int_t^\infty\d s\int_{\T}\d y\
			\left[ p_s(x\,,y)-p_s(0\,,y)\right] \left[ p_s(z\,,y)-p_s(0\,,y)\right] \\
		&= \frac{1}{4\pi^2}
			\sum_{n=1}^\infty \left( \frac{1-\e^{-i\pi xn}}{n} \right)
			\left( \frac{1-\e^{i\pi zn}}{n}\right) \e^{-2\pi^2 n^2t},
	\end{align*}
	for every $x,z\in\T$. Since $t\ge16\pi^2>0$, the preceding is a $C^\infty$ function of $x$
	and $z$. Therefore, a standard fact about Gaussian random fields
	implies that $\mathcal{K}$ is a.s.\ $C^\infty$. 
	In light of \eqref{beta:J:K}, we have
	proved the following version of an observation of \citet[Exercise 3.10, page 326]{wal86}:
	\[
		\mathcal{J}(t\,,x) -\mathcal{J}(t\,,0) = \frac{\beta(x)}{\sqrt{2}} + \text{ a $C^\infty$\!
		Gaussian process}.
	\] 
	Because $\beta$ has the following ``modulus of non-differentiability''
	\citep[see][Theorem 1.6.1]{CR81},
	\[
		\liminf_{r\downarrow0} \sqrt{\frac{8\log(1/r)}{\pi^2 r}}
		\adjustlimits\inf_{|y|<r}\sup_{x\in\T}|\beta(x+y)-\beta(x)|
		=1
		\qquad\text{a.s.,}
	\]
	it follows from \eqref{beta:J:K} that a.s.,
	\[
		\liminf_{r\downarrow0} \sqrt{\frac{16\log(1/r)}{\pi^2 r}}
		\adjustlimits\inf_{|y|<r}\sup_{x\in\T}|\mathcal{J}(t\,,x+y)-\mathcal{J}(t\,,x)|
		=1.
	\]
	Therefore, we can deduce from Lemma \ref{lem:localize} that
	\[
		\liminf_{r\downarrow0} \sqrt{\frac{16\log(1/r)}{\pi^2 r}}
		\adjustlimits\inf_{|y|<r}\sup_{x\in\T}|\psi(t\,,x+y)-\psi(t\,,x)|
		=\lambda\sup_{x\in\T}|\sigma(\psi(t\,,x))|.
	\]
	In other words, if $t\ge 2$, then
	\begin{align*}
		&\P\{\psi(t)\in\Lambda\}=1, \quad\text{where}\\
		&\Lambda := \left\{ \omega\in C(\T):\,
			\liminf_{r\downarrow0} \sqrt{\frac{\log(1/r)}{r}}
			\adjustlimits\inf_{|y|<r}\sup_{x\in\T}|\omega(x+y)-\omega(x)|
			=\frac{\pi\lambda}{4}\sup_{x\in\T}|\sigma(\omega(x))|\right\}.
	\end{align*}
	According to the ergodic theorem [1(d) of Theorem \ref{th:RD}],
	\[
		\mu_+(\Lambda) = \lim_{T\to\infty}\frac1T\int_0^T\P\{\psi(t)\in\Lambda\}\,\d t
		=1.
	\]
	This proves the second assertion of \eqref{eq:mu_+(C^1/2)}. The
	first assertion is proved using the same kind of arguments, except
	we appeal to the following
	\[
		\limsup_{r\downarrow0}\frac{1}{\sqrt{2r\log(1/r)}}
		\adjustlimits\sup_{|y|<r}\sup_{x\in\T}|\beta(x+y)-\beta(x)|=1
		\qquad\text{a.s.,}
	\]
	which is the celebrated modulus of continuity of Brownian motion; see \cite{Levy1937}.
\end{proof}

\begin{proof}[Proof of Theorem \ref{th:dimh}]
	If $\omega\in C^\alpha(\T)$ for some $\alpha\in(0\,,1]$, then a 
	standard covering argument 
	shows that
	\begin{equation}\label{ps:UB}
		\dim_{_{\rm H}}\omega(G) \le 1\wedge\alpha^{-1}
		\dim_{_{\rm H}}(G);
	\end{equation}
	see for example \cite{McKean}.
	Since $\mu_+(C^\alpha(\T))=1$ for all $\alpha\in(0\,,1/2)$ 
	[1(c) of Theorem \ref{th:RD}], it follows that \eqref{ps:UB}
	holds for  $\mu_+$-almost all $\omega\in C(\T)$. Let $\alpha\uparrow 1/2$
	to deduce from \eqref{ps:UB} that
	\begin{equation}\label{omega:UB}
		\dim_{_{\rm H}} \omega(G) \le 1\wedge 2\dim_{_{\rm H}}(G)
		\qquad\text{for  $\mu_+$-almost all $\omega\in C(\T)$}.
	\end{equation}
	Next we derive a matching lower bound.
	
	Choose and fix an arbitrary non-random number $s\in(0\,,1\wedge 2\dim_{_{\rm H}}(G))$.
	Frostman's theorem
	ensures that there exists a probability measure $m$ on $G$
	such that 
	\begin{equation}\label{I<infty}
		I_{s/2}(m)<\infty.
	\end{equation}
	See, for example, Theorem 2 of \citet[page 133]{Kahane}.
	Choose and fix not only the $s$, but also the probability measure $m$.
	
	Choose and fix an arbitrary non-random number
	\begin{equation}\label{t>16pi^2}
		t \ge 2.
	\end{equation}
	According to Lemma \ref{lem:localize},
	\[
		X := \sup_{\substack{x,z\in\T\\
		x\neq z}} \frac{\left| \psi(t\,,x) - \psi(t\,,z)-\lambda\sigma(\psi(t\,,z))\{\mathcal{J}(t\,,x)
		-\mathcal{J}(t\,,x)\}\right| }{|x-z|^{3/5}} \in \bigcap_{k\ge2} L^k(\Omega).
	\]
	Let 
	\[
		Y := \lambda \inf_{x\in\T}\sigma(\psi(t\,,x)).
	\]
	It is immediately clear from \eqref{th:RD} that
	$Y\ge {\rm L}_\sigma\inf_{x\in\T}\psi(t\,,x)>0$ a.s. Thus, we can write
	\begin{equation*}\begin{split}
		|\psi(t\,,x) - \psi(t\,,z)| &\ge \left[ Y|\mathcal{J}(t\,,x)-\mathcal{J}(t\,,z)| - X
			|x-z|^{3/5}\right]_+\\
		&\ge Y\left[ 
			|\mathcal{J}(t\,,x)-\mathcal{J}(t\,,z)|- \left(\frac{X}{Y}\right)|x-z|^{3/5}\right]_+,
	\end{split}\end{equation*}
	where $\mathcal{J}$ was defined in \eqref{J}. 
	
	Define  $m_\psi$ to be the push forward of $m$ by $\psi(t)$.
	More precisely,
	\[
		\int g\,\d m_\psi := \int g(\psi(t\,,x))\, m(\d x)
		\qquad\text{for all $g\in C_+(\R)$}.
	\]
	Clearly $m_\psi$ is a random probability measure on
	the random set
	\[
		\psi(t\,,G) := \{\psi(t\,,x):\, x\in G\}.
	\]
	Now,
	\[
		I_s(m_\psi) = \iint\frac{m(\d x)\,m(\d z)}{|\psi(t\,,x) - \psi(t\,,z)|^s}\\
		\le Y^{-s}\tilde{I}_s(m\,; X/Y),
	\]
	where
	\[
		\tilde{I}_s(m_\psi\,;N) :=
		\iint\frac{m(\d x)\, m(\d z)}{\left[ 
		|\mathcal{J}(t\,,x)-\mathcal{J}(t\,,z)|- N|x-z|^{3/5}\right]_+^s}
		\qquad\text{for all $N>0$}.
	\]
	The first portion of the proof is concerned with proving that
	\begin{equation}\label{tilde:I}
		\tilde{I}_s(m_\psi\,;N)<\infty\quad\text{almost surely for every $N>0$}.
	\end{equation}
	In this way we will see that
	$\tilde{I}_s(m_\psi\,;X/Y)<\infty$ almost surely on $\{X\le NY\}$, and in particular,
	\[
		\P\left\{ \tilde{I}_s(m_\psi\,;X/Y)<\infty\right\} \ge\lim_{N\to\infty}\P\{X \le NY\}=1.
	\]
	With this aim in mind, we first recall the incremental variance of
	$\mathcal{K}$ from \eqref{K-K}: For all $x,z\in\T$,
	\begin{align*}
		\E\left( |\mathcal{K}(x)-\mathcal{K}(z)|^2\right)
			&=\frac{1}{\pi^2}\sum_{n=1}^\infty\left( \frac{\e^{-2\pi^2n^2t}}{n^2}\right)
			\left[ 1-\cos(\pi|x-z|n)\right]\\
		&\le |x-z|^2\sum_{n=1}^\infty\e^{-2\pi^2n^2t} 
			\le|x-z|^2\int_0^\infty \e^{-2\pi^2 w^2t}\,\d w\\
		&= \frac{|x-z|^2}{2\sqrt{2\pi t}}.
	\end{align*}
	Since $t\ge2$ [see \eqref{t>16pi^2}],  the above quantity is $\le \frac14|x-z|^2$, and hence
	\begin{equation}\label{Var:LB}
		\E\left( |\mathcal{J}(t\,,x)-\mathcal{J}(t\,,z)|^2\right) =
		\frac{\E\left(|\beta(x)-\beta(z)|^2\right)}{2} + \frac{(x-z)^2}{4} -
		\E\left( |\mathcal{K}(x)-\mathcal{K}(z)|^2\right)
		\ge \frac{|x-z|}{2}.
	\end{equation}
	Thus, we find that
	\begin{align*}
		\sup_{N>0}
			\E\left[\tilde{I}_s(m_\psi\,;N)\right] &\le\frac{\Gamma\left(\frac{1-s}{2}\right)}{
			2^{s/2}\sqrt\pi}\iint\frac{m_\psi(\d x)\, m_\psi(\d z)}{
			\left[{\rm Var}\left(\mathcal{J}(t\,,x)-\mathcal{J}(t\,,z)\right) \right]^{s/2}}
			&[\text{see Lemma \ref{lem:Gauss}}]\\
		&\le\frac{\Gamma\left(\frac{1-s}{2}\right)}{
			\sqrt\pi} I_{s/2}(m_\psi)&[\text{by \eqref{Var:LB}}].
	\end{align*}
	Therefore, the above and \eqref{I<infty} together imply \eqref{tilde:I}.
	In turn, this and Frostman's theorem together imply that,
	$\dim_{_{\rm H}}\psi(t\,,G) \ge s$ a.s. This is valid for
	every $s\in(0\,,1\wedge2\dim_{_{\rm H}}(G))$. Therefore,
	we let $s$ converge upward to $1\wedge2\dim_{_{\rm H}}(G)$
	in order to deduce the following: For every $t\ge 2$,
	$\dim_{_{\rm H}}\psi(t\,,G) \ge 1\wedge 2\dim_{_{\rm H}}(G)$ a.s.
	In other words, if we define 
	\[
		\Theta := \left\{ \omega\in C(\T):\, \dim_{_{\rm H}}\omega(G)
		\ge 1\wedge 2\dim_{_{\rm H}}(G)\right\},
	\]
	then $\P\{\psi(t)\in\Theta\}=1$ for all $t\ge 2$,
	whence
	\[
		\lim_{T\to\infty}\frac1T\int_0^T\P\{\psi(t)\in\Theta\}\,\d t=1.
	\]
	Apply part 1(d) of Theorem \ref{th:RD} to see that
	$\mu_+(\Theta)=1$. This and \eqref{omega:UB} together 
	imply the theorem.
\end{proof}




\appendix\section{Appendix: Some technical results}

\subsection{A Gaussian integral}

\begin{lemma}\label{lem:Gauss}
	Let $X$ have a non-degenerate
	centered normal distribution. Then,
	\[
		\sup_{a\ge0}
		\E\left[ \frac{1}{\left( |X|-a\right)_+^s}\right] =
		\frac{\Gamma((1-s)/2)}{[2{\rm Var}(X)]^{s/2}\sqrt\pi}
		\qquad\text{for all $s\in(0\,,1)$}.
	\]
\end{lemma}

\begin{proof}
	Let $v^2 := {\rm Var}(X)$. Clearly,
	\begin{align*}
		\E\left[ \frac{1}{\left( |X|-a\right)_+^s}\right]  & =
			\frac{1}{v\sqrt{\pi/2}}\int_a^\infty
			\frac{\e^{-x^2/(2v^2)}}{(x-a)^s}\,\d x
			=\frac{1}{v\sqrt{\pi/2}}\int_0^\infty
			\frac{\e^{-(x+a)^2/(2v^2)}}{x^s}\,\d x\\
		&\le\frac{1}{v\sqrt{\pi/2}}\int_0^\infty
			\frac{\e^{-x^2/(2v^2)}}{x^s}\,\d x,
	\end{align*}
	with identity in place of ``$\le$'' when $a=0$. Now compute.
\end{proof}

\subsection{An elementary coupling}

Let us document the following very well known elementary fact about
couplings of random variables.

\begin{lemma}\label{lem:coupling}
	Let $X$ be a random variable such that $G(x):=\P\{X\le x\}\ge H(x)$ for all $x\in\R$,
	where $H:\R\to[0\,,1]$ is a cumulative distribution function on $\R$. Then, 
	one can construct a random
	variable $Y$ whose cumulative distribution function is $H$, and satisfies $\P\{X\le Y\}=1$.
\end{lemma}

We include a proof for the sake of completeness, also as it is so short.

\begin{proof}
	Let $H^{-1}$ denote the right-continuous inverse of $H$
	and recall that the distribution function of $Y:=H^{-1}(G(X))$ is $H$.
	The lemma follows from the facts that $G\ge H$ and $H^{-1}$ is monotone.
\end{proof}

\subsection{A random walk inequality}

We mention the following simple inequality about the expected number of large negative
excursions of a simple random walk on $\Z$ with positive upward drift.

\begin{lemma}\label{lem:E-num}
	Let $\{Z_n\}_{n=1}^\infty$ be i.i.d.\ with $p := \P\{Z_1=1\}>1/2$ and $q:=\P\{Z_1=-1\}=1-p$. 
	Then, 
	\[
		\E\left[ \sum_{n=1}^\infty \bm{1}_{\{Z_1+\cdots+Z_n\le -k\}}\right] \le 
		\frac{\sqrt{4pq}}{1- \sqrt{4pq}}\left( \frac qp\right)^{k/2}
		\qquad\text{for all $k\in\R_+$}.
	\]
\end{lemma}

One usually verifies results of this type by appealing to excursion theory.
We include a simpler proof instead.
	
\begin{proof}
	We use Markov's [Chernoff's] inequality to see
	that for all $n\in\N$, $\theta>0$, and $k\in\R_+$,
	\[
		\P\{ Z_1+\cdots + Z_n\le -k\} =
		\P\left\{ \e^{-\theta (Z_1+\cdots+Z_n)}\ge \e^{\theta k}\right\}
		\le \e^{-\theta k}
		\left[ p \e^{-\theta} + q\e^\theta \right]^n.
	\]
	Therefore, we set $\theta:=\tfrac12\log(p/q)$ to minimize the quantity in
	square brackets and see that
	\[
		\P\{ Z_1+\cdots + Z_n\le -k\} \le\left( \frac qp\right)^{k/2}
		\left[4pq\right]^{n/2}
		\quad\text{for all $n\in\N$ and $k\in\R_+$}.
	\]
	Because $4pq<1$, we may sum over $n\in\N$ to deduce the announced result.
\end{proof}

\subsection{On a class of stochastic differential inequalities}

The primary purpose of this portion of the appendix is to prove a bound
on the hitting probability of a small number by a certain non-negative It\^o process.

\begin{proposition}\label{pr:SDI}
	Suppose $X=\{X_t\}_{t\ge0}$ is a non-negative, continuous $L^2(\P)$-martingale 
	that starts from $X_0=a^2$ for a non-random number $a>0$, and solves
	the stochastic differential inequality,
	\[
		\d X_t \le X_t\,\d t  + \d M_t
		\qquad\text{for all $t>0$},
	\]
	where $\{M_t\}_{t\ge0}$ is a mean-zero, continuous 
	$L^2(\P)$-martingale. More precisely
	$\d X_t = C_t\,\d t + \d M_t$ where $C_t\le X_t$ and $\{C_t\}_{t\ge0}$
	is a.s.\ of bounded variation. Then,
	\[
		\P\left\{ \inf_{s\in(0,t)}X_s>\varepsilon^2\,, 
		\int_0^t\e^{-s}\,\frac{\d\<X\>_s}{X_s}\ge b^2\right\} 
		\le \sqrt{\frac{2}{\pi}}
		\int_{|x| < 2(a-\varepsilon \e^{-t/2})/b} \e^{-x^2/2}\,\d x,
	\]
	uniformly for all $b,t>0$ and $\varepsilon\in(0\,,a\e^{t/2})$.
\end{proposition}

The proof of Proposition \ref{pr:SDI}
hinges on a simple small-ball estimate for continuous $L^2(\P)$-martingales,
which we include next.

\begin{lemma}\label{lem:small-ball}
	Let $\{M_s\}_{s\ge0}$ be a mean-zero, continuous $L^2(\P)$-martingale. Then, 
	\begin{equation}\label{eq:small-ball}
		\P\left\{ \inf_{0<s<t} M_s  \ge -\varepsilon ~,~ \<M\>_t\ge A\right\}
		\le \sqrt{\frac{2}{\pi}}\int_{|x| < \varepsilon/\sqrt A} \e^{-x^2/2} \,\d x
		\qquad\text{for all $\varepsilon,A,t>0$},
	\end{equation}
\end{lemma}

The right-hand side of \eqref{eq:small-ball}
can be expressed more compactly in terms of $\text{\rm erf}(\varepsilon\sqrt{2/A})$, but we prefer the
above ``probabilistic'' description.

\begin{proof}
	By the Dubins, Dambis-Schwarz Brownian representation  
	of continuous martingales \citep[see][Theorem 1.6]{Revuz-Yor}
	there exists a linear Brownian motion $\{B(s)\}_{s\ge0}$
	such that $B(0)=0$ and $M_s = B(\< M\>_s)$ for all $s\ge0$. Therefore,
	\[
		\P\left\{ \inf_{0<s<t}M_s \ge -\varepsilon~,~\<M\>_t\ge A\right\}
		= \P\left\{ \inf_{0<s<\< M\>_t}B(s) \ge -\varepsilon~,~
		\<M\>_t\ge A\right\}
		\le \P\left\{ \inf_{0<s< A} B(s) \ge -\varepsilon\right\},
	\]
	which yields the result, thanks to Brownian scaling and the reflection principle.
\end{proof}

\begin{proof}[Proof of Proposition \ref{pr:SDI}]
	Apply It\^o's formula to see that
	$Y_t := X_t^{1/2}$ solves $Y_0=a$ subject to
	\begin{equation}\label{SDI}
		\d Y_t \le \frac12 Y_t \,\d t + \d N_t,
	\end{equation}
	where $\{N_t\}_{t\ge0}$ is a mean-zero continuous $L^2(\P)$-martingale
	with quadratic variation,
	\[
		\< N\>_t = \frac14\int_0^t\frac{\d\<X\>_s}{X_s}
	\]
	We remove the drift in \eqref{SDI}
	in a standard way by setting $Z_t := \e^{-t/2} Y_t$ for all $t\ge0$.
	Then,  $Z_0=a$ and $Z$ satisfies the stochastic differential inequality,
	$\d Z_t = \e^{-t/2}\,\d Y_t - \frac12 \e^{-t/2}Y_t\,\d t
	\le \e^{-t/2}\,\d N_t.$
	In other words,
	\[
		Z_t \le a + \int_0^t \e^{-s/2}\,\d N_s := a + \mathcal{M}_t
		\qquad\text{for all $t>0$},
	\]
	where $\{\mathcal{M}_t\}_{t\ge0}$ is a mean-zero continuous
	$L^2(\P)$ martingale whose quadratic variation is given by
	\[
		\d\<\mathcal{M}\>_t = \e^{-t}\,\d\<N\>_t = \frac{\e^{-t}}{4}\frac{\d\< X\>_t}{X_t}.
	\]
	Since
	\[
		\left\{ \inf_{s\in(0,t)}X_s>\varepsilon^2
		\,,\, \int_0^t \e^{-s}\,\frac{\d\<X\>_s}{X_s} \ge b^2\right\}
		\subset \left\{ \inf_{s\in(0,t)}\mathcal{M}_s>\varepsilon \e^{-t/2} - a
		\,,\, \<\mathcal{M}\>_t \ge \frac{b^2}{4}\right\},
	\]
	Proposition \ref{pr:SDI} follows from the above and Lemma
	\ref{lem:small-ball}.
\end{proof}

\subsection{On martingale measures}

Consider a continuous-in-time $L^2(\P)$-martingale 
measure $M:=\{M_t(A)\}_{t\ge0,A\in\mathcal{B}(\T)}$,
in the sense of \cite{wal86}, where
$\mathcal{B}(\T)$ denotes the collection of Borel subsets of $\T$.
The following is a ready consequence of It\^o calculus.

\begin{lemma}\label{lem:mart:meas}
	If $\< M(A) \,, M(B)\>_t = t|A\cap B|$
	for all $t\ge0$ and $A,B\in\mathcal{B}(\T)$, then
	we can realize $M$ as
	\[
		M_t(A) = \int_{(0,t)\times A} w(\d s\,\d y)\qquad
		\text{for all $t\ge0$ and $A\in\mathcal{B}(\T)$},
	\]
	where $\dot{w}=\{\dot{w}(t\,,x)\}_{t\ge0,x\in\T}$ is a
	space-time white noise.
\end{lemma}

\begin{proof}
	We prove that $M:=\{M_t(A)\}_{t\ge0,A\in\mathcal{B}(\T)}$
	is a Gaussian process; the remainder of the lemma
	follows from this and a simple covariance computation which we skip.
	
	If $A_1$ and $A_2$ are two Borel subsets of $\T$ such that the Haar 
	measures of $A_1$ and $A_2$ are 1 and $A_1\cap A_2$ has zero
	Haar measure, then $\< M(A_i) \,, M(A_j)\>_t = \delta_{i,j}t$ for all $t\geq 0$.
	Thus, L\'evy's characterization theorem \citep[see][Theorem 3.6]{Revuz-Yor}
	asserts that $\{M_t(A_1), M_t(A_2)\}$ is a 2-dimensional Brownian motion, 
	whence $M$ is a Gaussian process by induction.This completes the proof.
\end{proof}

Lemma \ref{lem:mart:meas} immediately implies the following result, which will play
a key role in our coupling construction in Section \ref{sec:coupling}.

\begin{corollary}\label{cor:WN}
	Let $\dot{V}_1$ and $\dot{V}_2$ denote two independent space-time
	white noises, and suppose $\{\phi_1(t\,,x)\}_{t\ge0,x\in\T}$ and
	$\{\phi_2(t\,,x)\}_{t\ge0,x\in\T}$ are continuous
	predictable random fields, in the 
	filtration defined by $\dot{V}_1$ and $\dot{V}_2$, that takes values in $[0\,,1]$
	and satisfy
	\[
		\phi_1^2(t\,,x) + \phi_2^2(t\,,x)=1\quad\text{for all $t>0$ and $x\in\T$ a.s.}
	\]
	Define $\dot{M}(t\,,x) = \phi_1(t\,,x)\dot{V}_1(t\,,x) + \phi_2(t\,,x)\dot{V}_2(t\,,x)$;
	more formally,
	\[
		M_t(A) := \int_{(0,t)\times A}\phi_1(s\,,y)\, V_1(\d s\,\d y) + \int_{(0,t)\times A}
		\phi_2(s\,,y)\, V_2(\d s\,\d y),
	\]
	for $t>0$ and $A\in\mathcal{B}(\T)$. Then, $\dot{M}$ is space-time
	white noise; more formally, there exists a space-time white
	noise $\dot{w}$ such that
	$M_t(A)=\int_{(0,t)\times A}w(\d s\,\d y)$ a.s.\ for every $t>0$ and $A\in\mathcal{B}(\T)$.
\end{corollary}

\begin{proof}
	Since $\phi_1$ and $\phi_2$ are bounded, the stochastic integrals are Walsh integrals
	that define a martingale measure. It remains to verify that
	$\< M(A)\,, M(B)\>_t = t|A\cap B|$
	for all $t>0$ and $A,B\in\mathcal{B}(\T)$; see Lemma \ref{lem:mart:meas}. But this
	mutual variation formula is an immediate consequence
	of the construction of Walsh integrals.
\end{proof}

\section{Appendix: Sketch of proof of Lemma \ref{lem:large-dev}}

There have been many results similar to Lemma \ref{lem:large-dev}, 
as we mention below. We will give an outline of the argument.

By \eqref{eq:sup-sigma} and the definition of $\sigma_n$ 
in \eqref{eq:def-sigma-n}, we have that 
\begin{equation*}
	|\sigma_n(t\,,x)|\le 4\text{Lip}_\sigma L_{\tau_n}(v_n).
\end{equation*}
Consulting definition \eqref{eq:def-I-tr}, we see that 
$\mathcal{I}_n^{\text{tr}}(t\,,x)$ is a white noise integral of 
$\lambda p_{t-s}(x-y)\sigma_n(s\,,y)$ against a time-shifted 
space-time white noise.  

Now Lemma \eqref{lem:large-dev} would follow from \citet[Theorem 4.2, page 126]{spde-utah},
except in that reference $p_t(x\,,y)$ is
replaced by the heat kernel on $\R$, not $\T$.  However, we can easily modify the proof in
\cite{spde-utah} to cover our case.  Using the expansion 
\eqref{eq:heat-kernel-expansion}, we find that Lemma 4.3 on  page
126 of \cite{spde-utah} still holds, except that $|x-y|$ is replaced
by the distance from $x$ to $y$ on $\T$.  The rest of the argument goes through as
before, giving us Lemma \ref{lem:large-dev}.\qed


\end{document}